\documentclass[a4paper,11pt]{amsart}
\newcommand{\query}[1]%
{\mbox{}\marginpar{\raggedright\hspace{0pt}{\small\em #1}}}%
\usepackage{amssymb}
\usepackage{amsmath}
\usepackage{amsfonts}
\usepackage{mathrsfs}
\usepackage[all]{xy}
\usepackage{multirow}
\usepackage{multicol}
\usepackage{graphicx}
\usepackage{caption}
\usepackage{float}
\usepackage{color}

\usepackage{hyperref}

\newcommand{\xdownarrow}[1]{%
  {\left\downarrow\vbox to #1{}\right.\kern-\nulldelimiterspace}
}

\theoremstyle{plain}
\newtheorem{theorem}{Theorem}[section]

\newtheorem{lemma}[theorem]{Lemma}
\newtheorem{proposition}[theorem]{Proposition}

\newtheorem{corollary}[theorem]{Corollary}
\newtheorem*{theorem*}{Theorem}

\theoremstyle{definition}
\newtheorem{definition}[theorem]{Definition}

\newtheorem{example}[theorem]{Example}
\newtheorem{notation}[theorem]{Notation}

\theoremstyle{remark}

\newtheorem{remark}[theorem]{Remark}

\def\ZZ{{\mathbb Z}}
\def\NN{{\mathbb N}}
\def\QQ{{\mathbb Q}}
\def\RR{{\mathbb R}}
\def\CC{{\mathbb C}}

\def\coker{{\rm CoKer}}
\def\ker{{\rm Ker}}
\def\im{{\rm Im}}
\def\cyl{{\rm cyl}}

\def\HH{{\mathbb H}}
\def\H{{\mathcal H}}

\begin{document}

\author[M. Agust\'in, J. Fern\'andez de Bobadilla]{M. Agust\'in, J. Fern\'andez de Bobadilla}

\title[Intersection Space Constructible Complexes]{Intersection Space Constructible Complexes}

\address{(1) IKERBASQUE, Basque Foundation for Science, Maria Diaz de Haro 3, 48013, Bilbao, Spain
(2) BCAM, Basque Center for Applied Mathematics, Mazarredo 14, E48009 Bilbao, Spain}

\email{martaav22@gmail.com}

\email{jbobadilla@bcamath.org}

\thanks{First author is supported by ERCEA 615655 NMST Consolidator Grant, MINECO by the project reference
MTM2013-45710-C2-2-P, by the Basque Government through the BERC 2014-2017 program,
by Spanish Ministry of Economy and Competitiveness MINECO: BCAM Severo Ochoa excellence accreditation SEV-2013-0323. Second author is supported by ERCEA 615655 NMST Consolidator Grant, MINECO by the project reference
MTM2013-45710-C2-2-P, by the Basque Government through the BERC 2014-2017 program,
by Spanish Ministry of Economy and Competitiveness MINECO: BCAM Severo Ochoa excellence accreditation SEV-2013-0323 and
by Bolsa Pesquisador Visitante Especial (PVE) - Ciências sem Fronteiras/CNPq Project number:  401947/2013-0.}


\begin{abstract}
We present an obstruction theoretic inductive construction of intersection space pairs, which generalizes Banagl's construction of intersection spaces for arbitraty
depth stratifications. We construct intersection space pairs for pseudomanifolds with compatible trivial structures at the link fibrations; this includes the case of
toric varieties. We define intersection space complexes in an axiomatic way, similar to Goresky-McPherson axioms for intersection cohomology. We prove that
if the intersection space exists, then the pseudomanifold has an intersection space complex whose hypercohomology recovers the cohomology of the intersection space pair.
We characterize existence and uniqueness of intersection space complexes in terms of the derived category of constructible complexes.
We show that intersection space complexes of algebraic varieties lift to the derived category of Mixed Hodge Modules, endowing intersection space 
cohomology with a Mixed Hodge Structure. We find classes of examples
admitting intersection space complex, and counterexamples not admitting them; they are in particular the first examples known not admitting Banagl intersection
spaces. We prove that the (shifted) Verdier dual of an intersection space complex is an intersection space complex. We prove a generic Poincare duality theorem
for intersection space complexes.
\end{abstract}

\maketitle

\section{Introduction}

Intersection spaces have been recently introduced by Banagl as a Poincare duality homology theory for topological pseudomanifolds
which is an alternative to Goreski and McPherson intersection homology. When they are available they present the advantages of being spatial modifications
of the given topological pseudomanifold, to which one can later apply algebraic topology functors in order to obtain invariants. In this sense, if one
applies (reduced) singular cohomology one obtains a homology theory with internal cup products and a Poincar\'e duality is satisfied between the homology
theories corresponding to complementary perversities. Moreover, one can apply many other functors leading
to richer invariants. The idea of intersection spaces was sketched for the first time in~\cite{Ban1}, and was fully developed for spaces with
isolated singularities in~\cite{Ban10}.

In~\cite{Ban10} Banagl carefully analyzed the case of quintic $3$-folds with ordinary double points appearing in the conifold transition and noticed that,
in the same way that intersection cohomology gives the cohomology of a small resolution, cohomology of intersection spaces gives the cohomology of a
smoothing in this case. This fitted with predictions motivated by string theory (see Banagl papers for full explanations).

This motivated further work
by Banagl, Maxim and Budur (\cite{BM1}, \cite{BM2}, \cite{BBM}, \cite{Max}) in which the relation between the cohomology of intersection spaces for the middle
perversity and the
Milnor fibre of a hypersurface $X$ with isolated singularities is analyzed. The latest evolution of the results of these papers, contained in~\cite{BBM},
is the construction of a perverse sheaf in $X$ whose hypercohomology computes the cohomology of the intersection space of $X$ in all degrees except for
the top degree. Such a perverse sheaf is a modification of the nearby cycle complex and, in fact, when the monodromy is semi-simple in the eigenvalue $1$,
the middle perversity intersection space perverse sheaf is a direct summand of the nearby cycle complex.

The results described up to now are valid for only isolated singularities (sometimes even assuming that they are hypersurface singularities).
In~\cite{Ban10} Banagl generalizes the construction for the case of topological pseudomanifolds with two strata and trivial link fibration
and sketches a method for more general class of non-isolated singularities. In~\cite{BaCh} intersection spaces are constructed for the case of two strata
assuming non-trivial conditions on the fibration by links. The only case in which intersection spaces are constructed for a topological pseudomanifold
with more than two strata is in \cite{Ban2}. There, the depth $1$ strata are circles or intervals, and the depth $2$ strata are isolated singularities;
in this case, it is the topology of the strata which is very restrictive.

In~\cite{BM1} the following open questions are proposed: Is there a sheaf theoretic approach to intersection spaces, similar to the one
of Goreski and McPherson \cite{GorMP}
for intersection cohomology?. Up to which kind of singularities the intersection space constructions can be extended? Is intersection space cohomology
of algebraic varieties endowed with a Hodge structure?

The papers \cite{BM2}, \cite{BBM} are contributions to the first and third question for the case of isolated singularities. The paper~\cite{BaHu} is also
a contribution towards the third question.

In a recent paper~\cite{Ge}, Genske takes a new viewpoint: instead of only giving up the topological construction and focusing in producing a complex of
sheaves at the original space, he constructs a complex of vector spaces which is related with the complex computing the (co)homology of the original
space $X$, but that satisfies Poinca\'re duality. The construction is valid for any analytic variety (Poincar\'e duality is satisfied in the compact case).
His construction is a bit further to original Banagl ideas than ours, since his procedure is to make a modification which is global in a neighbourhood of
the singular set, instead of stratifying it conveniently and making a fibrewise construction.

Finally, in order to finish our review of existing results, let us mention the rational Poincar\'e spaces approach developed in~\cite{Kl}.

The present paper is a contribution to the three
questions formulated above for the general singularity case. Before explaining our results in detail in the next Section, let us enumerate them in a very
condensed way:

We realize that, for constructing Banagl's intersection spaces in a
general pseudomanifold $X$, one needs to adopt the viewpoint of pairs of spaces and associate an {\em intersection space pair}, which is a
spatial modification of the pair $(X, Sing(X))$. We find a procedure which runs inductively on the codimension of the strata andh, if it is not obstructed,
produces the intersection space pair. We show that, if the link fibrations of the pseudomanifold are trivial and the trivializations of these fibrations verify some compatibility
conditions, the intersection space pair exists. This includes the case of toric varieties.

We prove that, if an intersection
space pair exists for a topological pseudomanifold and a given perversity, then there exist a constructible complex of sheaves in our original space $X$
that satisfies a set of properties of the
same kind that those that characterize intersection cohomology complexes in \cite{GorMP}; we call this complex an {\em intersection space complex} for the
given perversity. Its hypercohomology recovers the reduced cohomology of the intersection space in the case of isolated singularities. In the case
of depth $1$ topological pseudomanifold, it recovers the cohomology of the intersection space relative to the singular stratum (like in~\cite{Ban10}), which
is the one that satisfies Poincar\'e duality for complementary perversities. For depth $2$ and higher, if the dimension of the strata is sufficiently
high, the intersection space construction is intrinsically a construction of pairs of spaces, as we will see below; the hypercohomology of our
intersection space complex computes the rational cohomology of the pair of spaces.

Next, we leave the realm of topology and shift to a sheaf theoretic viewpoint studying under which conditions intersection space complexes exist.
We find obstructions
for existence and uniqueness of intersection space complexes and give spaces parametrizing the possible intersection space complexes in case that the
obstruction for existence vanish. Both of these obstructions vanish in the case of isolated singularities and the obstruction for existence also vanish
in the case considered in \cite{Ban2}, as one should expect. If one assumes that the topological pseudomanifold is an algebraic variety, we show
how to carry our constructions in the category of mixed Hodge modules, yielding a polarizable mixed Hodge structure in the hypercohomology of the
intersection space complex, and hence in the cohomology of intersection spaces when they exist.

We turn to analyze classes of topological pseudomanifolds in which we can prove the existence of intersection space complexes. We show that they exist for any
perversity when the successive link fibrations are trivial (without the compatibility conditions needed to construct the intersection space pair).
We also prove the existence if the
homological dimension of the strata with respect to local systems is at most $1$. This includes the case treated in~\cite{Ban2}. On the other hand, building
on the obstructions for existence, we produce the first examples of topological pseudomanifolds such that intersection space complexes do not
exist for given perversities. As a consequence, Banagl intersection spaces can not exist either. One of the examples is a normal algebraic variety
whose stratification has depth $1$ and whose transversal singularity is an ordinary double point of dimension $3$ (those appearing in the conifold
transition examples); the perversity used is the middle one.

Finally, we turn to duality questions. We show that the Verdier dual of an intersection space complex with a given perversity is an intersection space complex
with the complementary perversity. The proof resembles the one given in~\cite{GorMP} for intersection cohomology complexes. However, since (unlike
intersection cohomology complexes) intersection space complexes are not unique, this does not yield self dual sheaves for the middle perversity
on algebraic varieties. In the case of depth $1$ stratifications, we prove that generic choices of the intersection space complex yield the same
Betti numbers in hypercohomology and we obtain Poincar\'e duality at the level of generic Betti numbers for complementary perversities.

Some open questions and further directions are hinted at the end of next section.

\section{Main results}

This section is the guide to the paper. Here we describe in detail the content of the paper section by section, with cross-references to the main results.
The reader may jump to the corresponding sections for a full exposition.

\subsection{Topological constructions}

The paper starts with a topological construction of pairs of intersection spaces in Section~\ref{sec:top}.

Banagl construction of the intersection space~\cite{Ban10} for a $d$-dimensional topological pseudomanifold $X$ with isolated singularities for a given perversity
$\bar p$ runs as follows. Let $\Sigma=\{p_1,...,p_r\}$ be the singular set. Around $p_i$ consider a conical neighbourhood $B_i$. Let $L_i$ denote the link $\partial B_i$.
Let $\bar q$ be the complementary perversity of $\bar p$. Consider a homological truncation
$$(L_i)_{\leq \bar q(d)}\to L_i.$$
This is a mapping of spaces inducing isomorphisms in homology in degrees up to $\bar q(d)$ and such that $H_i((L_i)_{\leq \bar q(m)})$ vanish for
$i>\bar q(d)$. Assume for simplicity that the truncation map is an inclusion. Construct the intersection space replacing each of the $B_i$'s by the
cone over $(L_i)_{\leq \bar q(d)}$, call the resulting space $Z$. The vertices of the cone are called $\Sigma=\{p_1,...,p_r\}$ as well.
The intersection space is the result of attaching the cone over $\Sigma$ to $Z$. If the truncation map is not an inclusion, one may force this
using an appropriate homotopy model for it. The intersection space homology is the reduced homology of the intersection space.

The construction for the case of topological pseudomanifolds $X=X_d\supset X_{d-m}$ of dimension $d$ with a single singular stratum of codimension $m$ contained
in~\cite{Ban10} and~\cite{BaCh} is the following generalization. Let $T$ be a tubular neighbourhood of the singular set $X_{d-m}$. Consider the locally trivial
fibration $T\to X_{d-m}$, and let $\partial T\to X_{d-m}$ be the associated fibration of links. Consider a fibrewise homology truncation
$$\partial T_{\leq \bar q(m)}\to \partial T.$$
This is a morphism of locally trivial fibrations which is a $\bar q(m)$-homology truncation at each fibre. Remove
$T$ from $X$ and replace it by the fibrewise cone over $\partial T_{\leq \bar q(m)}$ (see Definition~\ref{def:fibrecone}); call the resulting space $Z$.
As before $\Sigma$ is a subspace of $Z$. The intersection space is the result of attaching the cone over $\Sigma$ to $Z$ and the intersection space homology
is the reduced homology of the intersection space.

Notice that the intersection space homology coincides with the relative homology $H_*(Z,\Sigma)$. This observation is the starting point of our
construction for more than two strata and of our constructible complex approach to intersection space homology. Consider a topological pseudomanifold
with stratification
$$X=X_d\supset X_{d-2} \supset ... \supset X_0 \supset X_{-1}=\emptyset$$
and a suitable system of tubular neighbourhoods of the strata (the conical structure of Definition \ref{def:conical}). A great
variety of topological pseudomanifolds can be endowed with this structure (see Remark \ref{re:fixconical}).

Our topological construction of intersection spaces consist in modifying
the space $X$ inductively, going each step deeper in the codimension of the strata, by taking successive fibrewise homology truncations.
In doing so, one necessarily
modifies the singular set $X_{d-2}$ and, as a result, the singular set is not going to be contained in the modified space $Z$. Instead, one obtains
a modification $Y$ of $X_{d-2}$ contained in $Z$. One obtains a pair $(Z,Y)$, and it is the homology of this pair our definition of intersection space
homology. An important feature of the construction is that one needs to adopt the viewpoint of pairs of spaces right from the beginning
if one wants to have a chance of proving duality results; this incarnates in the need of taking homology truncations of pairs of spaces. This new
feature appears from the depth $2$ strata and hence it did not appear in Banagl constructions explained above. It also may happen that, if the singular
stratum, is of small dimension in comparison with the perversity, the situation does not appear at all.

As in Banagl construction, the homology truncations need not be inclusions. This forces us to work with an adequate homotopy model for $X$.

It is important to record for future reference that, at the $k$-th inductive step of the
construction, one obtains a pair of spaces $(I_k^{\bar p}X,I_k^{\bar p}(X_{d-2}))$ which contain $X_{d-k-1}$ and verify
\begin{enumerate}
 \item the pair $(I_k^{\bar p}X\setminus X_{d-k-1},I_k^{\bar p}(X_{d-2})\setminus X_{d-k-1})$ is an intersection space pair of $X\setminus X_{d-k-1}$.
 \item there is a system of tubular neighbourhoods $T_{k-1}$ of $X_{d-k-1}\setminus X_{d-k-2}$ in $I_k^{\bar p}X\setminus X_{d-k-2}$ such that we have locally
 trivial fibrations of pairs
 $$T_{k-1}\cap (I_k^{\bar p}X\setminus X_{d-k-2},I_k^{\bar p}(X_{d-2})\setminus X_{d-k-2})\to X_{d-k-1}\setminus X_{d-k-2},$$
 $$\partial T_{k-1} \cap (I_k^{\bar p}X\setminus X_{d-k-2},I_k^{\bar p}(X_{d-2})\setminus X_{d-k-2})\to X_{d-k-1}\setminus X_{d-k-2},$$
 being the first the fibrewise cone over the second (see Definition~\ref{def:fibrecone}).
\end{enumerate}

The second locally trivial fibration is called {\em the fibration of link pairs at the $k$-th step of the construction}.

The construction follows the scheme of obstruction theory: it is inductive and, at each step, choices are made. The next step may be obstructed and this
may depend on the previous choices. The obstruction consists in the impossibility of constructing a fibrewise homology truncation of the fibration of
link pairs at the $k$-th step of the construction.

When there is a set of choices so that the process terminates, we say that an intersection space pair exists. It needs
not be unique.

In the case where the conical structure is trivial, that is, if the link fibrations are trivial and the trivializations are compatible with each other
(see Definition \ref{def:strongtriv}), the intersection space pair exists
(see Theorem \ref{th:existTr}). This includes the case of arbitrary toric varieties.

\subsection{From topology to constructible complexes}

In the rest of the paper, we investigate the existence and uniqueness of intersection space pairs and their duality properties by sheaf theoretic methods.
For this, we associate to each intersection space pair an element in the derived category of constructible complexes whose
hypercohomology computes the rational cohomology of the intersection space pair.

To get this, we need to construct a sequence of intersection space pairs which modify the pair $(X_d,X_{d-2})$ in increasingly smaller
neighbourhoods of the strata of $X$. This is done in Section~\ref{sec:sequence}.

In Section~\ref{sec:sheafification}, we exploit the sequence of intersection spaces to derive a constructible complex $IS$ (see Definition~\ref{def:IS})
and prove, in Theorem~\ref{th:COHOMOLOGY}, that the hypercohomology of $IS$ recovers the cohomology of the intersection space pair. Finally, in
Theorem~\ref{th:existints}, we prove that $IS$ satisfies a set of properties in the same spirit that those that characterize intersection cohomology
complexes in~\cite{GorMP}. This is the basis for the axiomatic treatment of the next section.

\subsection{A derived category approach to intersection space (co)homology}

In Section~\ref{sec:axiomatic}, we take an axiomatic approach to intersection space complexes in the same way as Goresky-McPherson approach to
intersection cohomology in \cite[section 3.3]{GorMP}. We define two sets of properties in the derived category of cohomologically constructible
sheaves on $X$. The first set are the properties of the intersection cohomology sheaf composed with a shift.
The second set of properties are inspired by Theorem~\ref{th:existints}.

We will call a complex of sheaves verifying the second set of properties
intersection space complex of $X$ (Definition \ref{def:intspacecomplex}). Theorem~\ref{th:existints} implies that if there exist an intersection
space pair of $X$ (see Definition \ref{def:intersection space}), then there exist an intersection space complex of $X$ whose hypercohomology coincides with the cohomology of the intersection space pair. Moreover, we compare the support and cosupport
properties of intersection cohomology complexes and intersection space complexes. From the comparison, one sees that intersection space complexes,
except possibly in the case of isolated singularities, can not be perverse sheaves.

At this point, we investigate in a purely sheaf theoretical way in which conditions an intersection space complex may exist.
For $k=2,..., d$, we define $U_k:=X \setminus X_{d-k}$ and we denote the canonical inclusions by $i_k:U_k \rightarrow U_{k+1}$ and
$j_k: X_{d-k} \setminus X_{d-(k+1)} \rightarrow U_{k+1}$.

In Theorem~\ref{th:existunic},
we give necessary and sufficient conditions for the existence of intersection space complexes. As in the topological setting, the construction proceeds
inductively on each time deeper strata. At the $(k+1)$-th step, we have constructed an intersection space complex $IS_{k}$ on $X\setminus X_{d-k-1}$
such that the complex $j_{k+1}^*i_{k+1*}IS_{k}$ on $X_{d-k-1}\setminus X_{d-k-2}$ is cohomologically cohomologically locally constant.
Comparing with the topological construction these local systems are the
cohomology local systems of the fibration of link pairs at the $(k+1)$-th step of the construction. One can consider the natural triangle in the derived
category:
$$\tau_{\leq \bar q(k+1)} j_{k+1}^*i_{k+1*}IS_{k}\to j_{k+1}^*i_{k+1*}IS_{k}\to \tau_{> q(k-1)} j_{k+1}^*i_{k+1*}IS_{k}\xrightarrow{[1]}.$$
The obstruction to perform the next step in the construction is the obstruction to split the triangle in the derived category. This is the constructible
sheaf counterpart of the possibility to construct fibrewise homology truncations in the topological world. In Theorem~\ref{th:existunic}, we study the
parameter spaces classifying the possible intersection space complexes in step $k+1$ having fixed the step $k$ (they are not unique in general).
In the same theorem, for the sake of comparison we provide a proof of the existence and uniqueness of intersection cohomology complexes using the same
kind of techniques.

There are extension groups controlling the obstructions for existence and uniqueness at each step. They are recorded in Corollary~\ref{cor:obstructions}.

If $X$ is an algebraic variety, we can lift our construction of intersection space complexes to the category of mixed Hodge modules over $X$. This is
Theorem~\ref{th:mhm} and, as a corollary, one obtains a mixed Hodge structure in the cohomology of intersection spaces. The
obstructions for existence and uniqueness are the same kind as extension groups, but taken in the category of
mixed Hodge modules (see Corollary~\ref{cor:obstructionsmhm}). Using the same techniques, we show that,
for an arbitrary perversity, the intersection cohomology complexes are mixed Hodge modules (Theorem~\ref{th:icmhm}). This puts a mixed Hodge structure on
intersection cohomology with arbitrary perversity. We believe that this should be well known, but we provide a proof since it is a simple consequence of our ideas.

\subsection{Classes of spaces admitting intersection space complexes and counterexamples}

From the previous section, it is clear that the spaces admitting intersection space complexes need to be special. In this Section~\ref{sec:examples}, we find
two sufficient conditions for this, yielding an ample class of (yet special) examples.

By the results explained previously it is clear that pseudomanifolds admitting a trivial conical structure (see Definition \ref{def:strongtriv}),
the intersection space complex exists by the combination of Theorems \ref{th:existTr} and \ref{th:existints}.

The point is that if one only needs the existence of intersection space complexes one can relax the triviality properties. In Theorem~\ref{th:exintspacecomplex}
we prove that if the link fibrations are trivial in the sense of Definition~\ref{def:Tr}, then the intersection space complex exists. This 
leaves out the hipothesis on the compatibility of trivializations.

The next class of examples admitting intersection space complexes are pseudomanifolds whose strata are of ``cohomological dimension at most $1$ for local systems''
in the sense of Definition~\ref{def:homdim}. This is proved in Theorem~\ref{th:homdim}. This includes the case
studied by Banagl in~\cite{Ban2} and the case of complex analytic varieties with critical set of dimension $1$ which are sufficiently singular, in the sense
that there are no positive dimensional compact strata (Corollary~\ref{cor:smalldim}).

From our previous results, it is clear that a necessary condition for the existence of an intersection space is the existence of an intersection space complex.
We find a few examples (see Example~\ref{ex:hopf}, Example~\ref{exhopf2} and Example~\ref{ex:algebraic}) not admitting intersection space complex.
The last example is an algebraic variety with two strata and the perversity is the middle one. So, one should not expect that algebraicity helps
in the existence of intersection spaces. The idea to produce the examples is to observe that, if an space admits intersection space complexes, certain
differentials in the Leray spectral sequence of the fibrations of links (which in our sheaf theoretic treatment is a local to global 
spectral sequence) have to vanish (Proposition~\ref{prop:lerayobs} and Corollary~\ref{cor:cex}).

\subsection{Duality}

In Section~\ref{sec:dual}, we prove that, if $\bar p$ and $\bar q$ are complementary perversities and $IS_{\bar p}$ is an intersection space complex
for perversity $\bar p$, then its Verdier dual is an intersection space complex for perversity $\bar q$ (Theorem~\ref{th:verdual}). So, the Verdier
duality functor exchanges the sets of intersection space complexes for complementary perversities. The proof follows the axiomatic treatment of~\cite{GorMP}
for intersection cohomology complexes.

A surprising consequence is that the existence of intersection space complexes is equivalent for complementary perversities
(Corollary~\ref{cor:dualityexchange}).

Then, we move to the case of depth $1$ stratifications and prove, in Proposition~\ref{prop:genericbetti},
that, for generic choices of the intersection space complexes, the Betti numbers are always
the same (they are minimal). Then, in Theorem~\ref{th:genericduality}, we show that the Betti numbers symmetry predicted by Poincar\'e duality
for complementary perversities is satisfied for generic intersection space Betti numbers.

\subsection{Open questions}

Here is a list of natural questions for further study:
\begin{enumerate}
 \item We conjecture that the intersection space complexes associated via Definition~\ref{def:IS} to the intersection space pairs constructed for pseudomanifolds
 with trivial conical structure in Theorem~\ref{th:existTr} are self Verdier dual when the strata are of even codimension and the perversity is the middle one.
 \item Assume that the intersection space complex exist.  Does there exist an associated rational homotopy intersection space? Are there further
 restrictions than the existence of the intersection space complex?
 \item If the intersection space complex exist, can one define on its hypercohomology a natural internal cup product? Can one find a product
 turning its space of sections into a differential graded algebra inducing a cup product? This would lead to a definition of ``intersection space rational homotopy
 type''.
 \item Toric varieties have intersection space pairs. Compute their Betti numbers in terms of the combinatorics of the fan.
 \item Generalize the generic Poincare duality Theorem~\ref{th:genericduality} to the case of arbitrary depth stratifications.
 \item This is a suggestion of Banagl: relate our sheaf theoretical methods with the characteristic class obstructing Poincare duality discussed in~\cite{BaCh}.
 If the characteristic class vanishes, is the intersection space complex self-dual in the case of pseudomanifolds
 with trivial conical structure, even codimensional strata and middle perversity?
\end{enumerate}

\section{A topological construction of intersection spaces}

\label{sec:top}

\subsection{Topological preliminaries}

First, we give some basic definitions about $t$-uples of spaces in order to fix notation.

\begin{definition}
\begin{enumerate}
 \item A $t$-uple of spaces is a set of topological spaces $(Z_1,...,Z_t)$.
 \item A morphism from a $t$-uple of spaces into a space $(Z_1,...,Z_t)\rightarrow Z$
 is a set of morphisms $\varphi_i:Z_1 \rightarrow Z$.
 \item A morphism between $t$-uples of spaces $(Z_1,...,Z_t)\rightarrow (Z'_1,...,Z'_t)$ is a set of morphisms
 $\varphi_i:Z_i \rightarrow Z'_i$.
 \item  The mapping cylinder of a morphism $\varphi=(\varphi_1,...,\varphi_t):(Z_1,...,Z_t)\rightarrow Z$, $cyl(\varphi)$, is the union of the $t$-uple
 $(Z_1,...,Z_t) \times [0,1]$ with $(Z, \im(\varphi_2),...,\im (\varphi_t))$ with the equivalence relation $\sim$ such that for $i=1,...,t$ and for every
 $x\in Z_i$, we have $(x,1)\sim \varphi_i(x)$.
\end{enumerate}
\end{definition}

\begin{remark}
Remember that the mapping cylinder of a morphism of spaces $f:X\rightarrow Y$ is $cyl(f):=(X\times[0,1] \sqcup Y) /\sim$ where $\sim$ is the equivalence relation such that for every $x\in X$, $(x,1)\sim f(x)$.
\end{remark}

\begin{definition}
\label{def:fibrecone}
Let $\sigma:(Z_1,...,Z_t)\to B$ be a locally trivial fibration of $t$-uples of spaces. The \emph{cone} of $\sigma$ over the base
$B$ is the locally trivial fibration
$$\pi:cyl(\sigma)\to B,$$
where $cyl(\sigma)$ is the mapping cylinder of $\sigma$, $\pi(x,t):=\sigma(x)$ for $(x,t)\in (Z_1,...,Z_t) \times[0,1]$ and $\pi(b):=b$
for $b\in B$(the definition of $\pi$ is compatible with the identifications made to construct $cyl(\sigma)$). The cone over a fibration
has a canonical {\it vertex section}
$$s:B\to cyl(\sigma)$$
sending any $b\in B$ to the vertex of the cone $(cyl(\sigma))_b$.
\end{definition}

The following figure shows schematicatically $cyl (\sigma)$ when the fiber of $\sigma$ is the pair $(\mathbb T, \Sigma)$ where $\mathbb T$ is a torus
and $\Sigma$ is isomorphic to $S^1$, and the base $B$ is a circle. The fibre $\Sigma$ is depicted into the torus $\mathbb T$ as a discontinuous circle.

\begin{figure}[H]
  \centering
 \includegraphics[width=12cm]{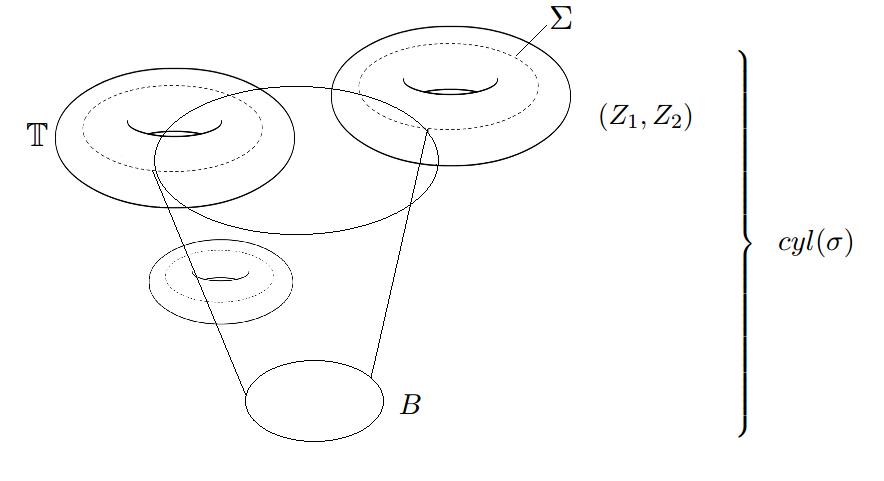}
  \caption{}\label{fig:fibration_cone}
\end{figure}

We use a definition of topological pseudomanifold similar to \cite[Definition 4.1.1]{Ban07}.

\begin{definition}\label{def:strspace}
A \emph{topological pseudomanifold} is a paracompact Hausdorff topological space with a filtration by closed subspaces

$$X=X_d\supset X_{d-2} \supset ... \supset X_0 \supset X_{-1}=\emptyset. $$

such that

\begin{itemize}
\item Each pair $(X_i,X_{i-1})$ is a locally finite relative $CW$-complex.

\item Every non-empty $X_{d-k} \setminus X_{d-k-1}$ is a topological manifold of dimension $d-k$ called pure stratum of $X$.

\item $X\setminus X_{d-2}$ is dense in $X$.

\item \textbf{Local normal triviality.} For each point $x \in X_{d-k} \setminus X_{d-k-1}$, there exists an open neighborhood $U$ of $x$ in $X$,
a compact topological pseudomanifold $L$ of dimension $k-1$ with stratification

    $$ L = L_{k-1} \supset L_{k-3} \supset...\supset L_0 \supset L_{-1}=\emptyset$$

    and a homeomorphism

          $$\varphi:U \xrightarrow{\cong} \RR^{d-k} \times c^\circ(L),$$

    where $c^\circ(L)$ is the open cone of $L$, such that it preserve the strata, that is, $\varphi(U\cap X_{d-r})=\RR^{d-k} \times × c^\circ(L_{k-r-1})$.

    $L$ is called the link of $X$ over the point $x$.
\end{itemize}
\end{definition}

The following new notion is important in our constructions.

\begin{definition}
\label{def:conical}
Let $(X,Y)$ be a pair of topological spaces and let
$$X_{d-k}\supset X_{d-k-1} \supset ... \supset X_0 \supset X_{-1}=\emptyset$$
be a topological pseudomanifold such that $X_{d-k}$ is a subspace of $Y$. We say that the pair $(X,Y)$ has a
{\it conical structure
with respect to the stratified subspace} if for every $r\geq k$ there exits an open neighbourhood
$TX_{d-r}$ of $X_{d-r}\setminus X_{d-r-1}$ in $X\setminus X_{d-r-1}$, with the following properties:
\begin{enumerate}
\item Let $\overline{TX_{d-r}}$ be the closure of $TX_{d-r}$ in $X$. There is a locally trivial fibration of $2(r-k+1)$-uples of
spaces
$$\begin{array}{c}
(\overline{TX_{d-r}}\setminus X_{d-r-1})\cap  {(X,Y,\overline{TX_{d-k}}, X_{d-k}, \overline{TX_{d-k-1}}, X_{d-k-1},
...,\overline{TX_{d-r+1}}, X_{d-r+1})}\\ \xdownarrow{5mm} \sigma_{d-r}\\ X_{d-r}\setminus X_{d-r-1} \end{array}$$
such that its restriction to the boundary
$$\begin{array}{c}
(\partial\overline{TX_{d-r}}\setminus X_{d-r-1})\cap { (X,Y,\overline{TX_{d-k}}, X_{d-k}, \overline{TX_{d-k-1}}, X_{d-k-1},
...,\overline{TX_{d-r+1}}, X_{d-r+1})}\\ \xdownarrow{5mm} \sigma^\partial_{d-r}\\ X_{d-r}\setminus X_{d-r-1}\end{array}$$
is a locally trivial fibration.
\item\label{prop2} The fibration $\sigma_{d-r}$ is the cone of $\sigma^\partial_{d-r}$ over the base $X_{d-r}\setminus X_{d-r-1}$.
\item Let $k\leq r_1 < r_2 \leq d$ and consider the isomorphism induced by property (\ref{prop2})
$$\overline{TX_{d-r_1}}\cap (\overline{TX_{d-r_2}}\setminus X_{d-r_2-1}) \cong
(\overline{TX_{d-r_1}}\cap (\partial\overline{TX_{d-r_2}}\setminus X_{d-r_2-1})) \times [0,1]/\sim$$
where $\sim$ is the equivalence relation of the mapping cylinder.

If we remove $X_{d-r_1-1}$ in both parts of this isomorphism, we obtain an isomorphism
$$\phi_{r_1, r_2}:(\overline{TX_{d-r_1}}\setminus X_{d-r_1-1})\cap \overline{TX_{d-r_2}} \cong
((\overline{TX_{d-r_1}}\setminus X_{d-r_1-1})\cap \partial\overline{TX_{d-r_2}}) \times [0,1).$$
Note that since $X_{d-r_2}$ is contained in $X_{d-r_1-1}$, the vertex section of the cone is not included in the previous spaces.

With this notation, we have the equality
$$(\sigma_{d- r_1})_{|(\overline{TX_{d-r_1}}\setminus X_{d-r_1-1})\cap \overline{TX_{d-r_2}}}= \phi_{r_1, r_2}^{-1} \circ
((\sigma_{d- r_1})_{|(\overline{TX_{d-r_1}}\setminus X_{d-r_1-1})\cap \partial\overline{TX_{d-r_2}}}, Id_{[0,1)}) \circ \phi_{r_1, r_2},$$
that is, the fibration $\sigma_{d- r_1}$ in the intersection $(\overline{TX_{d-r_1}}\setminus X_{d-r_1-1})\cap \overline{TX_{d-r_2}}$ is
determined by its restriction to $(\overline{TX_{d-r_1}}\setminus X_{d-r_1-1})\cap \partial\overline{TX_{d-r_2}}$.

Figure \ref{fig:propiedad_3} illustrates this property.
\item Let $k\leq r_1 < r_2 \leq d$. If $\partial \overline{TX_{d-r_2}}\cap (X_{d-r_1}\setminus X_{d -r_1-1}) \neq \emptyset$, then we have the following
equality of $2(r_1-k+1)$-uples
$$(\overline{TX_{d-r_1}}\setminus X_{d-r_1-1})\cap (X,Y,\overline{TX_{d-k}}, X_{d-k}, \overline{TX_{d-k-1}}, X_{d-k-1},...,\overline{TX_{d-r_1+1}}, X_{d-r_1+1})
\cap $$ $$\cap\partial\overline{TX_{d-r_2}}=\sigma_{d-r_1}^{-1}(\partial\overline{TX_{d-r_2}}\cap (X_{d-r_1}\setminus X_{d -r_1-1}))$$
and, in this space, we have
$$\sigma^\partial_{d-r_2}\circ \sigma_{d-r_1} =\sigma^\partial_{d-r_2}.$$

Figure \ref{fig:propiedad_4} illustrates this property.
\end{enumerate}
\end{definition}

\begin{notation}
\label{not:fibrelinkbundle}
Consider a tuple $(Y_1,...,Y_l)$ of subspaces whose componets are a subset of the components of the tuple
$(X,Y,\overline{TX_{d-k}}, X_{d-k}, \overline{TX_{d-k-1}}, X_{d-k-1},...,\overline{TX_{d-r+1}}, X_{d-r+1})$
considered in the previous definition. The fibre of the fibration $\sigma^\partial_{d-r}$ restricted to the tuple $(Y_1,...,Y_l)$ is called
the {\em fibre of the link bundle of $(Y_1,...,Y_l)$ over $X_{d-r}$}.
\end{notation}

\begin{remark}
The fibration $\sigma^\partial_{d-r}$ is the fibration of links of $X_{d-r}\setminus X_{d-r-1}$ and the fibration $\sigma_{d-r}$ is the fibration associated to a tubular neighborhood.

The fact that these fibrations of $2(r-k+1)$-uples are locally trivial yields that
the intersection of the link $L^x$ of over the point $x \in X_{d-r}\setminus X_{d-r-1}$ with the open neighbourhoods
$TX_{d-k},TX_{d-k-1},...,TX_{d-r+1}$ only depends on the connected component of $X_{d-r}\setminus X_{d-r-1}$ containing $x$.
\end{remark}

Let $k\leq r_1 < r_2 \leq d$. The following figure shows how the open neighbourhoods $TX_{d-r}$ intersect each other.
\begin{figure}[H]
  \centering
  \includegraphics[width=13cm]{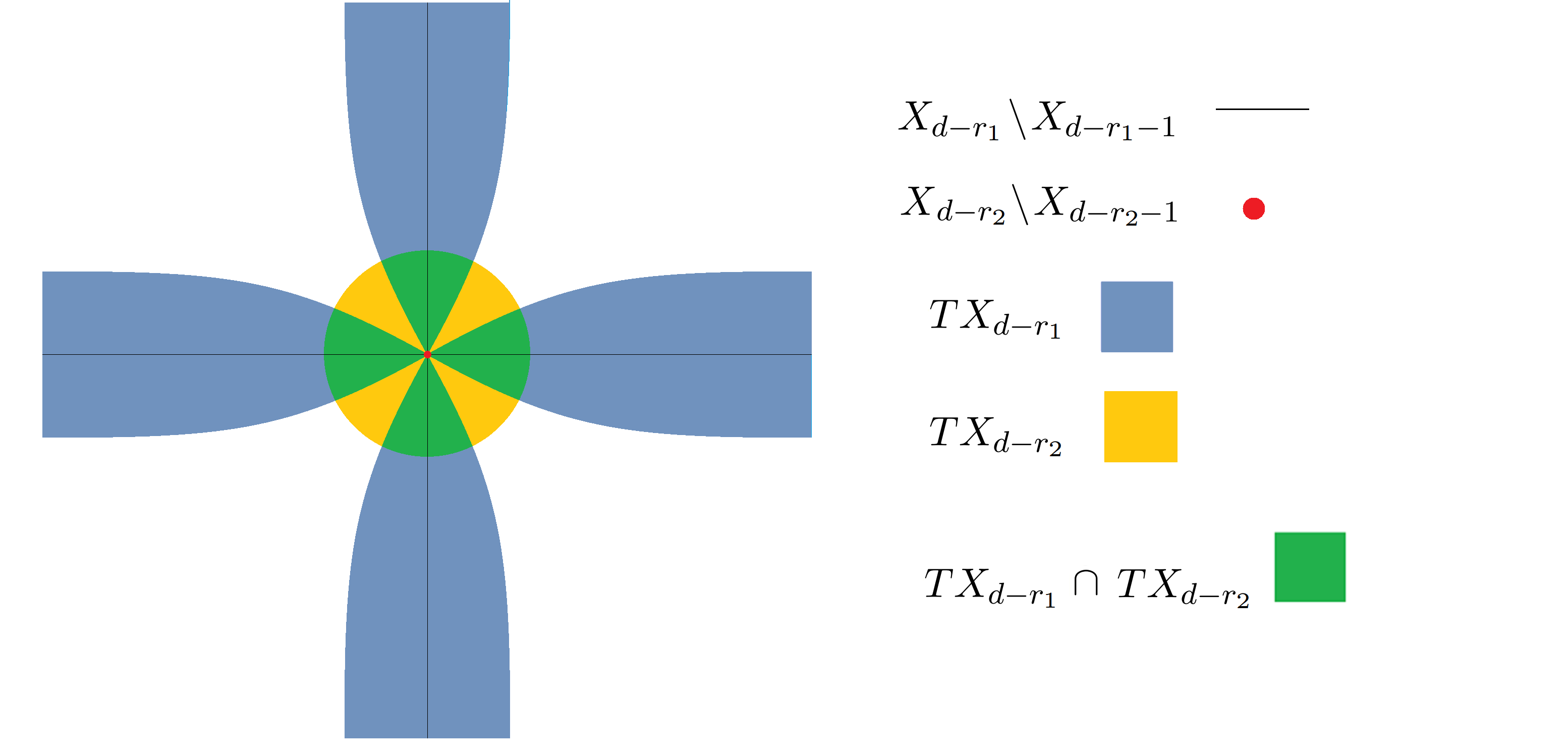}
  \caption{}\label{fig:conical_structure}
\end{figure}

The following figures show how are the morphisms
$\sigma_{d-r_1}$ and $\sigma^\partial_{d-r_2}$ in $(\overline{TX_{d-r_1}}\setminus X_{d-r_1-1})\cap \overline{TX_{d-r_2}}$ because of 
Properties (3) and (4) of Definition \ref{def:conical}.

\begin{figure}[H]
  \centering
  \includegraphics[width=13cm]{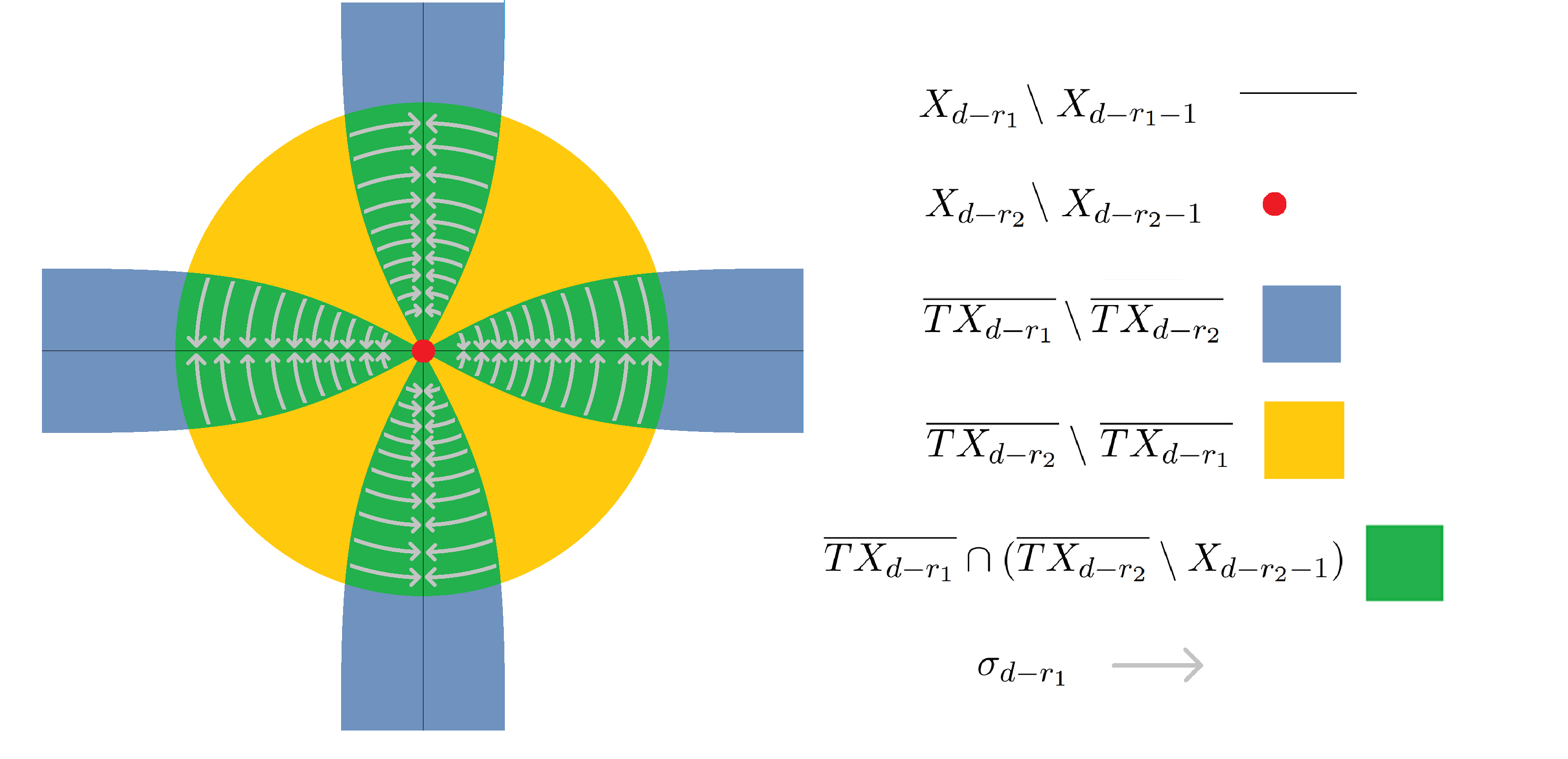}
  \caption{Property (3)}\label{fig:propiedad_3}
\end{figure}

\begin{figure}[H]
  \centering
  \includegraphics[width=13cm]{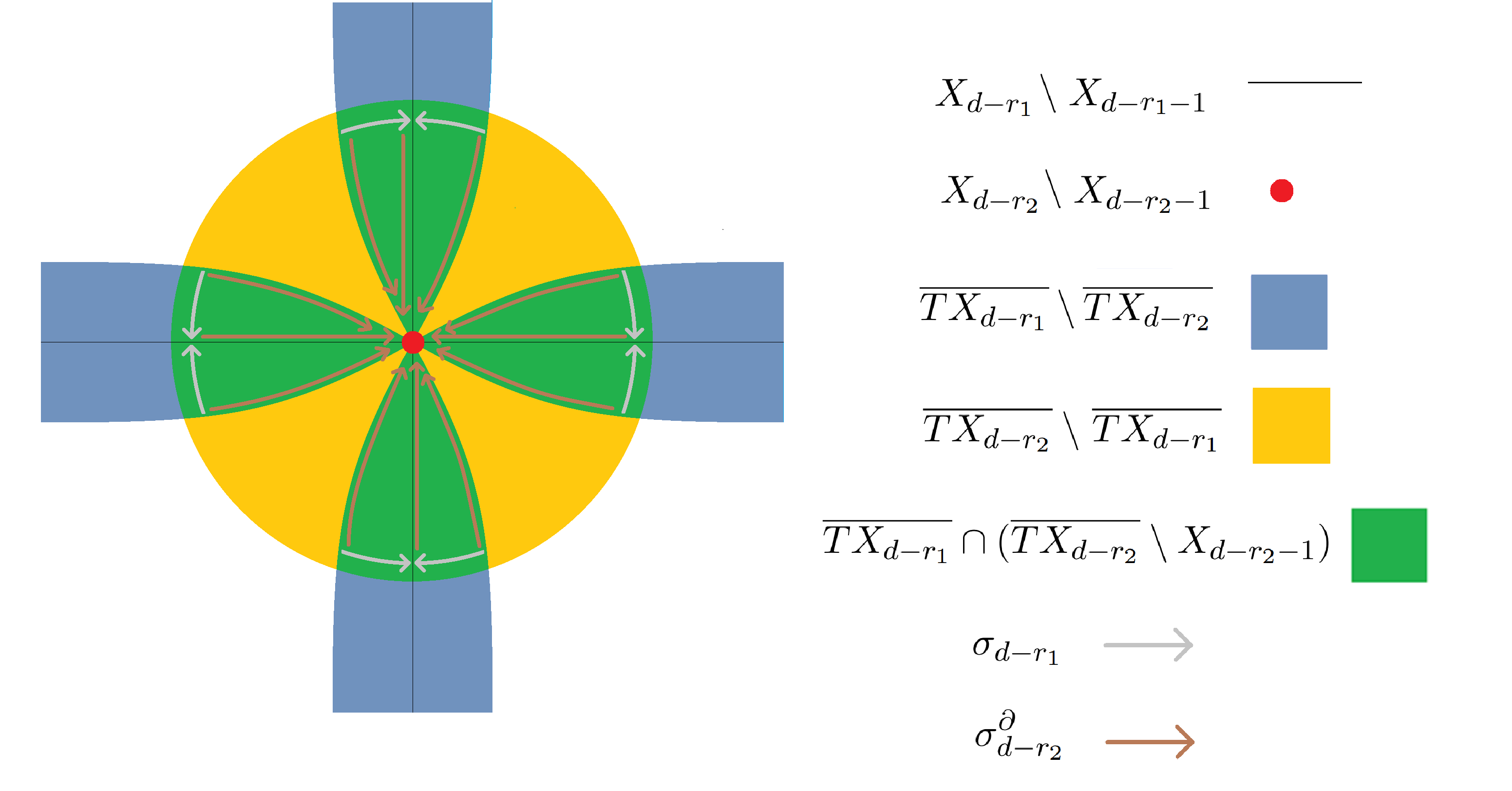}
  \caption{Property (4)}\label{fig:propiedad_4}
\end{figure}

\begin{remark}
Using properties (3) and (4) of Definition \ref{def:conical}, we can also deduce that
$$(\overline{TX_{d-r_1}}\setminus X_{d-r_1-1})\cap (X,Y,\overline{TX_{d-k}}, X_{d-k}, \overline{TX_{d-k-1}}, X_{d-k-1},...,\overline{TX_{d-r_1+1}}, X_{d-r_1+1})
\cap \overline{TX_{d-r_2}}=$$ $$=\sigma_{d-r_1}^{-1}(\overline{TX_{d-r_2}}\cap (X_{d-r_1}\setminus X_{d -r_1-1}))$$
and, in this space, we have
$$\sigma_{d-r_2}\circ \sigma_{d-r_1} =\sigma_{d-r_2}.$$
\end{remark}

\begin{notation}
Along this chapter, we will use the superindex $\partial$ to denote the fibrations of boundaries of suitable tubular neighborhoods.
\end{notation}

\begin{definition}\label{def:Tr}
We say that a conical structure verifies the \emph{$r$-th triviality property} $(T_r)$ if the locally trivial fibration $\sigma^\partial_{d-r}$ is trivial,
that is, the following two properties hold for every connected component $S_{d-r}$ of $X_{d-r} \setminus X_{d-r-1}$.
\begin{enumerate}
  \item There exist an isomorphism
  $$(\sigma^\partial_{d-r})^{-1}(S_{d-r})\cong L\times S_{d-r},$$
  where $L$ denotes the $2(r-k+1)$-uple of the links of $S_{d-r}$ in
  $$(X,Y,\overline{TX_{d-k}}, X_{d-k}, \overline{TX_{d-k-1}}, X_{d-k-1},...,\overline{TX_{d-r+1}}, X_{d-r+1}).$$
  \item Under the identification given by property (1), $\sigma^\partial_{d-r}$ restricted to $L \times S_{d-r}$ is the canonical projection over $S_{d-r}$.
\end{enumerate}
\end{definition}

\begin{definition}
\label{def:systemtr}
Let $(X,Y)$ be a pair of spaces with a conical structure as in Definition \ref{def:conical} which verifies the \emph{$r$-th triviality property} $(T_r)$
for any $r$.
Fix a trivialization
$$(\sigma^\partial_{d-r})^{-1}(S_{d-r})\cong L\times S_{d-r}$$
over each connected componet of each stratum. The set of all trivializations is called a {\em system of trivializations for the conical structure}.
\end{definition}

Let $(X,Y)$ be a pair of spaces with a conical structure as in Definition \ref{def:conical}. Fix a system of trivializations for the conical structure.

Let $k\leq r_1 <r_2 \leq d$ verifying that there exist connected components $S_{d-r_1}$ and $S_{d-r_2}$ of $X_{d-r_1} \setminus X_{d-r_1-1}$ and
$X_{d-r_2} \setminus X_{d-r_2-1}$, respectively, such that $(\sigma_{d-r_2}^\partial)^{-1}(S_{d-r_2})\cap S_{d-r_1} \neq \emptyset$, or what is the same,
that the closure $\overline{S_{d-r_1}}$ contains $S_{d-r_2}$.

By definition,  $(\sigma_{d-r_2}^\partial)^{-1}(S_{d-r_2})\cap \sigma_{d-r_1}^{-1}(S_{d-r_1})$ and
$\sigma_{d-r_1}^{-1}((\sigma_{d-r_2}^\partial)^{-1}(S_{d-r_2}) \cap S_{d-r_1})$ are $2(k-r+1)$-uplas.
To simplify the notation, along the following reasoning we denote by
$(\sigma_{d-r_2}^\partial)^{-1}(S_{d-r_2})\cap \sigma_{d-r_1}^{-1}(S_{d-r_1})$ and
$\sigma_{d-r_1}^{-1}((\sigma_{d-r_2}^\partial)^{-1}(S_{d-r_2}) \cap S_{d-r_1})$ the first componets of these $2(k-r+1)$-uplas. Consider also the space
$(\sigma_{d-r_2}^\partial)^{-1}(S_{d-r_2})\cap S_{d-r_1}$.

Since we have fixed a system of trivializations, we have isomorphisms
$$L^{S_{d-r_1}}_{d-r_2} \times S_{d-r_2} \cong (\sigma_{d-r_2}^\partial)^{-1}(S_{d-r_2})\cap S_{d-r_1}$$
and
$$L^{\sigma_{d-r_1}^{-1}(S_{d-r_1})}_{d-r_2}\times S_{d-r_2} \cong (\sigma_{d-r_2}^\partial)^{-1}(S_{d-r_2})\cap \sigma_{d-r_1}^{-1}(S_{d-r_1})$$
where $L^{S_{d-r_1}}_{d-r_2}$ denotes the fibre of the link bundle of $S_{d-r_1}$ over $S_{d-r_2}$ and $L^{\sigma_{d-r_1}^{-1}(S_{d-r_1})}_{d-r_2}$ denotes
the fibre of the link bundle of $\sigma_{d-r_1}^{-1}(S_{d-r_1})$ over $S_{d-r_2}$.
Moreover, by the property (4) of Definition \ref{def:conical}, we have an equality
$$\sigma_{d-r_1}^{-1}(S_{d-r_1}) \cap (\sigma_{d-r_2}^\partial)^{-1}(S_{d-r_2})=\sigma_{d-r_1}^{-1}((\sigma_{d-r_2}^\partial)^{-1}(S_{d-r_2})\cap S_{d-r_1})$$
and, again using the fixed system of trivializations, we obtain an isomorphism
$$\sigma_{d-r_1}^{-1}((\sigma_{d-r_2}^\partial)^{-1}(S_{d-r_2})\cap S_{d-r_1})\cong c(L^X_{d-r_1}) \times ((\sigma_{d-r_2}^\partial)^{-1}(S_{d-r_2})\cap S_{d-r_1})$$
where $L^X_{d-r_1}$ is the fibre of the link bundle of $X$ over $S_{d-r_1}$.

Combining the previous isomorphisms, we have
$$L^{\sigma_{d-r_1}^{-1}(S_{d-r_1})}_{d-r_2}\times S_{d-r_2}\cong c(L^X_{d-r_1}) \times ((\sigma_{d-r_2}^\partial)^{-1}(S_{d-r_2})\cap S_{d-r_1}) \cong
c(L^X_{d-r_1}) \times L^{S_{d-r_1}}_{d-r_2} \times S_{d-r_2}$$
Let us denote by $\gamma$ the isomorphism
\begin{equation} 
\label{map:gamma}
\gamma:c(L^X_{d-r_1}) \times L^{S_{d-r_1}}_{d-r_2} \times S_{d-r_2} \to L^{\sigma_{d-r_1}^{-1}(S_{d-r_1})}_{d-r_2}\times S_{d-r_2}.
\end{equation}

Since under the equivalences given by the trivializations the morphisms $\sigma_{d-r_1}$ and $\sigma_{d-r_2}^\partial$ are the canonical projections, using the
property (4) of Definition \ref{def:conical}, the diagram
$$\xymatrix{c(L^X_{d-r_1}) \times L^{S_{d-r_1}}_{d-r_2} \times S_{d-r_2}
\ar[r] \ar[d]_\gamma & L^{S_{d-r_1}}_{d-r_2} \times S_{d-r_2} \ar[d]\\
L^{\sigma_{d-r_1}^{-1}(S_{d-r_1})}_{d-r_2}\times S_{d-r_2}\ar[r] & S_{d-r_2}}$$
where all the morphisms except $\gamma$ are the canonical projections, is commutative.

So $\gamma$ verifies the following condition:
$$\xymatrix@R=1mm{\gamma:c(L^X_{d-r_1}) \times L^{S_{d-r_1}}_{d-r_2} \times S_{d-r_2}\ar[r] & L^{\sigma_{d-r_1}^{-1}(S_{d-r_1})}_{d-r_2}\times S_{d-r_2}\\
(x,y,z) \ar[r] & (\gamma_1(x,y,z),z) }$$

\begin{definition}
\label{def:systriv}
We say that the system of trivializations is {\em compatible} if for any two connected components $S_{d-r_1}$ and $S_{d-r_2}$ as above, the map
if $\gamma_1$ does not depend on $z$, that is, if there exist an isomorphism $\beta: c(L^X_{d-r_1}) \times L^{S_{d-r_1}}_{d-r_2} \to L^{\sigma_{d-r_1}^{-1}(S_{d-r_1})}_{d-r_2}$
such that $\gamma=(\beta, Id_{S_{d-r_2}})$.
\end{definition}

\begin{definition}\label{def:strongtriv}
We say that the conical structure is \emph{trivial} if it  verifies the \emph{$r$-th triviality property} $(T_r)$
for any $r$ and there exists a compatible system of trivializations.
\end{definition}

\begin{remark}
\label{re:fixconical}
For a great variety of topological pseudomanifolds
$$ X=X_d\supset X_{d-2} \supset ... \supset X_0 \supset X_{-1}=\emptyset,$$
the pair $(X,X_{d-2})$ has a conical structure with respect to the stratification
$$X_{d-2} \supset ... \supset X_0 \supset X_{-1}=\emptyset.$$
Whitney stratifications, for instance, verify this property (see \cite{GWPL}).
We consider every topological pseudomanifold which appear in sections \ref{sec:top}, \ref{sec:sequence} and \ref{sec:sheafification} holds this.
Moreover,
we fix such a conical structure and denote the relevant neighbourhoods $TX_{d-r}$ for $r$ varying.

From Section \ref{sec:axiomatic}, it is not necessary to do this assumption
\end{remark}

\begin{remark}
Toric varieties with its canonical stratification are topological pseudomanifolds which have a {\em trivial} conical structure with respect to the
stratification. This is derived easily via the torus action.
\end{remark}

\subsection{An inductive construction of intersection spaces}

Given a topological pseudomanifold we define an inductive procedure on the depth of the strata. The procedure depends on choices made at each
inductive step, and may be obstructed for a given set of choices or carried until the deepest stratum. If for a given set of choices can be carried until
the end, it produces a pair of spaces which generalizes Banagl intersection spaces.

\begin{definition}\label{def:perversity}
A perversity is a map $\bar p:\ZZ_{\geq 2}\rightarrow \ZZ_{\geq 0}$ such that $\bar p(2)=0$ and $\bar p(k) \leq \bar p(k+1) \leq \bar p(k) + 1$.
\end{definition}

Some special perversities are

\begin{itemize}
\item The zero perversity, $\bar 0(k)=0$.
\item The total perversity, $\bar t(k)=k-2$.
\item The lower middle perversity, $\bar m(k)=\lfloor\frac{k}{2}\rfloor -1$.
\item The upper middle perversity, $\bar n(k)=\lceil\frac{k}{2}\rceil -1$.
\item Given a perversity $\bar p$, the complementary perversity is $\bar t-\bar p$. It is usually denoted by $\bar q$.
\end{itemize}

The lower and the upper middle perversities are complementary.

For our construction we need the notion of fibrewise homology truncation of fibrations of pairs of locally finite $CW$-complexes:

\begin{definition}
\label{def:fibrewisetruncation}
Let $\sigma:(X,Y)\to B$ be a locally trivial fibration. We say that $\sigma$ admits a \emph{fibrewise rational
$q$-homology truncation} if
there exits a morphism of pairs of spaces $$\phi:(X_{\leq q},Y_{\leq q})\to (X,Y)$$ such that $\sigma\circ \phi$ is a locally
trivial fibration and, for any $b\in B$,
\begin{enumerate}
 \item the homomorphism in homology of fibres
$$H_i((X_{\leq q})_b,(Y_{\leq q})_b;\QQ)\to H_i(X_b,Y_b;\QQ)$$
is an isomorphism if $i\leq q$.
\item the homology group $H_i((X_{\leq q})_b,(Y_{\leq q})_b;\QQ)$ vanishes if $i> q$.
\end{enumerate}
\end{definition}

\begin{notation}
 Given a pair of spaces $(X,Y)$ we denote by $(X,Y)_{\leq q}$ the pair $(X_{\leq q},Y_{\leq q})$ appearing in the definition above.
\end{notation}

The following figure shows a fibrewise rational $1$-homology truncation of the fibration $\sigma$ in Figure \ref{fig:fibration_cone}. The fiber of the
resulting fibration becomes a pointed circle.
\begin{figure}[H]
  \centering
  \includegraphics[width=13cm]{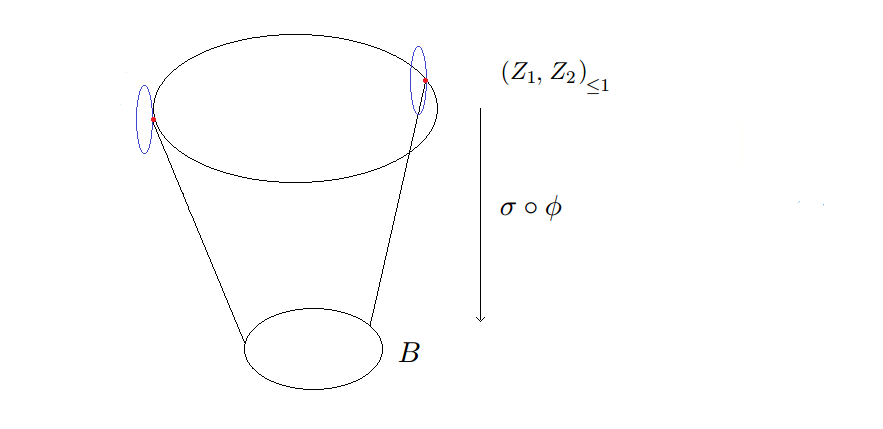}
  \caption{}\label{fig:truncation}
\end{figure}

\begin{definition}
Given a pair of spaces $(X,Y)$ with a conical structure as in Definition \ref{def:conical}, a fibrewise rational $q$-homology truncation of $\sigma^\partial_{d-r}$
$$\xymatrix{(\partial\overline{TX_{d-r}}\setminus X_{d-r-1}\cap (X,Y))_{\leq q}\ar[rd]_{(\sigma^\partial_{d-r})_{\leq q}}\ar[dd]^{\phi_{d-r}^\partial}& \\
& X_{d-r}\setminus X_{d-r-1} \\ \partial\overline{TX_{d-r}}\setminus X_{d-r-1} \cap (X,Y) \ar[ru]^{\sigma^\partial_{d-r}} &}$$
is \emph{compatible with the conical structure} if, for every $r'>r$,
$$\sigma_{d-r'} \circ (\sigma^\partial_{d-r})_{\leq q}: (\sigma^\partial_{d-r})_{\leq q}^{-1}(\overline{TX_{d-r'}}\cap (X_{d-r}\setminus X_{d-r-1}))
\to X_{d-r'}\setminus X_{d-r'-1}$$
is the cone of
$$\sigma_{d-r'}^\partial \circ (\sigma^\partial_{d-r})_{\leq q}: (\sigma^\partial_{d-r})_{\leq q}^{-1}(\partial\overline{TX_{d-r'}}\cap (X_{d-r}\setminus X_{d-r-1}))
\to X_{d-r'}\setminus X_{d-r'-1}$$
\end{definition}

\begin{proposition}\label{prop:fibcom}
Given a pair of spaces $(X,Y)$ with a conical structure as in Definition \ref{def:conical}, if there exist a fibrewise rational $q$-homology truncation of $\sigma^\partial_{d-r}$,
then there exist a fibrewise rational $q$-homology truncation of $\sigma^\partial_{d-r}$ compatible with the conical structure.
\end{proposition}
\begin{proof}
Let us consider a fibrewise rational $q$-homology truncation of $\sigma^\partial_{d-r}$, $(\sigma^\partial_{d-r})_{\leq q}$.

Then, the restriction of $(\sigma^\partial_{d-r})_{\leq q}$ to
$$(\sigma^\partial_{d-r})_{\leq q}^{-1}((X_{d-r}\setminus X_{d-r-1})\setminus \bigcup_{r'>r} TX_{d-r'})$$
is a fibrewise rational $q$-homology truncation of the restriction of $\sigma^\partial_{d-r}$ to
$$(\sigma^\partial_{d-r})^{-1}((X_{d-r}\setminus X_{d-r-1})\setminus \bigcup_{r'>r} TX_{d-r'})$$

Using the property (2) of Definition \ref{def:conical}, we can extend the previous restriction to a fibrewise rational $q$-homology truncation over
$$(X_{d-r}\setminus X_{d-r-1})\setminus \bigcup_{r''>r'} TX_{d-r''}$$
for any $r'>r$ inductively. Moreover, using the property (3) of Definition \ref{def:conical}, we can check that these extensions are compatible with the
conical structure. So, when $r'=d$, we obtain a fibrewise rational $q$-homology truncation of $\sigma^\partial_{d-r}$ compatible with the conical structure.
\end{proof}

\begin{definition}
\label{def:comptriv}
Given a pair of spaces $(X,Y)$ with a conical structure verifying the triviality property $(T_r)$ for any $r$ (see Definition \ref{def:Tr}), choose a
system of trivializations as in Definition~\ref{def:systemtr}.
A fibrewise rational $q$-homology truncation of $\sigma^\partial_{d-r}$
$$\xymatrix{(\partial\overline{TX_{d-r}}\setminus X_{d-r-1}\cap (X,Y))_{\leq q}\ar[rd]_{(\sigma^\partial_{d-r})_{\leq q}}\ar[dd]^{\phi_{d-r}^\partial}& \\
& X_{d-r}\setminus X_{d-r-1} \\ \partial\overline{TX_{d-r}}\setminus X_{d-r-1} \cap (X,Y) \ar[ru]^{\sigma^\partial_{d-r}} &}$$
is \emph{compatible with the trivialization} if the following conditions hold.
\begin{enumerate}
\item Given a connected component $S_{d-r}$ of $X_{d-r} \setminus X_{d-r-1}$, if $L_{d-r}=(L^X_{d-r},L^Y_{d-r})$ denotes the
fibre of the link bundle of $(X,Y)$ over $S_{d-r}$ and
$$(\sigma^\partial_{d-r})^{-1} (S_{d-r}) \cong L_{d-r} \times S_{d-r}$$
is the isomorphism induced by the system of trivializations,
then there exist a pair of spaces $(L_{d-r})_{\leq q}:=((L^X_{d-r})_{\leq q}, (L^Y_{d-r})_{\leq q})$ such that
the group $H_i((L^X_{d-r})_{\leq q},(L^Y_{d-r})_{\leq q};\QQ)$  vanishes if $i> q$, we have an isomorphism
$$(\sigma^\partial_{d-r})_{\leq q}^{-1} (S_{d-r}) \cong (L_{d-r})_{\leq q} \times S_{d-r}$$
and, under these identifications, $((\sigma^\partial_{d-r})_{\leq q})_{|(L_{d-r})_{\leq q} \times S_{d-r}}$
is the canonical projection and $(\phi_{d-r}^\partial)_{|(L_{d-r})_{\leq q} \times S_{d-r}} = (\phi_1, Id_{S_{d-r}})$ where $\phi_1:(L_{d-r})_{\leq q} \to L_{d-r}$
is a morphism such that
$$H_i(\phi_1):H_i((L^X_{d-r})_{\leq q},(L^Y_{d-r})_{\leq q};\QQ)\to H_i(L^X_{d-r},L^Y_{d-r};\QQ)$$
is an isomorphism if $i\leq q$.
\item Given $r'>r$ and a connected component $S_{d-r'}$ of $X_{d-r'} \setminus X_{d-r'-1}$ such that
$$(\sigma_{d-r'}^\partial)^{-1}(S_{d-r'})\cap S_{d-r} \neq \emptyset,$$
let $L^{S_{d-r}}_{d-r'}$ and $L^{\sigma^{-1}_{d-r}(S_{d-r})}_{d-r'}$ denote the fibres of the link bundles of $S_{d-r}$ and $\sigma^{-1}_{d-r}(S_{d-r})$
over $S_{d-r'}$ respectively.
Moreover, let
$$\gamma:c(L^X_{d-r})\times L^{S_{d-r}}_{d-r'} \times S_{d-r'} \cong L^{\sigma^{-1}_{d-r}(S_{d-r})}_{d-r'} \times S_{d-r'}$$ be
the isomorphism defined in Equation~(\ref{map:gamma}), in the discussion preceeding Definition~\ref{def:systemtr}.
Then, the image of the composition
$$\xymatrix@C=2cm{c((L^X_{d-r})_{\leq q})\times L^{S_{d-r}}_{d-r'} \times S_{d-r'}\ar[r]^{(c(\phi_1),Id)} &
c(L^X_{d-r})\times L^{S_{d-r}}_{d-r'} \times S_{d-r'} \ar[d]^{\gamma}\\ & L^{\sigma^{-1}_{d-r}(S_{d-r})}_{d-r'} \times S_{d-r'}}$$
is equal to $A \times S_{d-r'}$ for some subset $A\subset L^{\sigma^{-1}_{d-r}(S_{d-r})}_{d-r'}$.
\end{enumerate}
\end{definition}

\begin{remark}
If the conical structure is trivial (see Definition \ref{def:strongtriv}), the condition (1) of the previous definition implies the condition (2).
\end{remark}

\begin{remark}
If a fibrewise rational $q$-homology truncation of $\sigma^\partial_{d-r}$ is compatible with the trivialization, then it is also compatible with the conical structure.
\end{remark}

\textsc{The initial step of the induction}.

Let $X$ be a topological pseudomanifold such that the pair $(X,X_{d-2})$ has a conical structure with respect to the stratification, we consider the open
neighbourhoods $TX_{d-r}$ fixed in Remark~\ref{re:fixconical}.

Let $m$ be the minimum such that $X_{d-m} \setminus X_{d-m-1} \neq \emptyset$. If the fibration
$$\sigma^\partial_{d-m}:\partial\overline{TX_{d-m}}\setminus X_{d-m-1}\to X_{d-m}\setminus X_{d-m-1}$$
predicted in Definition~\ref{def:conical} does not admit
a fibrewise rational $\bar q(m)$-homology truncation, then the intersection space does not exist. Otherwise we choose a
fibrewise rational $\bar q(m)$-homology truncation compatible with the conical structure (see Proposition \ref{prop:fibcom})
\begin{equation}
\label{diag:trunc}
 \xymatrix{(\partial\overline{TX_{d-m}}\setminus X_{d-m-1})_{\leq \bar q(m)}\ar[rd]_{(\sigma^\partial_{d-m})_{\leq \bar q(m)}}\ar[dd]^{\phi_{d-m}^\partial}& \\ & X_{d-m}\setminus X_{d-m-1} \\ \partial\overline{TX_{d-m}}\setminus X_{d-m-1} \ar[ru]^{\sigma^\partial_{d-m}} &}
\end{equation}

We construct a new space $X'_m$, a homotopy equivalence $\pi_m:X'_m\to X$ with contractible fibres and a subspace
$I^{\bar p}_{m} X\hookrightarrow X'_m$ as follows.

Define the map
$$(\sigma_{d-m})_{\leq \bar q(m)}:cyl((\sigma^\partial_{d-m})_{\leq \bar q(m)})\to X_{d-m}\setminus X_{d-m-1}$$
to be the cone of the fibration $(\sigma^\partial_{d-m})_{\leq \bar q(m)}$ over $X_{d-m}\setminus X_{d-m-1}$.
By property (2) of Definition~\ref{def:conical} there exits a fibre bundle morphism
$$\phi_{d-m}:cyl((\sigma^\partial_{d-m})_{\leq \bar q(m)})\to \overline{TX_{d-m}}\setminus X_{d-m-1}$$
over the base $X_{d-m}\setminus X_{d-m-1}$ which preserves the vertex sections. Let
$$\theta_{d-m}:cyl((\sigma^\partial_{d-m})_{\leq \bar q(m)})\to X\setminus X_{d-m-1}$$
be the composition of the fibre bundle morphism $\phi_{d-m}$ with the natural inclusion of the closed subset $\overline{TX_{d-m}}\setminus X_{d-m-1}$ into $X\setminus X_{d-m-1}$.
Let $\cyl(\theta_{d-m})$ be the mapping cylinder of $\theta_{d-m}$.
It is by definition the union $cyl((\sigma^\partial_{d-m})_{\leq \bar q(m)})\times[0,1]\coprod (X\setminus X_{d-m-1})$
under the usual equivalence relation.
Denote by $$s_{d-m}:X_{d-m}\setminus X_{d-m-1}\to cyl((\sigma^\partial_{d-m})_{\leq \bar q(m)})$$
the vertex section. We define $Z_m$ to be the result of quotienting $\cyl(\theta_{d-m})$ by the equivalence relation which identifies,
for any
$x\in X_{d-m}\setminus X_{d-m-1}$, the subspace $s_{d-m}(x)\times [0,1]$ to a point.

In order to follow in an easier way our further constructions, observe at this point that the mapping cylinder $cyl(\phi^\partial_{d-m})$ of the
vertical map $\phi^\partial_{d-m}$ of diagram~(\ref{diag:trunc}) above is a subspace both of $\cyl(\theta_{d-m})$ and of $Z_m$.

The following figures show $\cyl(\theta_{d-m})$ and $Z_m$.

\begin{minipage}{0.65\textwidth}
  \includegraphics[width=1\linewidth]{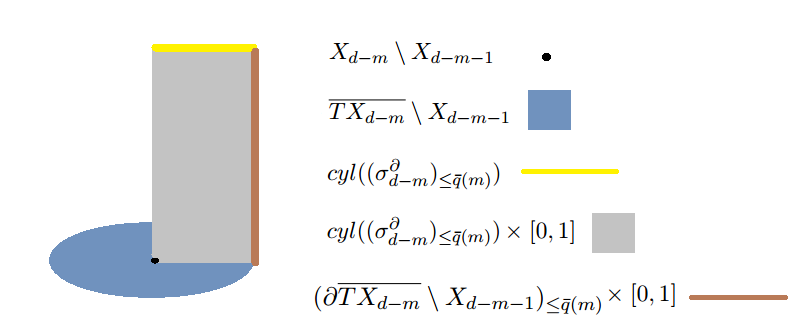}
  \captionof{figure}{$\cyl(\theta_{d-m})$}\label{fig:cylinder_theta}
\end{minipage}
\begin{minipage}{0.35\textwidth}
  \includegraphics[width=1\linewidth]{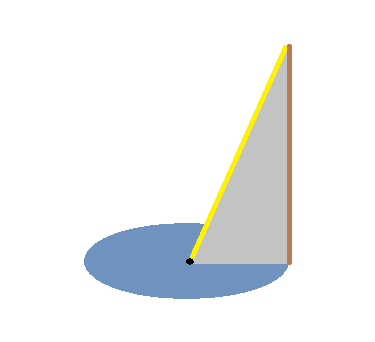}
  \captionof{figure}{$Z_m$}\label{fig:homotopic_model}
\end{minipage}

\bigskip

The equivalence relation collapses the vertical line over the origin of the horizontal plane in
Figure~\ref{fig:cylinder_theta}, therefore the yellow line in Figure~\ref{fig:homotopic_model} becomes diagonal.

We have a natural projection map $\pi_m:Z_m\to X\setminus X_{d-m-1}$
which is a homotopy equivalence whose fibres are contractible and has a natural section denoted by $\alpha_m$. $\alpha_m$ is a closed inclusion of
$X\setminus X_{d-m-1}$ into $Z_m$. In the previous figure, $\pi_m$ is the projection onto the horizontal plane and $\alpha_m$ is the inclusion of the
horizontal plane in the rest of the picture.

Define $X'_m$ as the set $Z_m\cup X_{d-m-1}$. The projection map extends to a projection
$$\pi_m:X'_m\to X$$

Consider in $X'_m$ the topology spanned by all the open subsets of $Z_m$ and the collection of subsets of the form
$\pi_m^{-1}(U)$ for any open subset $U$ of $X$.

This projection is also a homotopy equivalence whose fibres are contractible and such that the natural section $\alpha_m$ extends to it giving a
closed inclusion of $X$ into $X'_m$.

Define the step $m$ intersection space
$I^{\bar p}_{m}X$ to be the subspace of $X'_m$ given by
$$I^{\bar p}_{m}X:=cyl((\sigma^\partial_{d-m})_{\leq \bar q(m)})\times\{0\}\cup cyl(\phi_{d-m}^\partial) \cup (X \setminus TX_{d-m}),$$
with the restricted topology.

The following figure shows $I^{\bar p}_{m}X$.
\begin{figure}[H]
  \centering
  \includegraphics[width=13cm]{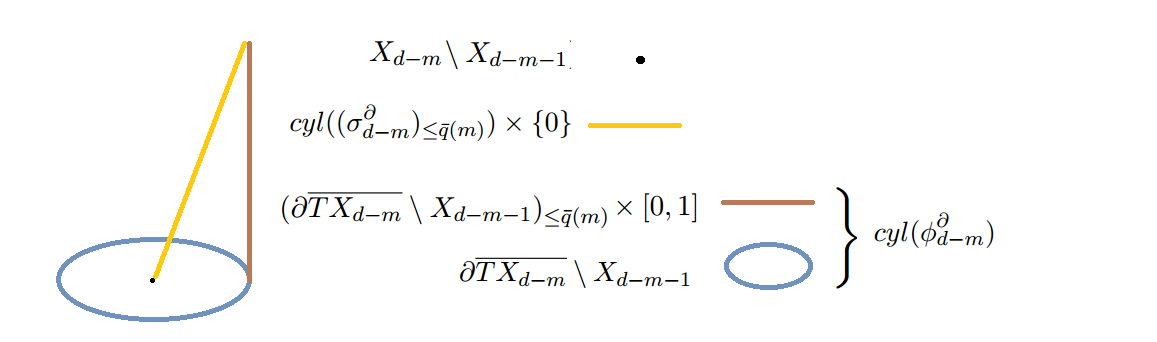}
  \caption{ }\label{fig:intersection_space}
\end{figure}

\begin{remark}\label{rem:description}
Note that we have the equality $cyl(\phi_{d-m}^\partial) = \pi_m^{-1}(\partial\overline{TX_{d-m}}\setminus X_{d-m-1})$.
\end{remark}

With the above definitions we have the following chains of inclusions
$$X'_m\supset X\supset X_{d-m}\supset...\supset X_0,$$
$$X'_m\supset I^{\bar p}_{m}X\supset X_{d-m}\supset...\supset X_0,$$
where $X$ is embedded in $X'_m$ via the section $\alpha_m$.

An immediate consequence of our construction is

\begin{lemma}
The pairs $(X'_m,X_{d-m})$ and $(I^{\bar p}_{m}X,X_{d-m})$ have a conical structure with respect to the stratified subspace
$X_{d-m-1}\supset...\supset X_0$, given by the following open neighbourhoods of $X_{d-r}\setminus X_{d-r-1}$: $\pi_m^{-1}(TX_{d-r})$ is a neighbourhood
in $X'_m\setminus X_{d-r-1}$ and $\pi_m^{-1}(TX_{d-r})\cap I^{\bar p}_{m}X$ is a neighbourhood in $I^{\bar p}_{m}X\setminus X_{d-r-1}$.
\end{lemma}

\textsc{The inductive step}.

At this point we are ready to set up the inductive step of the construction of intersection spaces.
The inductive step is different in nature to the initial step
in the following sense. The necessary condition to be able to carry the initial step is that a link fibration admits a
fibrewise rational $q$-homology truncation.
In the inductive step, this condition is replaced by the fact that a fibration of link pairs admits a fibrewise rational $q$-homology truncation. The smaller space
in the pair is constructed by iterated modifications of $X_{d-2}=X_{d-m}$. Define
$$I^{\bar p}_{m}(X_{d-2}):=X_{d-2}.$$

We assume by induction that, for $k\geq m$, we have constructed
\begin{enumerate}
 \item [(i)] a space $X'_{k}$ and a projection
$$\pi_{k}:X'_{k}\to X$$
which is a homotopy equivalence with contractible fibres, together with a section $\alpha_{k}$ providing a closed inclusion of $X$ into $X'_k$.
 \item [(ii)] subspaces $I^{\bar p}_{k}(X_{d-2})\subset I^{\bar p}_{k}X\subset X'_{k}$ such that, embedding $X$ into $X'_{k}$ via $\alpha_{k}$,
 we have the topological pseudomanifold
 $$X_0\subset X_1\subset...\subset X_{d-k-1}$$
 embedded into $I^{\bar p}_{k}(X_{d-2})$,
 \item [(iii)] the pairs $(X'_k,I^{\bar p}_{k}X)$, $(I^{\bar p}_{k}X, I^{\bar p}_{k}(X_{d-2}))$ have respective conical
 structures with respect to the stratified subspace described in the previous point.
 The open neighbourhoods of $X_{d-r}\setminus X_{d-r-1}$ appearing in these structures are $\pi_k^{-1}(TX_{d-r})$ in $X'_k\setminus X_{d-k-1}$ and
 $\pi_k^{-1}(TX_{d-r})\cap I^{\bar p}_{k}X$ in $I^{\bar p}_{k}X\setminus X_{d-k-1}$ respectively.
\end{enumerate}

If $X_{d-k-1}\setminus X_{d-k-2}$ is empty we define $X'_{k+1}:=X'_{k}$, $\pi_{k+1}:=\pi'_{k}$, $\alpha_{k+1}:=\alpha_{k}$, $I^{\bar p}_{k+1} X := I^{\bar p}_{k} X$,
and $I^{\bar p}_{k+1} (X_{d-2}) := I^{\bar p}_{k} (X_{d-2})$. It is clear that the required conditions are satisfied.

If $X_{d-k-1}\setminus X_{d-k-2}$ is not empty, since the pair $(I^{\bar p}_{k}X,I^{\bar p}_{k}(X_{d-2}))$
has a conical structure with respect to
the stratified subspace
$$X_0\subset X_1\subset...\subset X_{d-k-1},$$
we have a locally trivial fibration of pairs
$$\sigma^\partial_{d-k-1}:(\partial\overline{\pi_k^{-1}(TX_{d-k-1})}\setminus X_{d-k-2})\cap (I^{\bar p}_{k}X,I^{\bar p}_{k}(X_{d-2}))\to X_{d-k-1}\setminus X_{d-k-2}.$$

If the fibration does not admit
a fibrewise rational $\bar q(k+1)$-homology truncation, then the intersection space construction cannot be completed with the previous choices.

Otherwise we choose a
fibrewise rational $\bar q(k+1)$-homology truncation compatible with the conical structure (see Proposition \ref{prop:fibcom})

$$\xymatrix{((\partial\overline{\pi_k^{-1}(TX_{d-k-1})}\setminus X_{d-k-2})\cap (I^{\bar p}_{k}X,I^{\bar p}_{k}(X_{d-2}))_{\leq\bar q(k+1)}\ar[rd]_{(\sigma^\partial_{d-k-1})_{\leq \bar q(k+1)}}\ar[dd]^{\phi_{d-k-1}^\partial}& \\& X_{d-k-1}\setminus X_{d-k-2}\\ (\partial\overline{\pi_k^{-1}(TX_{d-k-1})}\setminus X_{d-k-2})\cap (I^{\bar p}_{k}X, I^{\bar p}_{k}(X_{d-2})) \ar[ru]^{\sigma^\partial_{d-k-1}} &}$$

We construct now a homotopy equivalence $\pi_{k+1}:X'_{k+1}\to X$ with contractible fibres and a pair of subspaces
$(I^{\bar p}_{k+1} X, I^{\bar p}_{k+1} (X_{d-2}))\hookrightarrow X'_{k+1}$ as follows.

Let
$$(\sigma_{d-k-2})_{\leq \bar q(k+1)}:cyl((\sigma^\partial_{d-k-1})_{\leq \bar q(k+1)})\to X_{d-k-1}\setminus X_{d-k-2}$$
be the cone of the fibration of pairs $(\sigma^\partial_{d-k-1})_{\leq \bar q(k+1)}$ over $X_{d-k-1}\setminus X_{d-k-2}$.
Recall that, according with Definition~\ref{def:fibrecone},
$cyl((\sigma^\partial_{d-k-1})_{\leq \bar q(k+1)})$ is a pair of spaces.

By property (2) of Definition~\ref{def:conical} there exits a morphism of fibre bundles of pairs
$$\phi_{d-k-1}:cyl((\sigma^\partial_{d-k-1})_{\leq \bar q(k+1)})\to  (\overline{\pi_k^{-1}(TX_{d-k-1})}\setminus X_{d-k-2})\cap (I^{\bar p}_{k}X,I^{\bar p}_{k}(X_{d-2}))$$
over the base $X_{d-k-1}\setminus X_{d-k-2}$ which preserves the vertex sections.

Let
$$\theta_{d-k-1}:cyl((\sigma^\partial_{d-k-1})_{\leq \bar q(k+1)})\to X'_k\setminus X_{d-k-2}$$
be the composition of the fibre bundle morphism $\phi_{d-k-1}$ with the natural inclusion
$$ (\overline{\pi_k^{-1}(TX_{d-k-1})}\setminus X_{d-k-2})\cap (I^{\bar p}_{k}X,I^{\bar p}_{k}(X_{d-2})) \hookrightarrow X'_k \setminus X_{d-k-2}.$$
Let $\cyl(\theta_{d-k-1})$ be the mapping cylinder of $\theta_{d-k-1}$.
It is by definition the union of the pair $cyl((\sigma^\partial_{d-k-1})_{\leq \bar q(k+1)})\times[0,1]$ with the pair
$(X'_k\setminus X_{d-k-2}, \im (\phi_{d-k-1})_2)$
with the usual equivalence relation (where $(\phi_{d-k-1})_2$ denotes the second component of the fibre bundle of pairs $\phi_{d-k-1}$).

Denote by $$s_{d-k-1}:X_{d-k-1}\setminus X_{d-k-2}\to cyl((\sigma^\partial_{d-k-1})_{\leq \bar q(k+1)})$$
the vertex section. We define $Z_{k+1}$ to be the pair of spaces which results of quotienting $\cyl(\theta_{d-k-1})$
by the equivalence relation which identifies, for any
$x\in X_{d-k-1}\setminus X_{d-k-2}$, the subspace $s_{d-k-1}(x)\times [0,1]$ to a point.

We denote the spaces forming the pair $Z_{k+1}$ by $Z_{k+1}=(Z_{k+1}^1,Z_{k+1}^2)$.
We have a natural projection map $\rho_{k+1}:Z_{k+1}^1\to X'_k\setminus X_{d-k-2}$
which is a homotopy equivalence whose fibres are contractible, and has a natural section denoted by $\beta_{k+1}$. The composition
$$\pi_{k+1}:=\pi_k|_{X'_k\setminus X_{d-k-2}}\circ \rho_{k+1}:Z^1_{k+1}\to X\setminus X_{d-k-2}$$
is a homotopy equivalence with contractible fibres, and has a section $\alpha_{k+1}:=\beta_{k+1}\circ\alpha_k|_{X\setminus X_{d-k-2}}$ providing a
closed inclusion of $X\setminus X_{d-k-2}$ into $Z^1_{k+1}$.

Define $X'_{k+1}$ as the set $Z_{k+1}^1\cup X_{d-k-2}$. The projection maps $\rho_{k+1}$ and $\pi_{k+1}$ extends to projections
\begin{equation}
\label{eq:rho}
\rho_{k+1}:X'_{k+1}\to X'_k
\end{equation}
\begin{equation}
\label{eq:pi}
\pi_{k+1}:X'_{k+1}\to X.
\end{equation}

Consider the topology in $X'_{k+1}$ spanned by the all the open subsets of $Z_{k+1}$
and the collection of subsets of the form
$\pi_{k+1}^{-1}(U)$ for any open subset $U$ of $X$. With this topology the projections are also homotopy equivalences whose fibres are contractible,
and such that the natural sections $\beta_{k+1}$ and $\alpha_{k+1}$ extend to them as closed inclusions.

Define the step $k+1$ intersection space pair
to be the pair of subspaces of $X'_{k+1}$ given by
$$(I^{\bar p}_{k+1}X,I^{\bar p}_{k+1}(X_{d-2})):=cyl((\sigma^\partial_{d-k-1})_{\leq \bar q(k)})\times\{0\}\cup cyl(\phi_{d-k-1}^\partial) \cup
(I^{\bar p}_{k}X,I^{\bar p}_{k}(X_{d-2})) \setminus \pi_k^{-1}(TX_{d-k-1}),$$
with the restricted topology.

\begin{remark}
Note that we have the equality $ cyl(\phi_{d-k-1}^\partial)=\pi_{k+1}^{-1}(\partial\overline{TX_{d-k-1}}\setminus X_{d-k-2})$.
\end{remark}

With the definitions above, and using that the homology truncation is compatible with the conical structure, it is easy to check that conditions
(i)-(iii) are satisfied replacing $k$ by $k+1$ and the induction step is complete.

\begin{definition}
\label{def:intersection space}
Given a topological pseudomanifold $X_d\supset...\supset X_0$ such that the pair $(X_d, X_{d-2})$ has a conical structure with respect to the stratification
(see Definition \ref{def:comptriv} and Remark \ref{re:fixconical}), we say that it has an intersection space pair
if there exist successive choices of suitable fibrewise homology truncations
so that the construction above can be carried up to $k=d$. In that case the pair
$$(I^{\bar p}X,I^{\bar p}(X_{d-2}))=(I^{\bar p}_{d}X,I^{\bar p}_{d}(X_{d-2}))$$
is called {\it an intersection space pair associated with the stratification}.
\end{definition}

\begin{definition}
We denote $X':=X'_d$. The \emph{homotopy model} of $X$ is the homotopy equivalence $\pi_d$ which we denote $\pi:X'\to X$. The section $\alpha_d$ is denoted by
$\alpha:X\to X'$ and provides a closed inclusion of $X$ into $X'$.
\end{definition}

\begin{remark}
\label{rem:nonuniquess}
If the intersection space pair exists it does not have to be unique up to homotopy.
The different choices of fibrewise homology truncations may yield different choices of intersection spaces.
The construction of intersection spaces follow the scheme of obstruction theory in algebraic topology:
previous choices of fibrewise homology truncation may affect the possibility of finishing the
construction in the subsequent steps.
\end{remark}

\subsection{Intersection spaces for pseudomanifolds having trivial conical structures}
\label{subsec:constr_trivial}
Let $X$ be a topological pseudomanifold with a trivial conical structure (see Definition \ref{def:strongtriv}). Fix a compatible system of
trivializations (see Definitions~\ref{def:systemtr} and \ref{def:systriv}). We carry the inductive construction of the intersection space pair as above,
but we add the following property to the properties (i)-(iii) which are checked along the induction:
\begin{enumerate}
 \item [(iv)] the conical structures of the pairs $(X'_k,I^{\bar p}_{k}X)$, $(I^{\bar p}_{k}X, I^{\bar p}_{k}(X_{d-2}))$ with respect to
 $$X_0\subset X_1\subset...\subset X_{d-k-1}$$
 are trivial, and a compatible system of trivializations is inherited from the inductive construction.
 \end{enumerate}

 At the initial step of the construction we have a topological pseudomanifold
 $$X\supset X_{d-m}\supset...\supset X_0$$
 with a trivial conical structure and a compatible set of trivializations (as before the codimension of the first non-open stratum is $m$).

The compatible system of trivializations gives us a fixed trivialization of the fibration
$$\sigma^\partial_{d-m}:\partial\overline{TX_{d-m}}\setminus X_{d-m-1}\to X_{d-m}\setminus X_{d-m-1}.$$
Choose a rational $\bar q(m)$-homology truncation of the fibre. This is always possible and elementary. Now, using the trivialization, the rational
$\bar q(m)$-homology truncation of the fibre propagates to a fibrewise $\bar q(m)$-homology truncation of the fibration above. This is the truncation
chosen at the initial step.

Now, using the compatibility of our system of trivializations, it is easy to show that the pairs
$(X'_k,I^{\bar p}_{m}X)$, $(I^{\bar p}_{m}X, I^{\bar p}_{m}(X_{d-2}))$ satisfy the required properties (i)-(iv). The compatible systems of
trivializations required in property (iv) are inherited, by construction, by the compatible system of trivializations used at the beggining.

The inductive step of the construction is carried in the same way: the fixed trivializations propagate rational homology truncations of the corresponding
fibrations of pairs of links.

We have proven:

\begin{theorem}
\label{th:existTr}
If $X$ is a topological pseudomanifold with a trivial conical structure (see Definition \ref{def:strongtriv},) then there exist an intersection space pair
associated with it for every perversity.
\end{theorem}

\begin{corollary}\label{rem:existToricr}
Let $X$ be a toric variety. Then, $X$ has
an intersection space pair for every perversity.
\end{corollary}

\section{A sequence of Intersection Space pairs}

\label{sec:sequence}

Our aim is to associate with any choice of intersection space pair, a constructible complex on the original topological pseudomanifold $X$, whose
hypercohomology coincides with the hypercohomology of the intersection space $I^{\bar p} X$. In order to do so, we define an
increasing sequence of modified intersection space pairs, all of them
included in the homotopy model $X'$. We provide precise definitions of the sequence, but leave many of the straightforward checking to the reader.

\subsection{Systems of neighborhoods.}

Given a topological pseudomanifold
$$X=X_d\supset X_{d-2} \supset ... \supset X_0 \supset X_{-1}=\emptyset,$$
such that the pair $(X,X_{d-2})$ has a conical structure with respect to the stratification
$$X_{d-2} \supset ... \supset X_0 \supset X_{-1}=\emptyset$$
(see Remark~\ref{re:fixconical}), we  denote the relevant neighbourhoods by $TX_{d-r}$ for $r$ varying.

Property (2) of Definition~\ref{def:conical} states that the fibration $\sigma_{d-r}$ is the cone of the fibration $\sigma^\partial_{d-r}$
over the base $X_{d-r}\setminus X_{d-r-1}$.
This means precisely that $\overline{TX_{d-r}}\setminus X_{d-r-1}$ is equal to the product
$$\partial\overline{TX_{d-r}}\setminus X_{d-r-1}\times [0,1],$$
modulo the equivalence relation which identifies $(x,1)$ and $(y,1)$ if $\sigma^\partial_{d-r}(x)$ equals $\sigma^\partial_{d-r}(y)$.

For any $r\in 2,...,d$ and any $n\in\NN$ we define the open neighborhood $T^nX_{d-r}$ to be the quotient of
$$\partial\overline{TX_{d-r}}\setminus X_{d-r-1}\times (1-1/(n+1),1]$$
under the same equivalence relation.

The open subsets $T^nX_{d-r}$ for $n$ varying, form a system of tubular neighborhoods of $X_{d-r}\setminus X_{d-r-1}$ in $X\setminus X_{d-r-1}$, whose
intersection is the stratum $X_{d-r}\setminus X_{d-r-1}$. Moreover, for any fixed $n$ the collection of neighborhoods $T^nX_{d-r}$,
for $r$ varying, give a conical structure to
$(X,X_{d-2})$ with respect to the topological pseudomanifold $X_{d-2}\supset...\supset X_0$.

\subsection{The sequence of intersection space pairs.}
\label{sec:seq_int_space}

Suppose that there exists successive choices of suitable fibrewise homology truncations so that the construction of intersection space pairs described in the previous section can be carried up to $k=d$. Fix such a choice.

Fix $n\in \NN$. Following the inductive construction of the previous section we produce a sequence of pairs
$$(I^{\bar p,n}_{k}X,I^{\bar p,n}_{k}(X_{d-2}))$$
for $k=2,...,r$ as follows.

Let $m$ be the minimum such that $X_{d-m} \setminus X_{d-m-1} \neq \emptyset$. Define
$$K^n_m(X):=\pi_m^{-1}(\overline{TX_{d-m}}\setminus T^nX_{d-m})\subset X'_m,$$
$$I^{\bar p,n}_{m}(X):=I^{\bar p}_{m}(X)\cup K^n_m(X),$$
$$C^n_m(X_{d-2}):=\emptyset,$$
$$I^{\bar p,n}_{m}(X_{d-2}):=I^{\bar p}_{m}(X_{d-2})=X_{d-2}.$$

\begin{remark}
Note that $ \pi_m^{-1}(\overline{(TX_{d-m}}\setminus T^nX_{d-m})\setminus X_{d-m-1}) \cong cyl(\phi_{d-m}^\partial) \times [0, 1-1/(n+1)]$ (see Remark \ref{rem:description}).
\end{remark}

The following figure shows the previous modification in Figure \ref{fig:intersection_space}. $K^n_m(X)$ is the union of blue set and brown set.
\begin{figure}[H]
  \centering
  \includegraphics[width=\textwidth]{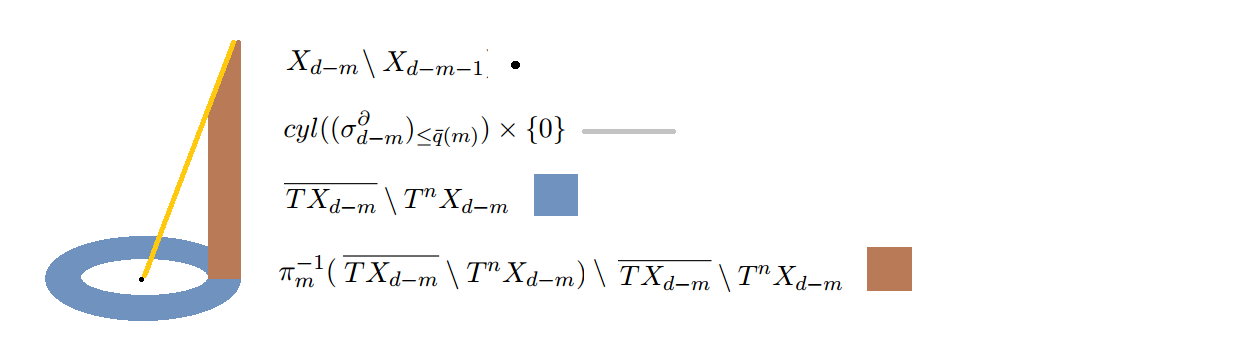}
  \caption{ }\label{fig:mod_int_space}
\end{figure}

The following figure shows $\overline{TX_{d-m}}\setminus T^nX_{d-m}$ in more dimensions than in the previous figure. $K^n_m(X)$
is the preimage of the blue set by the morphism $\pi_m$.
\begin{figure}[H]
  \centering
  \includegraphics[width=16cm]{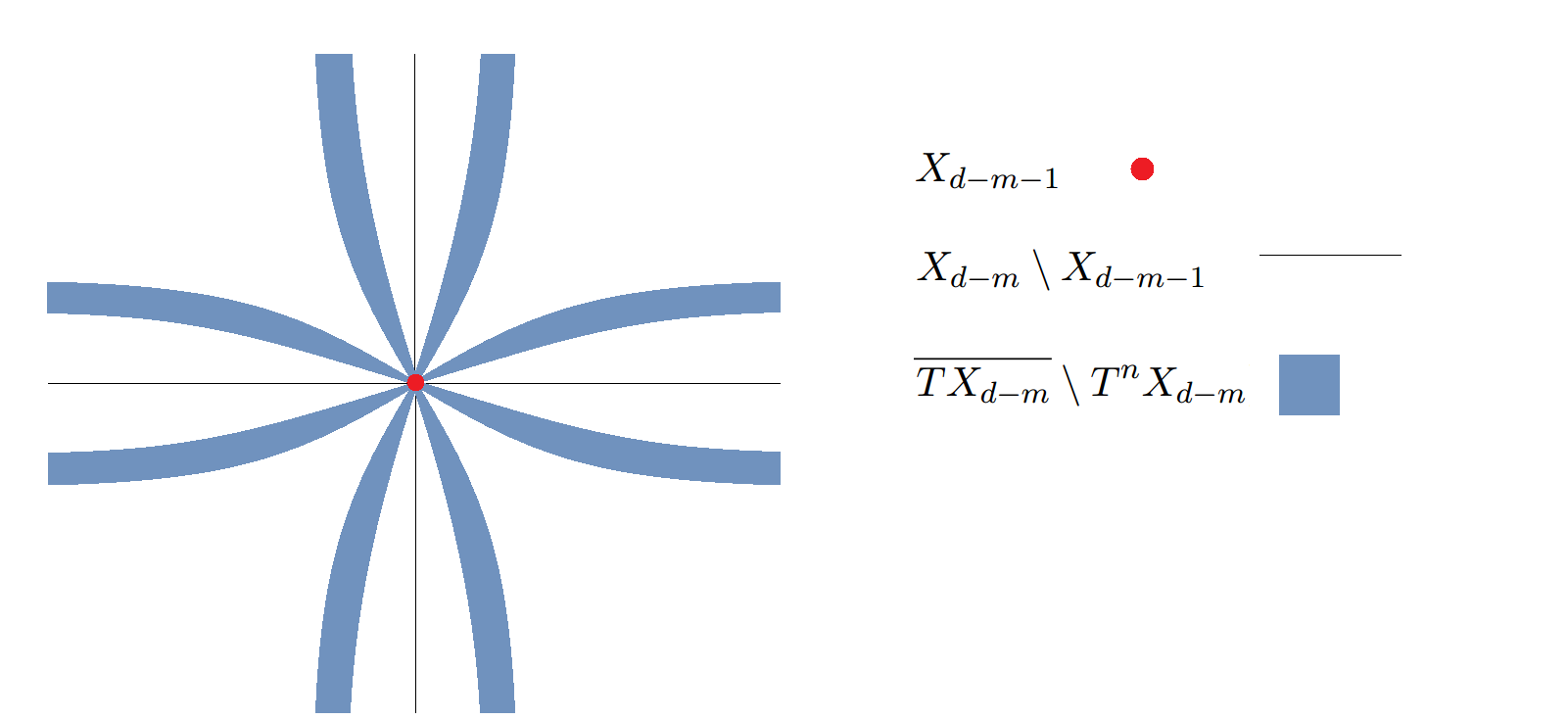}
  \caption{ }\label{fig:sequence_step1}
\end{figure}

Assume that $K^n_{k}(X),C^n_{k}(X_{d-2}), I^{\bar p,n}_{k}(X)$ and $I^{\bar p,n}_{k}(X_{d-2})$ have been defined.
Recall that $\rho_{k+1}$ is the projection defined in Equation~(\ref{eq:rho}).
Define
$$K^n_{k+1}(X):=\rho_{k+1}^{-1}\mathbf{\big(}((I^{\bar p}_{k}X\cap\pi_{k}^{-1}(\overline{TX_{d-(k+1)}}))\setminus \pi_{k}^{-1}(T^n X_{d-(k+1)}))\cup$$
$$\cup (K^n_k(X)\setminus \pi_{k}^{-1}(T^n X_{d-(k+1)}))\mathbf{\big)},$$
$$I^{\bar p,n}_{k+1}X:=I^{\bar p}_{k+1}X\cup K^n_{k+1}(X),$$
$$C^n_{k+1}(X_{d-2}):=\rho_{k+1}^{-1}\mathbf{\big(}((I^{\bar p}_{k}(X_{d-2})\cap\pi_{k}^{-1}(\overline{TX_{d-(k+1)}}))\setminus \pi_{k}^{-1}(T^nX_{d-(k+1)}))\cup$$
$$\cup (C^n_k(X_{d-2})\setminus \pi_{k}^{-1}(T^nX_{d-(k+1)}))\mathbf{\big)},$$
$$I^{\bar p,n}_{k+1}(X_{d-2}):=I^{\bar p}_{k+1}(X_{d-2})\cup C^n_{k+1}(X_{d-2}).$$

The following figure illustrates the second induction step. Recall that the codimension of the bigest non-open stratum is $m$. The figure
shows $(\overline{TX_{d-m}}\setminus T^nX_{d-m})\setminus T^n X_{d-m-1}$ in blue and green and
$(\overline{TX_{d-m-1}}\setminus T^n X_{d-m-1}) \setminus ( \overline{TX_{d-m}}\setminus T^nX_{d-m})$ in yellow. $K^n_{m+1}(X)$ is the union of
\begin{itemize}
  \item the preimage of the blue and green set by $\pi_{m+1}$
  \item the preimage of the yellow set by $\pi_{m+1}$ intersected with the preimage of $I^{\bar p}_{m}X$ by $\rho_{m+1}$
\end{itemize}
$C^n_{m+1}(X_{d-2})$ is the union of
\begin{itemize}
  \item the preimage of the blue and green set by $\pi_{m+1}$
  \item the preimage of the yellow set by $\pi_{m+1}$ intersected with the preimage of $I^{\bar p}_{m}(X_{d-2})$ by $\rho_{m+1}$
\end{itemize}

\begin{figure}[H]
  \centering
  \includegraphics[width=\textwidth]{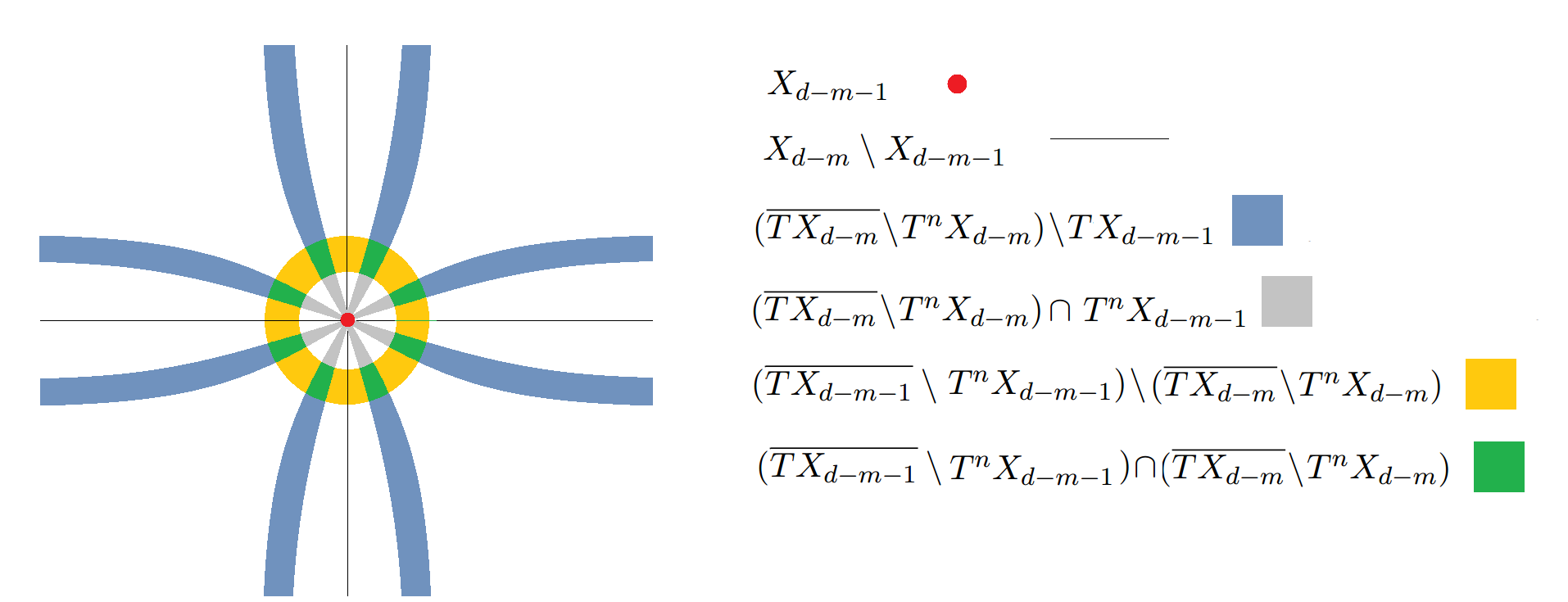}
  \caption{ }\label{fig:sequence_step2}
\end{figure}

Iterate the construction until $k=d$ and define
$$(I^{\bar p,n}X,I^{\bar p,n}(X_{d-2})):=(I^{\bar p,n}_{d}X,I^{\bar p,n}_{d}(X_{d-2})),$$
which is a pair of subsets of $X'$.

Since the closed subsets $K^n_{k}(X),C^n_{k}(X_{d-2})$ are increasingly larger when $n$ increases we have constructed
a sequence of pairs of closed subsets
$$(I^{\bar p}X,I^{\bar p}(X_{d-2}))\subset...\subset(I^{\bar p,n}X,I^{\bar p,n}(X_{d-2}))\subset (I^{\bar p,n+1}X,I^{\bar p,n+1}(X_{d-2}))\subset...$$

An easy inspection on the construction shows:

\begin{proposition}
\label{pro:tortura1}
The previous construction has the following properties.
\begin{enumerate}
 \item The inclusions $I^{\bar p,n}X\subset I^{\bar p,n+1}X$ and $I^{\bar p,n}(X_{d-2})\subset I^{\bar p,n+1}(X_{d-2})$ are strong deformation
 retracts for any $n\in\NN$. If we denote the inclusions by $\nu^n_{X_{d-2}}:I^{\bar p,n}(X_{d-2})\rightarrow I^{\bar p,n}X$,
 $i^n_X:I^{\bar p,n}X\rightarrow I^{\bar p,n+1}X$ and $i^n_{X_{d-2}}:I^{\bar p,n}(X_{d-2})\rightarrow I^{\bar p,n+1}(X_{d-2})$ and the retractions
 by $r^n_X:I^{\bar p,n+1}X\rightarrow I^{\bar p,n}X$ and $r^n_{X_{d-2}}:I^{\bar p,n+1}(X_{d-2})\rightarrow I^{\bar p,n}(X_{d-2})$,
 we have commutative diagrams
 \begin{equation}\label{diagtop}
 \xymatrix{I^{\bar p,n}(X_{d-2})\ar[r]^{\hspace{4mm}\nu^n} \ar[d]^{i_{X_{d-2}}^n} & I^{\bar p,n}X \ar[d]_{i_{X}^n}\\
 I^{\bar p,n+1}(X_{d-2})\ar[r]^{\hspace{5mm}\nu^{n+1}} \ar@/^/[u]^{r^n_{X_{d-2}}} & I^{\bar p,n+1}X \ar@/_/[u]_{r^n_{X}}}
 \end{equation}
 \item We have the equality $I^{\bar p,n}(X_{d-2})=I^{\bar p,n+1}(X_{d-2})\cap I^{\bar p,n}(X)$.
 \item\label{prop3} For any $x\in X \setminus X_{d-2}$, there exist a small contractible neighbourhood $U_x$ of $x$ in $X$ and a natural
 number $n_0$ such that, for every $n> n_0$, $\pi^{-1}(U_x)$ is contained in $I^{\bar p,n}X$ and $\pi^{-1}(U_x) \cap I^{\bar p, n}(X_{d-2})=\emptyset$.
 \item\label{prop4} For any $x\in X_{d-r}\setminus X_{d-r-1}$, there exists a small contractible neighbourhood $U_x$ of $x$ in $X$ and a natural number
 $n_0$ such that, for any $n > n_0$, the diagram (\ref{diagtop}) restricts to the diagram
 \begin{equation}\label{diagprneigh}
 \xymatrix{I^{\bar p,n}(X_{d-2})\cap \pi^{-1}(U_x)\ar[r]^{\hspace{4mm}\nu^n} \ar[d]^{i_{X_{d-2}}^n} & I^{\bar p,n}X \cap \pi^{-1}(U_x)\ar[d]_{i_{X}^n}\\
 I^{\bar p,n+1}(X_{d-2})\cap \pi^{-1}(U_x)\ar[r]^{\hspace{5mm}\nu^{n+1}} \ar@/^/[u]^{r^n_{X_{d-2}}} & I^{\bar p,+1n}X \cap \pi^{-1}(U_x)\ar@/_/[u]_{r^n_{X}}}
 \end{equation}
and we have the equalities $r^n_{X_{d-2}}( I^{\bar p,n+1}(X_{d-2})\cap \pi^{-1}(U_x)) = I^{\bar p,n}(X_{d-2})\cap \pi^{-1}(U_x)$ and
$r^n_{X}( I^{\bar p,n+1}(X)\cap \pi^{-1}(U_x)) = I^{\bar p,n}(X)\cap \pi^{-1}(U_x)$.

Then, the inclusions
 $$\pi^{-1}(U_x)\cap I^{\bar p,n}X\hookrightarrow\pi^{-1}(U_x)\cap I^{\bar p,n+1}X,$$
 $$\pi^{-1}(U_x)\cap I^{\bar p,n}(X_{n-2})\hookrightarrow\pi^{-1}(U_x)\cap I^{\bar p,n+1}(X_{n-2})$$
 are strong deformation retracts.
\end{enumerate}
\end{proposition}
\begin{proof}[Sketch of the proof]
For any $r$, we have an equality
$$(\overline{TX_{d-r}}\setminus T^n X_{d-r})\setminus X_{d-r-1}=\partial\overline{TX_{d-r}}\setminus X_{d-r-1}\times [0,1-1/(n+1)].$$

So, there are canonical retractions $\overline{TX_{d-r}}\setminus T^{n+1} X_{d-r} \rightarrow \overline{TX_{d-r}}\setminus T^n X_{d-r}$.
These retractions induce retractions $K^{n+1}_r(X) \rightarrow K^{n}_r(X)$ and $C^{n+1}_r(X_{d-2}) \rightarrow C^{n}_r(X_{d-2})$ which produce the morphisms
$r^n_X$ and $r^n_{X_{d-2}}$ respectively.

Let $x\in X \setminus X_{d-2}$. A small ball $U_x$ around of $x$ verifies property (3) if there exist a natural number $n_0$ such that
$U_x$ does not intersect $T^{n_0} X_{d-r}$ for any $r$. Moreover, this number $n_0$ exists if and only if $U_x \cap X_{d-2}$ is empty.

Let $x\in X_{d-r}\setminus X_{d-r-1}$. A small neighbourhood of $x$, $U_x$, verifying property (4) can be constructed as follows:
let $V_x$ be a ball around $x$ in the stratum  $X_{d-r}\setminus X_{d-r-1}$.  Take $V_x$ small enough so that there exist a natural number $n_0$ such
that $V_x$ does not intersect $T^{n_0} X_{d-k}$ for any $k>r$.
Consider the retraction $\sigma_{d-r}$ appearing in Definition~\ref{def:conical}, (1). Define
$$U_x:=\sigma_{d-r}^{-1}(V_x)\cap T^{n_0-1} X_{d-r}.$$

The following figure shows $U_x$ where $x\in X_{d-m}\setminus X_{d-m-1}$ in Figure \ref{fig:sequence_step2}.

\begin{figure}[H]
  \centering
  \includegraphics[width=\textwidth]{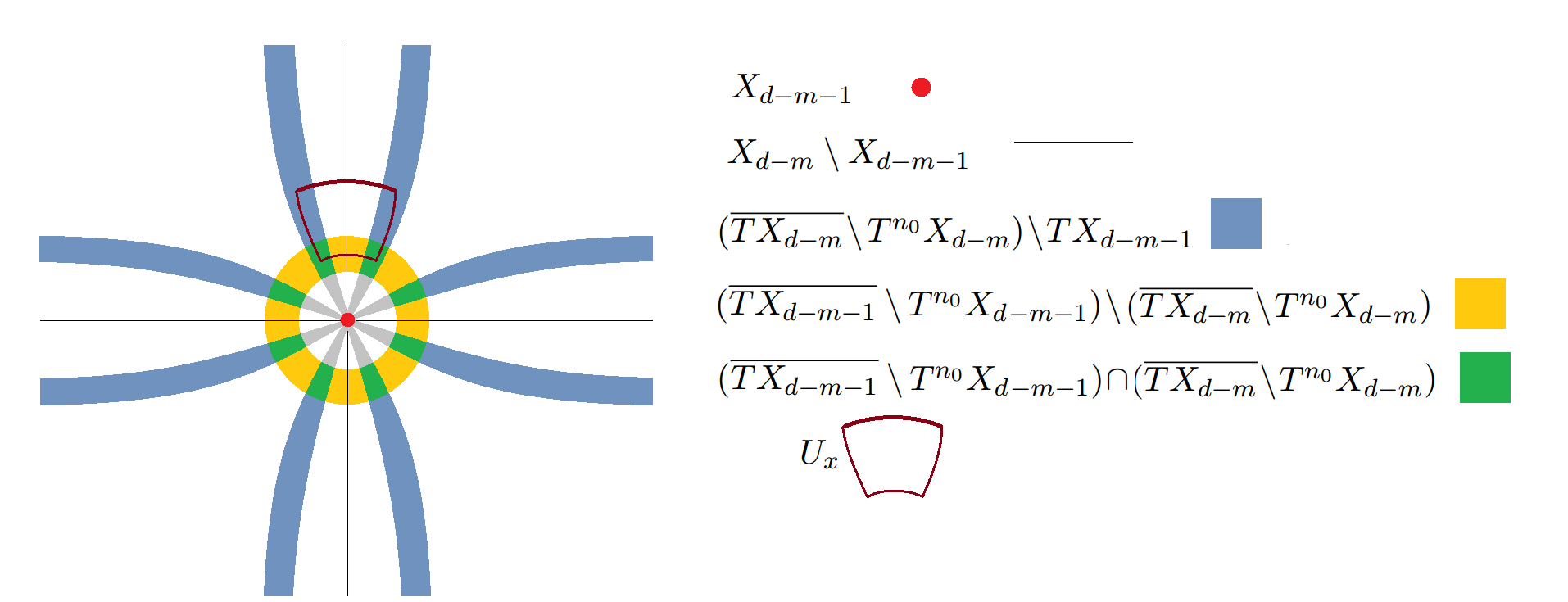}
  \caption{ }\label{fig:pr_neigh}
\end{figure}

\end{proof}

\begin{definition}\label{prneigh}
Let $x$ be any point of $X$. If $x\in X \setminus X_{d-2}$, a \emph{principal neighbourhood} of $x$ is a
small neighbourhood which verifies Property (3) of Proposition \ref{pro:tortura1}. If $x\in X_{d-2}$, a \emph{principal neighbourhood}
of $x$ is a small neighbourhood which verifies Property (4) of Proposition \ref{pro:tortura1}.
\end{definition}

\begin{definition}\label{rprneigh}
Let $x \in X_{d-r}\setminus X_{d-r-1}$. A \emph{carved principal neighbourhood} of $x$ is an open subset $U_x^*$
equal to $U_x \setminus X_{d-r}$ where $U_x$ is a principal neighbourhood of $x$.
\end{definition}

Analogously to Property (4) of Proposition \ref{pro:tortura1}, we have
\begin{proposition}
\label{pro:diagrprneigh2}
If $U_x^*$ is a carved principal neighbourhood of $x \in X_{d-r}\setminus X_{d-r-1}$,
there exist $n_0\in \NN$ such that, for every $n>n_0$, the diagram (\ref{diagtop}) restricts to the diagram
 \begin{equation}\label{diagrprneigh2}
 \xymatrix{I^{\bar p,n}(X_{d-2})\cap \pi^{-1}(U_x^*)\ar[r]^{\hspace{4mm}\nu^n} \ar[d]^{i_{X_{d-2}}^n} & I^{\bar p,n}X \cap \pi^{-1}(U_x^*)\ar[d]_{i_{X}^n}\\
 I^{\bar p,n+1}(X_{d-2})\cap \pi^{-1}(U_x^*)\ar[r]^{\hspace{5mm}\nu^{n+1}} \ar@/^/[u]^{r^n_{X_{d-2}}} & I^{\bar p,n+1}X \cap \pi^{-1}(U_x^*)\ar@/_/[u]_{r^n_{X}}}
 \end{equation}
\end{proposition}

For the next propositions recall that $\sigma_{d-r}^\partial$ is the fibration of Definition~\ref{def:conical}, (2).

\begin{proposition}
\label{cohprng}
If $U_x$ is a principal neighbourhood of $x\in X_{d-r}\setminus X_{d-r-1}$ for any $r>0$ and $n\in\NN$ big enough,
the cohomology group
$$H^i(I^{\bar p,n}X \cap \pi^{-1}(U_x),I^{\bar p,n}X_{d-2} \cap \pi^{-1}(U_x);\QQ)$$
is $0$ if $i\leq \bar q(r)$ and isomorphic to the $i$-th cohomology group of the pair
$(\sigma^\partial_{d-r})^{-1}(x)\subset (I^{\bar p}_{r-1}X,I^{\bar p}_{r-1}(X_{d-2}))$
if $i> \bar q(r)$.
\end{proposition}

\begin{proposition}
\label{cohrprng}
If $U_x^*$ is a carved principal neighbourhood of $x\in X_{d-r}\setminus X_{d-r-1}$ for any $r>0$ and $n\in\NN$ big enough,
the cohomology group
$$H^i(I^{\bar p,n}X \cap \pi^{-1}(U_x),I^{\bar p,n}X_{d-2} \cap \pi^{-1}(U_x);\QQ)$$
is isomorphic to the $i$-th cohomology group of the pair $(\sigma^\partial_{d-r})^{-1}(x) \subset (I^{\bar p}_{r-1}X,I^{\bar p}_{r-1}(X_{d-2}))$
for every $i\in \ZZ$.
\end{proposition}

\section{Sheafification}

\label{sec:sheafification}

\subsection{Sheaf of cubical singular cochains}
\label{sec:cubical}
In this section, every topological space is hereditally paracompact and locally contractible. In particular, the topological pseudomanifold and the intersection
spaces of the previous section verify these properties.

In order to produce constructible complexes whose hypercohomology compute the cohomology of intersection spaces we use sheaves of singular cohomology
cochains. For technical reasons cubical cochains, as developed by Massey in \cite[Chapters 7 and 12]{Mas}, adapt best to our construction. Here
we sketch very briefly the main points we need; the reader should check \cite{Mas} for complete definitions and proofs.

We denote by $(C_\bullet(X,\QQ),\partial)$ the complex of cubical chains of a space $X$. The group $C_i(X,\QQ)$ is defined to be the quotient
$$C_i(X,\QQ):=Q_i(X, \QQ)/D_i(X,\QQ),$$
where $Q_i(X, \QQ)$ is the vector space spanned by maps from the $i$-cube to $X$ and $D_i(X,\QQ)$ is the subspace of degeneratemaps (maps which are
constant in one direction of the cube). The differential $\partial$ is defined in the usual way. The functor given by the homology of the complex
$(C_\bullet(\centerdot,\QQ),\partial)$ defines a homology theory with coefficients in $\QQ$.

Let $(C^\bullet(X,\QQ),\delta)$ be the complex of cubical cochains of $X$. It is by definition the dual of $(C_\bullet(X,\QQ),\partial)$, and hence
$C^i(X,\QQ)$ is the subspace of $Hom(Q_i(X, \QQ),\QQ)$ formed by elements vanishing at $D_i(X,\QQ)$. The functor given by the cohomology of the complex
$(C^\bullet(\centerdot,\QQ),\partial)$ defines a cohomology theory with coefficients in $\QQ$.

Let $f:X\to Y$ be a continuous map. We denote by
$$f_{\# \, i}:C_i(X,\QQ) \rightarrow C_i (Y,\QQ),$$
$$f^{\# \, i}:C^i(Y,\QQ) \rightarrow C^i (X,\QQ)$$
the associated transformations of complexex of cubical chains and cochains. They form morphisms of complexes
$$f_{\# }:(C_\bullet(X,\QQ),\partial) \rightarrow (C_\bullet (Y,\QQ),\partial),$$
$$f^{\# }:(C^\bullet(Y,\QQ),\delta) \rightarrow (C^\bullet (X,\QQ),\delta).$$

Let $f,g:X\rightarrow X$ two continuous maps. If $f$ and $g$ are homotopic, then $f_{\#}$ and $g_{\#}$ are homotopic morphism of complexes,
and the same
happens for $f^{\#}$ and $g^{\#}$. We need for later use an explicit form of a homotopy of complexes between $f^{\#}$ and $g^{\#}$.
Let
$$\rho: C_i (X,\QQ) \rightarrow C_{i+1} (I \times X,\QQ)$$
be the morphism such that,
if $\sigma_i$ is a singular $i$-cube in $X$, $\rho (\sigma_i)=Id_I \times \sigma_i$ (the homomorphism $\rho$ takes degenerate cubical chains to
degenerate cubical chains). Let $H:I \times X \rightarrow Y$ be an homotopy between $f$ and $g$, that is, $H_0 =f$ and $H_1=g$.
A homotopy between the morphism of complexes $f_\#$ and $g_\#$ is given by $H_\# \circ \rho$.
The dual morphism of $H_\# \circ \rho$ is an homotopy between $f^\#$ and $g^\#$.

\begin{lemma}\label{homdeg}
Let $h:Z\rightarrow X$ a continuous map. If $(H_t)_{|\im(h)}$ is independent of $t\in I$, then
for every $\sigma_i \in Q_i(Z,\QQ)$, $H_\# \circ \rho \circ h_\#(\sigma_i)$ is degenerate.
\end{lemma}

Now we produce a sheafification of cubical cochains. This is an adaptation of the sheafification of singular chains appearing in~\cite{Ram}.

\begin{definition}
For every $i\in \ZZ_{\geq 0}$, let $C^i$ be the presheaf of vector spaces
$$\xymatrix{U \ar@{~>}[r] & C^i(U,\QQ)}$$
where the restriction morphisms are the obvious ones.

The {\it sheaf of cubical singular $i$-cochains} $\mathcal{C}_X^i$ is defined to be the sheafification of $C^i$.
\end{definition}

For every $i\in \ZZ_{\geq 0}$, let $C^i_\circ(X)$ be the vector subspace of $C^i(X)$ given by the set of cochains $\xi^i \in C^i(X)$ such that
there exist an open covering $\{U_j\}_{j\in J}$ of $X$ such that $\xi^i_{|U_j}=0$ for every $j\in J$. As in~\cite{Ram} one shows that the
sheafification is defined by
$$\mathcal{C}_X^i (U)= C^i(U)/ C^i_\circ(U).$$

At the level of sheaves we have functoriality as well. Let $f:X \rightarrow Y$ be a continuous map.
Then, $f$ induces a morphism of complexes of sheaves on $Y$
$$f^\#:\mathcal{C}_Y^\bullet\rightarrow f_*\mathcal{C}_X^\bullet.$$

As one expects, if $X$ is contractible then
$$H^i (\mathcal{C}^i (X)) \cong  \left\{ \begin{array}{ccc} \QQ & \text{if} & i=0 \\ 0 & \text{if} & i\neq 0 \end{array} \right.$$

This implies that the complex of sheaves $\mathcal{C}_X^\bullet$ is a resolution of the constant sheaf $\QQ_X$.

Moreover, for every $i\in \ZZ_{\geq 0}$, the sheaf $\mathcal{C}^i$ is flabby. Indeed,
it is enough to prove the restriction morphisms of the presheaf, $C^i(X,\QQ) \rightarrow C^i(U,\QQ)$, are surjective for every open subset $U\subset X$.
Given $\xi \in C^i(U,\QQ)$, let $\xi_X \in C^i(X,\QQ)$ be the linear morphism $C_i(X,\QQ) \rightarrow \QQ$ such that, for every singular
$i$-cube $\sigma$ in $X$, we have
$$\xi_X (\sigma)=\left\{ \begin{array}{ccc} \xi(\sigma) & \text{if} & \im (\sigma) \subset U \\ 0 & \text{if} & \im (\sigma) \not \subset U \end{array}\right.$$
Then, $(\xi_X)_{|U}=\xi$.

\begin{corollary}
For every $i\in \ZZ_{\geq 0}$, the $i$-th cohomology group $H^i(X, \QQ)$ is isomorphic to $i$-th group of cohomology of the complex
$\mathcal{C}_X^\bullet(X)$.
\end{corollary}

\subsection{The intersection space constructible complex}
\label{sec:ISCC}

Let $X$ be a topological pseudomanifold with stratification:
$$ X=X_d\supset X_{d-2} \supset ... \supset X_0 \supset X_{-1}=\emptyset $$
and a conical structure given by the stratification
such that there exist a set of choices so that the inductive construction of the intersection space of $X$ is not obstructed.
Let $X'$ be the homotopy model of $X$ and $\pi:X'\to X$ the homotopy equivalence.
Let $(I^{\bar p, \, n} X, I^{\bar p, \, n} (X_{d-2}))$ with $n\in\NN$ be the associated sequence of intersection space pairs and
$$j^n:I^{\bar p, \, n} X \rightarrow X'$$
$$\mu^n:I^{\bar p, \, n}(X_{d-2}) \rightarrow X'$$
the canonical inclusions.

In order to lighten the formulas appearing in this section we denote by  $\mathcal{C}_X^{n, \, \bullet}$ and $\mathcal{C}_{X_{d-2}}^{n, \, \bullet}$
the complex of sheaves of cubical singular chains in
$I^{\bar p, \, n} X$ and $I^{\bar p, \, n}(X_{d-2})$ respectively.

\begin{proposition}
For every $n\in \NN$, there exist a commutative diagram
\begin{equation}\label{morsh}
\xymatrix{j^{n+1}_* \mathcal{C}_X^{n+1, \, \bullet}  \ar[r]^{\nu^{n+1 \#}} \ar[d]^{i^{n \#}_X} &
\mu^{n+1}_* \mathcal{C}_{X_{d-2}}^{n+1, \, \bullet} \ar[d]^{i^{n \#}_{X_{d-2}}} \\
j^{n}_* \mathcal{C}_X^{n, \, \bullet} \ar[r]^{\nu^{n \#}} & \mu^{n}_* \mathcal{C}_{X_{d-2}}^{n, \, \bullet}}
\end{equation}
where all the morphisms are surjective.
\end{proposition}
\begin{proof}
For every open subset $U \subset X'$, the inclusions of Diagram~(\ref{diagtop}) restrict to a diagram
$$
\xymatrix{I^{\bar p,n}(X_{d-2})\cap U\ar[r]^{\hspace{4mm}\nu^n} \ar[d]^{i_{X_{d-2}}^n} & I^{\bar p,n}X \cap U\ar[d]_{i_{X}^n}\\
I^{\bar p,n+1}(X_{d-2})\cap U\ar[r]^{\hspace{5mm}\nu^{n+1}}  & I^{\bar p,n+1}X \cap U}
$$

So, we have the following diagram between the cubical cochain groups:
$$
\xymatrix{C^i(I^{\bar p,n+1}X \cap U, \QQ)\ar[r]^{\nu^{n+1 \#}} \ar[d]^{i^{n \#}_X} & C^i(I^{\bar p,n+1}(X_{d-2})\cap U, \QQ)\ar[d]^{i^{n \#}_{X_{d-2}}}\\
C^i(I^{\bar p,n}X \cap U, \QQ) \ar[r]^{\nu^{n \#}}  & C^i(I^{\bar p,n}(X_{d-2})\cap U, \QQ)}
$$

The morphisms of this diagram induce the morphisms of the proposition.

Moreover, these morphisms are surjective since every inclusion of topological spaces induces a surjection between the corresponding
cubical cochain groups.
\end{proof}

Denote by $\mathcal{K}^{n,\bullet}$ the kernel of $\nu^{n \#}$.
There is a canonical morphism
$$i^{n \#}:\mathcal{K}^{n+1,\bullet} \to \mathcal{K}^{n,\bullet}.$$

\begin{remark}
For every $i\in \ZZ_{\geq 0}$ and every $n\in \NN$, the $i$-th  rational cohomology group of the pair
$(I^{\bar p, \, n} X,I^{\bar p, \, n} (X_{d-2}))$ is isomorphic to $i$-th cohomology group of the complex $\mathcal{K}^{n,\bullet}$.
\end{remark}

\begin{definition}
Given a pair of natural numbers $n_1<n_2$, we will define
$$i^{n_1,n_2}:= i^{n_1 \#} \circ ... \circ i^{n_2-1 \#}: \mathcal{K}^{n_2,\bullet}\rightarrow \mathcal{K}^{n_1,\bullet}$$
\end{definition}

Then, the complexes of sheaves $\{\mathcal{K}^{n, \,\bullet}\}_{n\in \NN}$ and the morphisms $i^{n_1,n_2}$ form an inverse system and
we can
consider the inverse limit
$$\varprojlim_{n \in \NN} \mathcal{K}^{n, \,\bullet},$$
which is a complex of sheaves.

\begin{lemma}\label{qis}
$(\pi_* \varprojlim_{n \in \NN}\mathcal{K}^{n,\bullet})_{X\setminus X_{d-2}}$ is quasi-isomorphic to $\QQ_{X\setminus X_{d-2}}$.
\end{lemma}
\begin{proof}
Let $x \in X \setminus X_{d-2}$. We have the following equalities:

$$(\pi_* \varprojlim_{n \in \NN} \mathcal{K}^{n, \, \bullet})_x =
\varinjlim_{x \in U \text{open}} \pi_*(\varprojlim_{n \in \NN}  \mathcal{K}^{n, \, \bullet})(U)
= \varinjlim_{x \in U \text{open}} (\varprojlim_{n \in \NN} \mathcal{K}^{n, \, \bullet})(\pi^{-1}(U))=$$
$$=\varinjlim_{x \in U \text{open}} \varprojlim_{n \in \NN} (\mathcal{K}^{n, \, \bullet}(\pi^{-1}(U))).$$

Let $U_x$ be a principal neighbourhood of $x$ (see definition \ref{prneigh}). There exist $n_0\in \NN$ such that,
for every open subset $U\subset U_x$ and for every $n>n_0$, we have
$$\pi^{-1}(U) \cap I^{\bar p, n}X=\pi^{-1}(U)$$
and
$$\pi^{-1}(U) \cap I^{\bar p, n}(X_{d-2})=\emptyset$$

Consequently,
$\mathcal K^{n, \, \bullet}(\pi^{-1}(U))=j^n_*\mathcal C_X^{n, \, \bullet}(\pi^{-1}(U))=\mathcal C_{X'}^{\bullet}(\pi^{-1}(U))$
where $\mathcal C_{X'}$ is the sheaf of singular $i$-cochains in $X'$.

Moreover, for every $n>n_0$, $i^{n \#}=i^{n \#}_X=Id_{\mathcal C_{X'}^{\bullet}(\pi^{-1}(U))}$.
So, $\varprojlim_{n \in \NN} (\mathcal{K}^{n, \, \bullet}(\pi^{-1}(U)))=\mathcal C_{X'}^{\bullet}(\pi^{-1}(U))$.

Thus, we have shown that $(\pi_* \varprojlim_{n \in \NN}\mathcal{K}^{n, \,\bullet})_{X\setminus X_{d-2}}$ is quasi-isomorphic to
$\pi_*\QQ_{\pi^{-1}(X\setminus X_{d-2})}$, and the later sheaf is quasi-isomorphic to $\QQ_{X\setminus X_{d-2}}$ since
$\pi|_{\pi^{-1}(X\setminus X_{d-2})}$ is a homotopy equivalence.
\end{proof}

We study now the cohomology of the complex $\pi_* \varprojlim_{n \in \NN}\mathcal{K}^{n, \,\bullet}$ over each of the deeper strata of $X$
 and over the global sections.
With this purpose, we study the cohomology of $\pi_* \varprojlim_{n \in \NN}\mathcal{K}^{n, \,\bullet}$ in the principal and the carved
principal neighbourhoods (see definitions \ref{prneigh} and \ref{rprneigh}) and in the total space.

\begin{proposition}\label{propprng}
Let $U$ be equal to $X$ or a principal neighbourhood or a carved principal neighbourhood of some
$x\in X_{d-k}\setminus X_{d-(k+1)}$. Then,
$$H^i(\varprojlim_{n \in \NN} (\mathcal{K}^{n, \, \bullet}(\pi^{-1}(U))))\cong
\varprojlim_{n \in \NN} H^i(\mathcal{K}^{n, \, \bullet}(\pi^{-1}(U)))$$
\end{proposition}

Now, we need some preliminary work in order to prove Proposition \ref{propprng}.

\begin{lemma}\label{retract}
Let $U$ be equal to $X$ or a principal neighbourhood or a carved principal neighbourhood of some
$x\in X_{d-k}\setminus X_{d-(k+1)}$. Then, there exist $n_0\in \NN$ such that if $n>n_0$, there exist a morphism
$$r^{n \#}_U:\mathcal{K}^{n, \, \bullet}(\pi^{-1}(U)) \rightarrow \mathcal{K}^{n+1, \, \bullet}(\pi^{-1}(U))$$
such that $i^{n \#}(\pi^{-1}(U))\circ r^{n \#}_U =Id_{\mathcal{K}^{n, \, \bullet}(\pi^{-1}(U))}$ and
$r^{n \#}_U \circ i^{n \#}(\pi^{-1}(U))$ is homotopic to the identity.

Moreover, there exists a homotopy $h^n_U$
between $r^{n \#}_U\circ i^{n \#} (\pi^{-1}(U))$ and $Id_{\mathcal{K}^{n+1, \, \bullet}(\pi^{-1}(U))}$
such that $i^{n \#} (\pi^{-1}(U)) \circ h^n_U =0$.

\end{lemma}
\begin{proof}
Let $U$ be a principal neighbourhood and let $n_0$ be the natural number of Proposition \ref{pro:tortura1}~(4).
For every $n>n_0$, the diagram (\ref{diagprneigh}) induces a diagram
$$
\xymatrix{j^{n+1}_* \mathcal{C}_X^{n+1, \, \bullet}(\pi^{-1}(U))  \ar[r]^{\nu^{n+1 \#}} \ar[d]^{i^{n \#}_X} &
\mu^{n+1}_* \mathcal{C}_{X_{d-2}}^{n+1, \, \bullet}(\pi^{-1}(U)) \ar[d]_{i^{n \#}_{X_{d-2}}}
\\ j^{n}_* \mathcal{C}_X^{n, \, \bullet}(\pi^{-1}(U)) \ar[r]^{\nu^{n \#}} \ar@/^/[u]^{r^{n \#}_{X}}
& \mu^{n}_* \mathcal{C}_{X_{d-2}}^{n, \, \bullet}(\pi^{-1}(U)) \ar@/_/[u]_{r^{n \#}_{X_{d-2}}}}
$$
where $i_{X_{d-2}}^{n \#} \circ r^{n \#}_{X_{d-2}}= Id_{ \mu^{n}_* \mathcal{C}_{X_{d-2}}^{n, \, \bullet}(\pi^{-1}(U))}$,
$i_{X}^{n \#} \circ r^{n \#}_{X}= Id_{ j^{n}_* \mathcal{C}_{X}^{n, \, \bullet}(\pi^{-1}(U))}$ and
$r^{n \#}_{X_{d-2}}\circ i_{X_{d-2}}^{n \#}$ and $r^{n \#}_{X}\circ i_{X}^{n \#}$ are homotopic to the identity.

Then, we obtain a canonical morphism
$$r^{n \#}_U:\mathcal{K}^{n, \, \bullet}(\pi^{-1}(U)) \rightarrow \mathcal{K}^{n+1, \, \bullet}(\pi^{-1}(U))$$
 such that $i^{n \#}(\pi^{-1}(U))\circ r^{n \#}_U =Id_{\mathcal{K}^{n, \, \bullet}(\pi^{-1}(U))}$ and
$r^{n \#}_U \circ i^{n \#}(\pi^{-1}(U))$ is homotopic to the identity.

Consider the diagram (\ref{diagprneigh}) of Proposition \ref{pro:tortura1}~(4). There exist homotopies
$$H^{n}_X:\pi^{-1}(U) \cap I^{\bar p,n+1} X \times I \rightarrow \pi^{-1}(U) \cap I^{\bar p,n+1} X$$
between $i^n_X \circ r^n_X$ and $Id_{\pi^{-1}(U) \cap I^{\bar p,n+1} X}$ and
$$H^{n}_{X_{d-2}}:\pi^{-1}(U) \cap I^{\bar p,n+1} (X_{d-2}) \times I \rightarrow \pi^{-1}(U) \cap I^{\bar p,n+1} (X_{d-2})$$
between $i^n_{X_{d-2}} \circ r^n_{X_{d-2}}$ and $Id_{\pi^{-1}(U) \cap I^{\bar p,n+1} (X_{d-2})}$.
Moreover, we can suppose that, for every $t\in I$, the restrictions of $H^n_{X t}$ and $H^n_{X_{d-2} t}$
to $\pi^{-1}(U) \cap I^{\bar p,n} X$ and $\pi^{-1}(U) \cap I^{\bar p,n} (X_{d-2})$
are the identity respectively.

Following the procedure explained in Section~\ref{sec:cubical} the mapping $H^n_X$ induces a homotopy between $i^{n \#}_X \circ r^{n \#}_X$ and
$Id_{j^{n+1}_* \mathcal{C}_{X}^{n+1, \, \bullet}(\pi^{-1}(U))}$.  Moreover, applying Lemma \ref{homdeg} we have that
$i^{n \#}_X \circ h^n_X$ is equal to $0$.

Similarly the mapping $H^n_{X_{d-2}}$ induce a homotopy $h^n_{X_{d-2}}$ between
$i^{n \#}_{X_{d-2}} \circ r^{n \#}_{X_{d-2}}$ and \\$Id_{\mu^{n+1}_* \mathcal{C}_{X_{d-2}}^{n+1, \, \bullet}(\pi^{-1}(U))}$ such that
$i^{n \#}_{X_{d-2}} \circ h^n_{X_{d-2}}$ is equal to $0$.

So, there exist a homotopy $h^n_U$ between $i^{n \#} (\pi^{-1}(U)) \circ r^{n \#}_U$ and
$Id_{\mathcal{K}^{n+1, \, \bullet}(\pi^{-1}(U))}$ such that $i^{n \#} (\pi^{-1}(U)) \circ h^n_U =0$.

If $U$ is equal to $X$ or a carved principal neighbourhood we can apply the same method using the diagram (\ref{diagtop})
of Proposition \ref{pro:tortura1} or the diagram of Proposition~\ref{pro:diagrprneigh2}, respectively.
\end{proof}

\begin{remark}
The diagram (\ref{diagtop}) of Proposition \ref{pro:tortura1} is valid for every $n\in \NN$. So, in the previous lemma, we can take
$n_0=0$ if $U$ is the total space.
\end{remark}

\begin{remark}
Note that the morphisms $r^{n \#}_U$ do not induce a morphism of complexes of sheaves since they are
not defined for every open subset.
\end{remark}

\begin{definition}
Given a pair of natural numbers $n_1<n_2$ such that $n_0<n_1$, we define

$$r_U^{n_1,n_2}:= r_U^{n_2 - 1 \#} \circ ... \circ r^{n_1 \#}_U:
\mathcal{K}^{n_1, \, \bullet}(\pi^{-1}(U)) \rightarrow \mathcal{K}^{n_2, \, \bullet}(\pi^{-1}(U))$$

\end{definition}

\begin{remark}\label{homeq}
For every $n>n_0$, the inclusions
$\pi^{-1}(U)\cap I^{\bar p, n} X \hookrightarrow \pi^{-1}(U)\cap I^{\bar p, n +1} X$ and
$\pi^{-1}(U)\cap I^{\bar p, n} X_{d-2} \hookrightarrow \pi^{-1}(U)\cap I^{\bar p, n +1} X_{d-2}$
are homotopy equivalences. So, $i^{n \#}_X(\pi^{-1}(U))$ and $i^{n \#}_{X_{d-2}}(\pi^{-1}(U))$
are quasi-isomorphisms.

Then,
$$i^{n \#}(\pi^{-1}(U)):\mathcal{K}^{n+1, \, \bullet}(\pi^{-1}(U))) \rightarrow
\mathcal{K}^{n, \, \bullet}(\pi^{-1}(U)))$$
is also a quasi-isomorphism and we have an isomorphism
$$\xymatrix@R=0.1cm{\varprojlim_{n \in \NN} H^i(\mathcal{K}^{n, \, \bullet}(\pi^{-1}(U)))\ar[r]
& H^i(\mathcal{K}^{n_0 +1, \, \bullet}((\pi^{-1}(U))))\\
\{[\xi^i_n]\}_{n \in \NN}\ar[r] & [\xi^i_{n_0+1}]}$$
\end{remark}

\begin{notation}
For every open subset $V\subset X'$, the elements of
$H^i(\varprojlim_{n \in \NN} \mathcal{K}^{n, \, \bullet}(V))$ are equivalence classes of elements
$$\{\xi^i_n\}_{n\in\NN} \in \ker (\varprojlim_{n\in \NN} \partial_n^{i+1}(V))
 \subset \varprojlim_{n \in \NN} \mathcal{K}^{n, \, i}(V)$$
which we are going to denote with $[\{\xi^i_n\}_{n\in\NN}]$.

In addition, given $n\in \NN$ and an element
$\xi^i_n \in \ker (\partial_n^{i+1}(V)) \subset \mathcal{K}^{n, \, i}(V)$, we are going to
denote its equivalence class in $H^i( \mathcal{K}^{n, \, \bullet}(V))$ with $[\xi^i_n]$.
\end{notation}

\begin{proof}[Proof of Proposition \ref{propprng}]
It is enough to prove that, if $U$ is equal to $X$ or a principal neighbourhood or a carved principal
neighbourhood of $x$, then the morphism
$$\xymatrix@R=0.1cm{\ker (\varprojlim_{n\in \NN} \partial_n^{i+1}(\pi^{-1}(U))) \ar[r]^{\alpha}
&\varprojlim_{n \in \NN} H^i(\mathcal{K}^{n , \bullet}(\pi^{-1}(U)))\\ \{\xi^i_n\}_{n \in \NN}\ar[r] &
\{[\xi^i_n]\}_{n \in \NN}}$$
factorices into a morphism
$$\xymatrix{H^i(\varprojlim_{n \in \NN} \mathcal{K}^{n, \bullet}(\pi^{-1}(U)))
\ar[r]^{\beta} & \varprojlim_{n \in \NN} H^i(\mathcal{K}^{n, \bullet}(\pi^{-1}(U)))}$$
which is an isomorphism.

First, we prove $\alpha$ factorices. Let us consider an element
$\{\xi^i_n\}_{n\in \NN} \in \im (\varprojlim_{n\in \NN} \partial_n^i(\pi^{-1}(U)))$.
Then, there exist an element
$\{\delta^{i-1}_n\}_{n\in \NN} \in \varprojlim_{n \in \NN} \mathcal{K}^{n, i-1}(\pi^{-1}(U))$
such that

$$\varprojlim_{n\in \NN} \partial_n^i(\pi^{-1}(U))(\{\delta^{i-1}_n\}_{n\in \NN})=
\{\xi^i_n\}_{n\in \NN}.$$

So, for every $n\in \NN$, $\partial^{i}_n(\pi^{-1}(U)) (\delta^{i-1}_n)=\xi^i_n$. Consequently,
$\alpha(\{\xi^i_n\}_{n\in \NN})=\{[\xi^i_n]\}_{n \in \NN}=0$ and we conclude that the morphism factorices.

Now, we prove that $\alpha$ and, consequently, $\beta$ are surjective.

Because of Remark \ref{homeq}, it is enough to prove that, for every element
$\xi \in \ker(\partial_{n_0+1}^{i+1}((\pi^{-1}(U))))$, there exist an element
$\{\xi^i_n\}_{n\in \NN} \in \ker (\varprojlim_{n\in \NN} \partial_n^{i+1}(\pi^{-1}(U)))$
such that $\xi_{n_0+1}^i= \xi$.

Let us consider
$$\xi^i_n = \left\{ \begin{array}{ccc}
(i^{n, n_0 +1}(\pi^{-1}(U)))(\xi) & \text{if} & n<n_0+1\\
\xi & \text{if} & n=n_0+1\\
r^{n_0 +1, n}_U(\xi) & \text{if} & n>n_0+1
\end{array}\right.$$

Then, for every pair of natural numbers $n_1 <n_2$, we have the equality
$$((i^{n_1,n_2})(\pi^{-1}(U)))(\xi^i_{n_2})=\xi^i_{n_1}.$$
Hence $\{\xi^i_n\}_{n\in \NN}$ belongs to $\varprojlim_{n \in \NN} \mathcal{K}^{n , i} (\pi^{-1}(U))$.

Moreover
\begin{itemize}
 \item we have the vanishing $(\partial^{i+1}_{n_0+1} (\pi^{-1}(U)))(\xi)=0$,
 \item for every pair of natural numbers $n_1 < n_2$ we have the equality
$$\partial^i_{n_1} \circ i^{n_1,n_2} = i^{n_1,n_2} \circ \partial^i_{n_2} $$ and,
\item if $n_1> n_0$, we have the equality
$$ \partial^i_{n_2}(\pi^{-1}(U)) \circ r^{n_1,n_2}_U = r^{n_1,n_2}_U \circ \partial^i_{n_1}(\pi^{-1}(U)).$$
\end{itemize}

Consequently $(\partial^{i+1}_n (\pi^{-1}(U)))(\xi^i_n)=0$ for every $n\in \NN$,
$\{\xi^i_n\}_{n\in \NN}$ belongs to $\ker (\varprojlim_{n\in \NN}$ $\partial_n^{i+1}(\pi^{-1}(U)))$
and $\alpha$ and $\beta$ are surjective.

Finally, we prove $\beta$ is injective, that is,
$\ker (\alpha)=\im (\varprojlim_{n\in \NN} \partial_n^{i}(\pi^{-1}(U)))$.

Let $\{\xi^i_n\}_{n \in \NN} \in \ker (\varprojlim_{n\in \NN} \partial_n^{i+1})(\pi^{-1}(U))$ such that
$\alpha(\{\xi^i_n\}_{n \in \NN})=\{[\xi^i_n]\}_{n \in \NN}=0$. Then,
$\xi^i_{n_0+1} \in \im(\partial^{i}_{n_0+1}(\pi^{-1}(U)))$. So, there exist
$\delta \in \mathcal{K}^{n_0 + 1, \, i-1}(\pi^{-1}(U))$ such that
$(\partial^{i}_{n_0+1}(\pi^{-1}(U)))(\delta)=\xi^i_{n_0+1}$.

For every $n\in \NN$, we define
$$\delta^{i-1}_n = \left\{ \begin{array}{ccc}
(i^{n, n_0 +1}(\pi^{-1}(U)))(\delta) & \text{if} & n<n_0+1\\
\delta & \text{if} & n=n_0+1\\
r^{n_0 +1, n}_U(\delta) & \text{if} & n>n_0+1
\end{array}\right.$$
Then, $\{\delta_n^{i-1}\}_{n\in\NN} \in \varprojlim_{n\in \NN} \mathcal{K}^{n, \, i-1}(\pi^{-1}(U))$.

Let us denote $\widetilde{\xi_n^i}:=(\partial^{i}_n (\pi^{-1}(U)))(\delta_n^{i-1})$
for every $n\in \NN$. Then we have the equality $(\varprojlim_{n\in \NN} \partial_n^{i} (\pi^{-1}(U)))(\{\delta_n^{i-1}\}_{n\in\NN})
= \{\widetilde{\xi_n^i}\}_{n\in\NN}$ and $[\{\widetilde{\xi_n^i}\}_{n\in\NN}]=0$
in $H^i(\varprojlim_{n \in \NN} \mathcal{K}^{n, \, \bullet}(\pi^{-1}(U)))$.

So, to prove $\beta$ is injective, it is enough to prove the equality
$[\{\widetilde{\xi_n^i}\}_{n\in\NN}]=[\{\xi_n^i\}_{n\in\NN}]$ in
$H^i(\varprojlim_{n \in \NN} \mathcal{K}^{n, \, \bullet}(\pi^{-1}(U)))$.

If $n<n_0+1$, we have
$$\widetilde{\xi_n^i}=\partial^{i}_n (\pi^{-1}(U))(\delta_n^{i-1})=
\partial^{i}_n (\pi^{-1}(U))(i^{n, n_0 +1}(\pi^{-1}(U)))(\delta))=$$ $$=i^{n, n_0 +1}(\pi^{-1}(U))
(\partial^{i}_{n_0 +1} (\pi^{-1}(U))(\delta))=i^{n, n_0 +1}(\pi^{-1}(U))(\xi_{n_0+1}^i)=\xi_n^i.$$

If $n=n_0+1$, we have
$$\widetilde{\xi_{n_0 +1}^i}=\partial^{i}_{n_0 +1} (\pi^{-1}(U))(\delta)=\xi_{n_0+1}^i.$$

If $n>n_0+1$, we have
$$\widetilde{\xi_n^i}=\partial^{i}_n (\pi^{-1}(U))(\delta_n^{i-1})=\partial^{i}_n (\pi^{-1}(U))
(r_U^{n_0+1, n}(\delta))=r_U^{n_0+1, n}(\partial^{i}_{n_0 +1} (\pi^{-1}(U))(\delta))=$$ $$=r_U^{n_0+1, n}(\xi^i_{n_0 +1})
=r_U^{n_0+1, n} (i^{n_0+1, n}(\pi^{-1}(U)) (\xi_n^i)).$$

For every $n>n_0+1$, let $h^n_U$ be the homotopy defined in Lemma \ref{retract}. A simple computation shows that,
for every $n>n_0+1$, we have the equality:
$$\widetilde{\xi_n^i} - \xi_n^i = \partial^{i}_{n}(\sum_{k=n_0+1}^{n-1} (r^{k,n}_U \circ
h^{k-1}_U)(\xi^i_k)+ h^{n-1}_U(\xi_n^i)).$$

Let $$\epsilon^{i-1}_n=\left\{ \begin{array}{ccc}
0 & \text{if} & n \leq n_0+1 \\
\sum_{k=n_0+2}^{n-1} (r^{k,n}_U \circ h^{k-1}_U)(\xi^i_k)+ h_U^{n-1}(\xi_{n}^i) & \text{if} & n > n_0+1
\end{array}\right.$$

Let us prove $\{\epsilon^{i-1}_n\}_{n\in \NN} \in \varprojlim_{n\in \NN} (\mathcal{K}^{n, \, i-1}(\pi^{-1}(U)))$.

Since, for every $n>n_0+1$, $i^{n \#}(\pi^{-1}(U)) \circ h^n_U=0$  and $i^{n \#}(\pi^{-1}(U)) \circ r^{n \#}_U =
Id_{\mathcal{K}^{n, \, \bullet}(\pi^{-1}(U))}$, if $n_1 > n_0 +1$, we have the equality
$$i^{n_1,n_2}(\pi^{-1}(U)) (\sum_{k=n_0+2}^{n_2-1} (r^{k,n_2}_U \circ h^{k-1})(\xi^i_k)+ h_U^{n_2-1}(\xi_{n_2}^i))
= \sum_{k=n_0+2}^{n_1-1} (r^{k,n_1}_U \circ h_U^{k-1})(\xi^i_k)+$$
$$+ h_U^{n_1-1}(\xi_{n_1}^i),$$

and if $n_1 \leq n_0 +1$, we have

$$i^{n_1,n_2}(\pi^{-1}(U))(\sum_{k=n_0+2}^{n_2-1} (r^{k,n_2}_U \circ h_U^{k-1})(\xi^i_k)+ h_U^{n_2-1}(\xi_{n_2}^i))= 0.$$

Then, $\{\epsilon^{i-1}_n\}_{n\in \NN}$ belongs to $\varprojlim_{n\in \NN} (\mathcal{K}^{n, \, i-1}(\pi^{-1}(U)))$ and we have the equality
$$\varprojlim_{n\in \NN} \partial^{i}(\pi^{-1}(U)) (\{\epsilon^{i-1}_n\}_{n\in \NN})= \{\widetilde{\xi_n^i}\}_{n\in\NN}
- \{\xi_n^i\}_{n\in\NN}.$$
Therefore, $[\{\widetilde{\xi_n^i}\}_{n\in\NN}]$ equals $[\{\xi_n^i\}_{n\in\NN}]$ in
$H^i(\varprojlim_{n\in \NN} (\mathcal{K}^{n, \, i}(\pi^{-1}(U))))$ and we conclude.
\end{proof}

\begin{definition}
\label{def:IS}
 Let us define $IS:= \pi_* \varprojlim_{n\in\NN}\mathcal{K}^{n, \, \bullet}$.
\end{definition}

\begin{theorem}
\label{th:COHOMOLOGY}
The hypercohomlogy of $IS$ is isomorphic to the cohomology of the intersection space pair.
\end{theorem}
\begin{proof}
The sheaves $\mathcal{K}^{n, \, i}$ are flabby. Let $U$ be any open subset in $X'$. Every section in $\mathcal{K}^{n, \, i}(U)$, $\xi$, is also a section of
$j^{n}_*\mathcal{C}^{n, \, i}(U)$. Then, $\xi$ is the equivalence class of a singular cubical $i$-cochain $\xi'\in C^i(U\cap I^{\bar p, n} X, \QQ)$. We can extend
$\xi'$ by $0$ to get a singular cubical $i$-cochain $\xi'_{I^{\bar p, n} X} \in C^i(I^{\bar p, n} X, \QQ)$, that is, for every singular
$i$-cube $\sigma$ in ${I^{\bar p, n} X}$, we have
$$\xi'_{I^{\bar p, n} X} (\sigma)=\left\{ \begin{array}{ccc} \xi'(\sigma) & \text{if} & \im (\sigma) \subset U \\
0 & \text{if} & \im (\sigma) \not \subset U \end{array}\right.$$
The equivalence class of $\xi'_{I^{\bar p, n} X}$ in the sheaf of singular cubical $i$-cochains is a section $\xi_{X'}$ in $j^{n}_*\mathcal{C}^{n, \, i}(X')$.
It is easy to check that $\xi_{X'}$ is contained in $\mathcal{K}^{n, \, i}(X')$ since it is a extension by $0$ of $\xi$.

Now, we prove $\varprojlim_{n\in\NN}\mathcal{K}^{n, \, i}$ is also flabby. Let $\{\xi_n\}_{n\in\NN}$ be a section of
$(\varprojlim_{n\in\NN}\mathcal{K}^{n, \, i})(U)=\varprojlim_{n\in\NN}(\mathcal{K}^{n, \, i}(U))$. Since
$\mathcal{K}^{n, \, i}(U)$ is flabby for every natural $n$, the sections $\xi_n$ extend by $0$ to sections $\xi_{n}^{X'}$ in
$\mathcal{K}^{n, \, i}(X')$. It is easy to check that $i^{n_1,n_2}(X')(\xi_{n_2}^{X'})$ is equal to $\xi_{n_1}^{X'}$ for every pair of natural numbers
$n_1 <n_2$. So, $\{\xi_{n}^{X'}\}_{n\in \NN}$ is a global section of $\varprojlim_{n\in\NN}\mathcal{K}^{n, \, i}$.

Then, all the sheaves of the complex $IS$ are flabby and, consequently, the hypercohomology of $IS$ is equal to the cohomology of the global sections
of $IS$. By Proposition \ref{propprng}, there is an isomorphism
$$H^i(\Gamma(X, IS))\cong \varprojlim_{n \in \NN} H^i(\Gamma(X',\mathcal{K}^{n \, \bullet})).$$
Since the cohomology $H^i(\Gamma(X',\mathcal{K}^{n \, \bullet}))$ is the cohomology of the intersection space pair for every $n\in \NN$, we conclude.
\end{proof}

Now we prove a set of properties of the complex $IS$, in a similar vein that those satisfied by intersection cohomology sheaves.

\begin{definition}
For $k=2,..., d$, we define $U_k:=X \setminus X_{d-k}$ and we denote the canonical inclusions with $i_k:U_k \rightarrow U_{k+1}$ and $j_k: X_{d-k} \setminus X_{d-(k+1)} \rightarrow U_{k+1}$.
\end{definition}

\begin{theorem}
\label{th:existints}
The complex of sheaves $IS$ satisfies the following properties.
\begin{enumerate}
\item\label{th1} $IS_{|U_2}$ is quasi-isomorphic to $\QQ_{U_2}$.
\item\label{th2} The cohomology sheaves $\mathcal H^i(IS)$ are $0$ if $i\notin \{0,1,...,d\}$
\item\label{th3} For $k=2,...,d$, the cohomology sheaves $\mathcal H^i(j_k^*IS_{|U_{k+1}})$ are $0$ if $i \leq \bar q(k)$.
\item For $k=2,...,d$, the usual morphisms between the  cohomology sheaves $\mathcal H^i(j_k^*IS_{|U_{k+1}}) \rightarrow \mathcal H^i(j_k^*i_{k \, *}IS_{|U_{k}})$ are isomorphisms if $i > \bar q(k)$.
\end{enumerate}
\end{theorem}

\begin{proof}
(\ref{th1}) is shown in lemma \ref{qis}.

Let $x \in X_{d-k}\setminus X_{d-(k+1)}$ for some $k\in\{2,3,...,d\}$. Given a complex of sheaves we denote by $H^i$ and $\mathcal H^i$ its
$i$-th cohomology presheaf and sheaf respectively. We have the obvious chain of equalities:

$$\mathcal H^i(\pi_* \varprojlim_{n \in \NN}\mathcal{K}^{n, \, \bullet})_x = H^i(\pi_* \varprojlim_{n \in \NN} \mathcal{K}^{n, \, \bullet})_x = \varinjlim_{x \in U \text{open}} H^i((\varprojlim_{n \in \NN} \mathcal{K}^{n, \, \bullet})(\pi^{-1}(U)))=$$ $$=\varinjlim_{x \in U \text{open}} H^i(\varprojlim_{n \in \NN} (\mathcal{K}^{n, \, \bullet}(\pi^{-1}(U)))).$$

Since the principal neighbourhoods form a system of neighborhoods for any point, we can suppose every open subset $U$ appearing in the previous formula
is a principal neighbourhood of $x$. Then, applying Proposition \ref{propprng} we have

$$\mathcal H^i(\pi_* \varprojlim_{n \in \NN} \mathcal{K}^{n, \, \bullet})_x = \varinjlim_{\substack{U \, \text{principal}\\ \text{neighbourhood of} \, x}} \varprojlim_{n \in \NN} H^i((\mathcal{K}^{n, \, \bullet}(\pi^{-1}(U)))).$$

So, applying proposition \ref{cohprng},
$\mathcal H^i(\pi_* \varprojlim_{n \in \NN} \mathcal{K}^{n, \, \bullet})_x$ is $0$ if
$i\leq \bar q(k)$ and equal to
$$H^i((\sigma^\partial_{d-r})^{-1}(x);\QQ)$$
if $i> \bar q(k)$.

Hence, we have proven (\ref{th2}) and (\ref{th3}) of the theorem.

Moreover, applying again Proposition \ref{propprng}
$$\mathcal H^i(j_k^*i_{k \, *}IS_{|U_{k}})_x=\varinjlim_{\substack{U \, \text{principal}\\ \text{neighbourhood of} \, x}}
H^i(\varprojlim_{n \in \NN} \pi_* \mathcal{K}^{n, \, \bullet}(U \setminus X_{d-k}))=$$
$$=\varinjlim_{\substack{U \, \text{principal}\\ \text{neighbourhood of} \, x}}
\varprojlim_{n \in \NN} H^i(\pi_* \mathcal{K}^{n, \, \bullet}(U \setminus X_{d-k})) $$
and, because of Proposition \ref{cohrprng},
$$\mathcal H^i(j_k^*i_{k \, *}IS_{|U_{k}})_x=H^i((\sigma^\partial_{d-r})^{-1}(x);\QQ)$$
for every $i\in \ZZ$, which concludes the proof.
\end{proof}

\section{Axioms of intersection space complexes}
\label{sec:axiomatic}

From now on, we do not need to assume that our topological pseudomanifold has a conical structure with respect to the stratification, like in
Remark \ref{re:fixconical}.

Let $X$ be a topological pseudomanifold with the following stratification:

\begin{equation}\label{str}
X=X_d\supset X_{d-2} \supset ... \supset X_0 \supset X_{-1}=\emptyset
\end{equation}

Let $U_k:=X\setminus X_{d-k}$ and let $i_k: U_k \rightarrow U_{k+1}$ and $j_k: X_{d-k} \setminus X_{d-k-1}
\rightarrow U_{k+1}$ be the usual inclusions.

Let us denote by $D^b_{cc}(X)$ the bounded derived category of cohomologically constructible sheaves of rational
vector spaces on $X$ with the previous stratification. Fix a perversity $\bar p$ and let us consider the
following sets of properties in this category:

\begin{enumerate}
  \item We say that $B^\bullet \in D^b_{cc}(X)$ verifies $[AX1]_k$ for perversity $\bar p$ if:

  \begin{itemize}
    \item[($a$)] $B^\bullet_{|U_2}$ is quasi-isomorphic to $\QQ_{U_2}$,
    \item[($b$)] the cohomology sheaf $\mathcal{H}^i(B^\bullet)$ is $0$ if $i \notin \{0,1,...,n\}$,
    \item[($c_k$)] $\mathcal{H}^i(j_k^* B^\bullet_{|U_{k+1}})$ is equal to $0$ if $i> \bar p(k)$,
    \item[($d_k$)] the natural morphism $\mathcal{H}^i(j_k^* B^\bullet_{|U_{k+1}})\rightarrow
    \mathcal{H}^i(j_k^* i_{k *} B^\bullet_{|U_{k}})$ is an isomorphism if $i\leq \bar p(k)$.
  \end{itemize}

  \item Let $\bar q$ be the complementary perversity of $\bar p$. We say that $B^\bullet \in D^b_{cc}(X)$
  verifies $[AXS1]_k$ for perversity $\bar p$ if:

  \begin{itemize}
    \item[($a$)] $B^\bullet_{|U_2}$ is quasi-isomorphic to $\QQ_{U_2}$,
    \item[($b$)] the cohomology sheaf $\mathcal{H}^i(B^\bullet)$ is $0$ if $i \notin \{0,1,...,n\}$,
    \item[($c_k$)] $\mathcal{H}^i(j_k^* B^\bullet_{|U_{k+1}})$ is equal to $0$ if $i \leq \bar q(k)$,
    \item[($d_k$)] the natural morphism $\mathcal{H}^i(j_k^* B^\bullet_{|U_{k+1}})\rightarrow
    \mathcal{H}^i(j_k^* i_{k *} B^\bullet_{|U_{k}})$ is an isomorphism if $i> \bar q(k)$.
  \end{itemize}
\end{enumerate}

\begin{remark}
$B^\bullet$ verifies $[AX1]_k$ for $k=2,..., d$ if and only if $B^\bullet[d]$ verifies the axioms [AX1] of
\cite[section 3.3]{GorMP}, that is, if $B^\bullet[d]$ is the intersection cohomology sheaf of $X$. So, we
will denote an object of $D^b_{cc}(X)$ verifying $[AX1]_k$ for $k=2,...,d$
by $IC_{\bar p}[-d]$.
\end{remark}

\begin{definition}\label{def:intspacecomplex}
An \emph{intersection space complex} of $X$ with perversity $\bar p$ and stratification (\ref{str}) is a complex of
sheaves verifying $[AXS1]_k$ for $k=2,...,d$.

We denote by $IS_{\bar p}$ a complex of sheaves in $X$ with these properties.
\end{definition}

\begin{remark}\label{re:relationintspace}
If the stratification of $X$ induces a conical structure (see Definition \ref{def:conical}) and there exist an
intersection space pair of $X$ with perversity $\bar p$ in the sense of
Definition \ref{def:intersection space}, then there exist an intersection space complex of $X$ (see Theorem
\ref{th:existints}).
\end{remark}

In the sequel, we will need equivalent axioms to $[AXS1]_k$. In the following remark, we review the method of
\cite[section 3.4]{GorMP} to get equivalent axioms to $[AX1]_k$ and $[AXS1]_k$.

\begin{remark}\label{rem:eqaxiomas}
Using the long exact sequence of cohomology associated to the distinguished triangle
$$ j_k^! B^\bullet_{|U_{k+1}}\rightarrow j_k^* B^\bullet_{|U_{k+1}} \rightarrow j_k^* i_{k *} B^\bullet_{|U_{k}}
\xrightarrow{[1]}$$
($c_k$) and ($d_k$) of $[AX1]_k$ are equivalent to ($c_k$) and($d'_k$) where ($d'_k$) is the following property:
\begin{itemize}
  \item[ ] $\mathcal{H}^i(j_k^! B^\bullet_{|U_{k+1}})=0$ if $i\leq \bar p(k)+1$.
\end{itemize}

Given $x\in X_{d-k}\setminus X_{d-k-1}$, let $u_x:\{x\} \rightarrow X_{d-k}\setminus X_{d-k-1}$ and $j_x:\{x\}
\rightarrow X$ be the canonical inclusions. Then, applying the property 1.13(15) of \cite{GorMP},

$$ j_x^! B^\bullet = u_x^! j_k^! B^\bullet_{|U_{k+1}} = u_x^* j_k^! B^\bullet_{|U_{k+1}} [k-d] $$

So, ($d'_k$) is equivalent to the following property ($d''_k$):

For every $x\in X_{d-k}\setminus X_{d-k-1}$, $\mathcal{H}^i(j_x^! B^\bullet_{|U_{k+1}})=0$ if $i\leq
\bar p(k)+1+d-k=d-\bar q(k)-1$.

Now, we apply the same method to properties $[AXS1]_k$. Using again the long exact sequence of cohomology associated to
$$ j_k^! B^\bullet_{|U_{k+1}}\rightarrow j_k^* B^\bullet_{|U_{k+1}} \rightarrow j_k^* i_{k *} B^\bullet_{|U_{k}}
\xrightarrow{[1]},$$
we deduce ($d_k$) of $[AXS1]_k$ is equivalent to ($d1_k$) and($d2_k$) where ($d1_k$) and($d2_k$) are the following properties.

\begin{itemize}
  \item[($d1_k$)] $\mathcal{H}^i(j_k^! B^\bullet_{|U_{k+1}})=0$ if $i> \bar q(k)+1$.
  \item [($d2_k$)] The canonical morphism $\mathcal{H}^{\bar q(k)+1}(j_k^! B^\bullet_{|U_{k+1}})\rightarrow
  \mathcal{H}^{\bar q(k)+1}(j_k^* B^\bullet_{|U_{k+1}})$ is the morphism $0$.
\end{itemize}

Moreover, using the property 1.13(15) of \cite{GorMP}, these properties are equivalent to:

\begin{itemize}
  \item[($d1'_k$)] For every $x\in X_{d-k}\setminus X_{d-k-1}$, $\mathcal{H}^i(j_x^! B^\bullet_{|U_{k+1}})=0$
  if $i> \bar q(k)+1+d-k= d- \bar p(k)-1$.
  \item [($d2'_k$)] For every $x\in X_{d-k}\setminus X_{d-k-1}$, the canonical morphism
  $$\mathcal{H}^{d- \bar p(k)-1}(j_x^! B^\bullet_{|U_{k+1}})\rightarrow
  \mathcal{H}^{\bar q(k)+1}(j_x^* B^\bullet_{|U_{k+1}})$$ (given by property 1.13(15) of \cite{GorMP}) is the
  morphism $0$.
\end{itemize}

\end{remark}

Now, we recall useful definitions to compare the axioms $[AX1]_k$ with $[AXS1]_k$.

\begin{definition}
Let $B^\bullet$ be a complex of sheaves in a topological space $X$ and, for every $x\in X$, let
$j_x:\{x\} \rightarrow X$ be the canonical inclusion. Then,
\begin{itemize}
  \item The \emph{local support} of $B^\bullet$ in degree $m$ is
  $$\{x\in X | \H^m(j_x^*B^\bullet) \neq 0\}$$
  \item The \emph{local cosupport} of $B^\bullet$ in degree $m$ is
  $$\{x\in X | \H^m(j_x^!B^\bullet) \neq 0\}$$
\end{itemize}
\end{definition}

The properties ($c_k$) of $[AX1]_k$, ($c_k$) of $[AXS1]_k$ and ($d'_k$), ($d1'_k$), ($d2'_k$) of Remark
\ref{rem:eqaxiomas} can be defined in terms of support and cosupport.

Let us consider a complex stratified variety $X$. Then, the upper middle perversity and the lower middle
perversity (see Definition \ref{def:perversity}) are equal over the codimension of the strata of $X$.

Let $\bar m$ be the middle perversity. The following table, taken from~\cite{CaMi}, illustrate the conditions of support and cosupport
for a complex of sheaves $IC_{\bar m}[-d]$ verifying $[AX1]_k$ with perversity $\bar m$ for $k=2,...,d$.

$$\begin{array}{cc|c|c|c|c|c}
    \multirow{10}{1.2cm}{degree}
    &8 & c & c & c & c & c \\ \cline{2-7}
    &7 &   &   & c & c & c \\ \cline{2-7}
    &6 &   &   &   & c & c \\ \cline{2-7}
    &5 &   &   &   &   & c \\ \cline{2-7}
    &4 &   &   &   &   &   \\ \cline{2-7}
    &3 &   &   &   &   & \times \\ \cline{2-7}
    &2 &   &   &   & \times & \times \\ \cline{2-7}
    &1 &   &   & \times & \times & \times \\ \cline{2-7}
    &0 & \times & \times & \times & \times & \times \\ \cline{2-7}
    &  & 0 & 1 & 2 & 3 & 4 \\
    \multicolumn{7}{r}{} \\
    \multicolumn{7}{r}{\substack{\text{complex codimension}\\ \text{of the strata}}}
  \end{array}$$

The symbol $c$ means the complex can have local cosupport at that place, while the symbol $\times$ means the
complex can have local support at that place.

The following tables illustrate the conditions of support and cosupport for an intersection space complex $IS_{\bar m}$ with perversity $\bar m$.

\begin{multicols}{2}

$$\begin{array}{cc|c|c|c|c|c}
    \multirow{10}{1.2cm}{degree}
    &8 &      &\times&\times&\times&\times\\ \cline{2-7}
    &7 &      &\times&\times&\times&\times\\ \cline{2-7}
    &6 &      &\times&\times&\times&\times\\ \cline{2-7}
    &5 &      &\times&\times&\times&\times\\ \cline{2-7}
    &4 &      &\times&\times&\times&\times^*\\ \cline{2-7}
    &3 &      &\times&\times&\times^*&      \\ \cline{2-7}
    &2 &      &\times&\times^*&      &      \\ \cline{2-7}
    &1 &      &\times^*&      &      &      \\ \cline{2-7}
    &0 &\times&      &      &      &      \\ \cline{2-7}
    &  & 0 & 1 & 2 & 3 & 4 \\
    \multicolumn{7}{r}{} \\
    \multicolumn{7}{r}{\substack{\text{complex codimension}\\ \text{of the strata}}\hspace{5mm}}
  \end{array}$$

$$\begin{array}{cc|c|c|c|c|c}
    \multirow{10}{1.2cm}{degree}
    &8 &c& & & & \\ \cline{2-7}
    &7 & &c^*& & & \\ \cline{2-7}
    &6 & &c&c^*& & \\ \cline{2-7}
    &5 & &c&c&c^*& \\ \cline{2-7}
    &4 & &c&c&c&c^*\\ \cline{2-7}
    &3 & &c&c&c&c\\ \cline{2-7}
    &2 & &c&c&c&c\\ \cline{2-7}
    &1 & &c&c&c&c\\ \cline{2-7}
    &\,0 \,& &c&c&c&c \\ \cline{2-7}
    &  &\,\, 0 \,\,&\,\, 1 \,\,&\,\, 2\, \,&\,\, 3 \,\,& \,\,4 \,\,\\
    \multicolumn{7}{r}{} \\
    \multicolumn{7}{r}{\substack{\text{complex codimension}\\ \text{of the strata}}\hspace{5mm}}
  \end{array}$$

\end{multicols}

The symbol $c$ means the complex can have local cosupport at that place, while the symbol $\times$ means the complex
can have local support at that place. Moreover, the symbol $*$ means the support and the cosupport must verify a
special condition given by ($d2'_k$).

Note that in $U_2$, $(IS_{\bar m})_{|U_2} \cong (IC_{\bar m}[-d])_{|U_2} \cong \QQ_{U_2}$. However, in $X_{d-2}$, the
place at which $IS_{\bar m}$ can have support is exactly the place at which $IC_{\bar m}[-d]$ cannot have support and
the place at which $IS_{\bar m}$ can have cosupport is exactly the place at which $IC_{\bar m}[-d]$ cannot have
cosupport.

\section{A derived category approach to intersection space complexes}
\label{sec:constructibleapproach}

In this section, we study necessary and sufficient conditions for the existence of an intersection space complex of $X$
with a perversity $\bar p$. Unlike intersection cohomology sheaves, intersection space complexes are not unique. We study the space parametrizing
the different choices of intersection space complexes for a fixed perversity.

\subsection{Homological algebra review}
We need the following lemma, that should be well known, we include a proof for convenience of the reader.

\begin{lemma}\label{lem:split}
Let
$$\xymatrix@C=1.4cm{A^\bullet \ar[r]^{f} & B^\bullet \ar[r]^{g} & C^\bullet \ar@/^{7mm}/[ll]^{[1]}_{\phi}}$$
be a distinguished triangle in the category $D^b_{cc}(X)$.

The following four conditions are equivalent:
\begin{enumerate}
  \item $f$ admits a retract
  \item $g$ admits a section
  \item There is an isomorphism in the derived category $\gamma:B^\bullet\cong A^\bullet \oplus C^\bullet$ such that
  $f=\gamma^{-1} \circ i_A$ and $g=p_C \circ \gamma$ where $i_A: A^\bullet \rightarrow A^\bullet \oplus C^\bullet$ is
  the natural inclusion and $p_C:A^\bullet \oplus C^\bullet \rightarrow C^\bullet$ is the natural projection.
  \item $\phi$ is the morphism $0$
\end{enumerate}
\end{lemma}

\begin{proof}
First we prove that (1) implies (3). Let $r$ be a retract of $f$. Then, $r\circ f=Id_{A^\bullet}$. So, the morphism
induced by $f$ between the cohomology sheaves $\H^i(f):\H^i(A^\bullet) \to \H^i(B^\bullet)$ is injective for every
$i\in \ZZ$. Hence, using the long exact sequence of cohomology
$$... \rightarrow\H^{i-1}(C^\bullet)\xrightarrow{\H^{i-1}(\phi)}\H^i(A^\bullet)\xrightarrow{\H^i(f)}\H^i(B^\bullet)
\xrightarrow{\H^i(g)} \H^i(C^\bullet) \xrightarrow{\H^{i}(\phi)} \H^{i+1}(A^\bullet)\rightarrow ... ,$$
we deduce $\H^i(\phi)=0$ for every $i\in \ZZ$. Consequently, we have short exact sequences
\begin{equation}\label{shortexseq}
0 \rightarrow \H^i(A^\bullet)\xrightarrow{\H^i(f)}\H^i(B^\bullet)\xrightarrow{\H^i(g)}\H^i(C^\bullet)\rightarrow 0
\end{equation}

Now, let us consider the morphism
$$\gamma=\left( \begin{array}{c} r \\ g \end{array}\right): B^\bullet \to A^\bullet \oplus C^\bullet$$

The morphism induced by $\gamma$ between the cohomology sheaves is
$$\H^i(\gamma)=\left( \begin{array}{c} \H^i(r) \\ \H^i(g) \end{array}\right): \H^i(B^\bullet) \to \H^i(A^\bullet)
\oplus \H^i(C^\bullet),$$
which is an isomorphism, since $\H^i(r)$ is a retract of $\H^i(f)$ and (\ref{shortexseq}) is exact.

Moreover, since $f\circ g=0$, we have $\gamma \circ f = i_A$ and it is clear that $p_C \circ \gamma=g$. So, we have proven
(1) implies (3).

Now, we prove (2) implies (3). Let $s$ be a section of $g$ and let us consider the morphism
$$\gamma'=(f,s): A^\bullet \oplus C^\bullet \rightarrow B^\bullet$$
In the same way that in the previous implication, we can show that $\gamma'$ is a quasi-isomorphism, and that we have $f=\gamma'
\circ i_A$ and $g\circ \gamma'=p_C$. So, $\gamma=(\gamma')^{-1}$ is the isomorphism which appears in condition (3).

Moreover, if condition (3) is true, $p_A:A^\bullet \oplus C^\bullet \rightarrow A^\bullet$ denotes the natural
projection and $i_C: C^\bullet \rightarrow A^\bullet \oplus C^\bullet$ denotes the natural inclusion, then
$p_A \circ \gamma$ is a retract of $f$ and $\gamma^{-1} \circ i_C$ is a section of $g$. So, (3) implies (1) and (2).

Now, it is enough to prove $(3) \Leftrightarrow (4)$. (3) implies that we have the following
isomorphism between distinguished triangles:
$$\xymatrix@C=1.4cm{A^\bullet \ar[r]^f \ar[d]_{Id_{A^\bullet}} & B^\bullet \ar[r]^g \ar[d]^\gamma & C^\bullet
\ar[d]^{Id_{C^\bullet}} \ar@/_{8mm}/[ll]_{[1]}^\phi \\ A^\bullet \ar[r]^{i_A\hspace{4mm}}& A^\bullet \oplus C^\bullet
\ar[r]^{\hspace{4mm}p_C} & C^\bullet \ar@/^{8mm}/[ll]^{[1]}_0}$$

So, $\phi$ is the morphism $0$.

Now, suppose $\phi=0$ and let us prove condition (3). By properties of the triangulated categories, we know
$$\xymatrix@C=1.4cm{C^\bullet[-1] \ar[r]^{\hspace{2mm}-\phi[-1] }& A^\bullet\ar[r]^{f} & B^\bullet
\ar@/^{7mm}/[ll]^{[1]}_{g}}$$
is a distinguished triangle. So, if $\phi=0$, there is an isomorphism $\gamma: B^\bullet \cong A^\bullet \oplus
C^\bullet$ which completes the following isomorphism of triangles
$$\xymatrix@C=1.4cm{C^\bullet[-1]\ar[d]_{Id_{C^\bullet}[-1]} \ar[r]^{\hspace{4mm}0} & A^\bullet \ar[r]^f
\ar[d]^{Id_{A^\bullet}} & B^\bullet \ar@/_{8mm}/[ll]^g_{[1]} \ar[d]^\gamma \\ C^\bullet[-1] \ar[r]^{\hspace{3mm}0} &
A^\bullet \ar[r]^{i_A\hspace{3mm}} & A^\bullet \oplus C^\bullet \ar@/^{8mm}/[ll]^{[1]}_{p_C} }$$
$\gamma$ is the isomorphism which appears in condition (3).

\end{proof}

\begin{definition}
The triangle
$$\xymatrix@C=1.4cm{A^\bullet \ar[r]^{f} & B^\bullet \ar[r]^{g} & C^\bullet \ar@/^{7mm}/[ll]^{[1]}_{\phi}}$$
is said to be split if it verifies the conditions of Lemma \ref{lem:split}.
\end{definition}

\subsection{Charaterization of existence and study of uniqueness}

Now, we can state the main theorem in this section.

\begin{theorem}\label{th:existunic}
The following holds:
\begin{itemize}

  \item[\textbf{1.}] [Goreski-McPherson] There exist one object in $D^b_{cc}(X)$ verifying $[AX1]_k$ for $k=2,..., d$. This object is
  unique up to isomorphism.

  \item[\textbf{2.}] Suppose there exist an intersection space complex in $U_r$, $IS_{r-1}$, that is, $IS_{r-1}$ is an object
  in $D^b_{cc}(U_r)$ and it verifies $[AXS1]_k$ for $k=2,...,r-1$. Then, there exist an intersection space complex in
  $U_{r+1}$, $IS_{r}$, such that $(IS_{r})_{|U_r}\cong IS_{r-1}$ if and only if the distinguised triangle:
  \begin{equation}\label{tr1}
  \tau_{\leq \bar q(r)} j_{r*} j_r^* i_{r*} IS_{r-1}\xrightarrow{f} j_{r*} j_r^* i_{r*} IS_{r-1} \rightarrow \tau_{> \bar q(r)}
  j_{r*} j_r^* i_{r*} IS_{r-1} \xrightarrow{[1]}
  \end{equation}
  is split.
  Moreover, there is a bijection
  $$
  \left\{\begin{array}{c}\text{intersection space complexes}\\ IS_{r} \in D^b_{cc}(U_{r+1}) \, \text{such}  \\
  \text{that} \, (IS_{r})_{|U_r}\cong IS_{r-1} \end{array}\right\} /
  \left\{\text{isomorphism}\right\} \longleftrightarrow \left\{\text{retracts of}\, f\right\}/\sim
  $$
  where $\sim$ is the equivalence relation such that $\lambda_1 \sim \lambda_2$ if and only if there exit isomorphisms
  $\alpha:\tau_{\leq \bar q(r)} j_{r*} j_r^* i_{r*} IS_{r-1} \rightarrow \tau_{\leq \bar q(r)} j_{r*} j_r^* i_{r*} IS_{r-1}$ and
  $\beta:i_{r*} IS_{r-1} \rightarrow i_{r*} IS_{r-1}$ such that $\lambda_2 = \alpha \circ \lambda_1 \circ j_{r*} j_r^* \beta $.

\end{itemize}
\end{theorem}

\begin{proof}
 \textbf{1.} is proved in \cite[section 3]{GorMP}, but we give proof adapted to our needs.

 There exist a unique object (up to isomorphism), $\QQ_{U_2}$, in $D^b_{cc}(U_2)$ verifying $(a)$ and $(b)$ of $[AX1]_k$.
 Suppose there exist a unique object (up to isomorphism), $IC_{\bar p \, r-1}[-d]$, in $D^b_{cc}(U_r)$ verifying $[AX1]_k$ for $k=2,..., r-1$,
 and consider the following composition of natural morphisms:
 $$\xymatrix{ i_{r*} IC_{\bar p \, r-1}[-d] \ar[r] \ar@/^1pc/[rr]^{\phi_r} & j_{r*} j_r^* i_{r*}IC_{\bar p \, r-1}[-d] \ar[r]
 & \tau_{> \bar p(r)} j_{r*} j_r^* i_{r*} IC_{\bar p \, r-1}[-d]}$$

 Let us define $IC_{\bar p \, r}[-d] := cone(\phi_r) [-1]$. Then, there is a distinguished triangle:
 \begin{equation}\label{tr2}
IC_{\bar p \, r}[-d]\rightarrow i_{r*} IC_{\bar p \, r-1}[-d] \xrightarrow{\phi_r} \tau_{> \bar p(r)} j_{r*} j_r^* i_{r*}
IC_{\bar p \, r-1}[-d] \xrightarrow{[1]}
 \end{equation}

 Using the long exact sequence of cohomology associated to this triangle we can prove that $IC_{r}[-d]$ verifies
 $[AX1]_k$ for $k=2,..., r$.

 Now, suppose there exist another object $B^\bullet$ in $D^b_{cc}(U_{r+1})$ verifying $[AX1]_k$ for $k=2,..., r$. Then,
 $B^\bullet_{|U_r}$ verifies $[AX1]_k$ for $k=2,..., r-1$. So, there exist an isomorphism $B^\bullet_{|U_r} \cong
 IC_{\bar p \, r-1}[-d]$.

 Let $\varphi: B^\bullet \rightarrow i_{r*} IC_{\bar p \, r-1}[-d]$ be the composition of the canonical morphism $B^\bullet
 \rightarrow i_{r *} B^\bullet_{U_r}$ and an isomorphism $i_{r*}B^\bullet_{|U_r} \cong i_{r*}IC_{\bar p \, r-1}[-d]$ and
 let $C^\bullet := cone(\varphi)$. Since $i^* \varphi$ is an isomorphism, we have the isomorphism $i^* C^\bullet\cong 0$ in the derived category.

 Then, the distinguised triangle
 $$ i_{r!}i_r^* C^\bullet \rightarrow C^\bullet \rightarrow j_{r*}j_r^* C^\bullet \xrightarrow{[1]}$$
 implies there exist an isomorphism $ C^\bullet \cong j_{r*}j_r^* C^\bullet$.

 Moreover, with the long exact sequence of cohomology associated to
 $$j_{r*}j_r^* B^\bullet \xrightarrow{j_{r*}j_r^* \varphi} j_{r*}j_r^* i_{r*} IC_{\bar p \, r-1}[-n] \xrightarrow{\psi}
 j_{r*}j_r^* C^\bullet$$
 we prove $\mathcal H^i(j_{r*}j_r^* C^\bullet)=0$ if $i \leq \bar p(r)$ and $\H^i(\psi):\mathcal H^i(j_{r*}j_r^* i_{r*} IC_{\bar p \, r-1}[-d]) \rightarrow
 \mathcal H^i(j_{r*}j_r^* C^\bullet)$ is an isomorphism if $i> \bar p(r)$. Then, we obtain isomorphisms
 $$C^\bullet \cong \tau_{> \bar p(r)} j_{r*} j_r^* i_{r*} IC_{\bar p \, r-1}[-d] \hspace{1cm} \text{and} \hspace{1cm} B^\bullet\cong IC_{\bar p \, r}[-d].$$

 Repeating this process finitely we obtain $IC_{\bar p}[-d] \in D^b_{cc}(X)$ verifying $[AX1]_k$ for $k=2,..., d$ and it
 is unique up to isomorphism.

  Now, we prove \textbf{2.} Let $IS_{r-1}$ be an intersection space complex in $U_r$. We have to prove that there is a bijective map

 $$\left\{\begin{array}{c}\text{intersection space complexes}\\ IS_{r} \in D^b_{cc}(U_{r+1}) \, \text{such}  \\
  \text{that} \, (IS_{r})_{|U_r}\cong IS_{r-1} \end{array}\right\} /
  \left\{\text{isomorphism}\right\} \longleftrightarrow \left\{\text{retracts of}\, f\right\}/\sim
$$

Let $\lambda$ be a retract of $f$ and consider the following composition of morphisms:
$$\xymatrix{ i_{r*} IS_{r-1} \ar[r]^{a \hspace{4mm}} \ar@/^{6mm}/[rr]^{\varphi_\lambda} & j_{r*} j_r^* i_{r*}IS_{r-1}
\ar[r]^{\lambda \hspace{0.5cm}} & \tau_{\leq \bar q(r)} j_{r*} j_r^* i_{r*}IS_{r-1}}$$
where $a$ is the canonical morphism.

Let us define $IS_{r}:= cone(\varphi_ \lambda) [-1]$. Then, there is a distinguised triangle:
\begin{equation}\label{tr3}
IS_{r}\rightarrow i_{r*} IS_{r-1}\xrightarrow{\varphi_\lambda}\tau_{\leq \bar q(r)}j_{r*}j_r^*i_{r*}IS_{r-1}\xrightarrow{[1]}
\end{equation}

 Since $(j_{r*}j_r^* i_{r*}IS_{r-1})_{|U_r}$ equals $0$, $(IS_{r})_{|U_r}$ is isomorphic to $IS_r$ and, using the long exact sequence
 associated to the triangle, one proves that $IS_{r}$ verifies $[AXS1]_k$ for $k=2,..., r$.

 Let $\lambda'$ be a different retract of $f$ such that $\lambda \sim \lambda'$. We have to prove $cone(\varphi_\lambda)
 \cong cone(\varphi_{\lambda'})$.

 Let $\alpha:\tau_{\leq \bar q(r)} j_{r*} j_r^* i_{r*} IS_{r-1} \rightarrow \tau_{\leq \bar q(r)} j_{r*} j_r^* i_{r*} IS_{r-1}$
 and $\beta:i_{r*} IS_{r-1} \rightarrow i_{r*} IS_{r-1}$ be isomorphisms such that $\lambda' = \alpha \circ \lambda \circ j_{r*}
 j_r^* \beta $.

 Since we know the equalities $\varphi_\lambda= \lambda \circ a$ and $\varphi_{\lambda'}= \lambda' \circ a$, we have to prove
 the isomorphism $cone(\lambda
 \circ a) \cong cone(\alpha \circ \lambda \circ j_{r*} j_r^* \beta \circ a)$. Moreover, since $\alpha$ is an isomorphism, $cone(
 \alpha \circ \lambda \circ j_{r*} j_r^*\beta \circ a)$ is isomorphic to $cone(\lambda \circ j_{r*} j_r^*\beta \circ a)$. Now, consider
 the following diagrams associated to the octahedral axiom of distinguished triangles.
\begin{equation}\label{diag1}
 \xymatrix@C=2mm{i_{r*} IS_{r-1} \ar@/^1.5pc/[rrrr]^{\lambda \circ a}\ar[rr]^a & & j_{r*} j_r^* i_{r*}IS_{r-1} \ar[rr]^\lambda
 \ar[ld]& & \tau_{\leq \bar q(r)} j_{r*} j_r^* i_{r*} IS_{r-1} \ar[ld] \ar@/^1.5pc/[lldd]\\ & cone(a) \ar[lu]_{[1]} \ar[rd]& &
 cone(\lambda) \ar[lu]^{[1]} \ar[ll]^{[1]} & \\ & & cone(\lambda\circ a) \ar[ru] \ar@/^2pc/[lluu]^{[1]} & & }
\end{equation}
\begin{equation}\label{diag2}
 \xymatrix@C=2mm{i_{r*} IS_{r-1} \ar@/^1.5pc/[rrrr]^{\lambda \circ j_{r*} j_r^* \beta \circ a}\ar[rr]^a & & j_{r*} j_r^* i_{r*}
 IS_{r-1} \ar[rr]^{\lambda \circ j_{r*} j_r^* \beta} \ar[ld]& & \tau_{\leq \bar q(r)} j_{r*} j_r^* i_{r*} IS_{r-1} \ar[ld]
 \ar@/^2.5pc/[lldd] \\ & cone(a) \ar[lu]_{[1]} \ar[rd]& & cone(\lambda \circ j_{r*} j_r^* \beta) \ar[lu]^{[1]} \ar[ll]^{[1]}
 & \\ & & cone(\lambda \circ j_{r*} j_r^*  \beta\circ a) \ar[ru] \ar@/^2.5pc/[lluu]^{[1]} & & }
 \end{equation}

 Let $\phi: cone (\lambda\circ j_{r*} j_r^* \beta) \rightarrow cone(\lambda)$ be an isomorphism completing the following isomorphism
 between triangles
$$
\xymatrix@C=1.4cm{j_{r*} j_r^* i_{r*}IS_{r-1} \ar[r]^{\lambda\circ j_{r*} j_r^* \beta\hspace{4mm}} \ar[d]^{j_{r*} j_r^* \beta} &
\tau_{\leq \bar q(r)} j_{r*} j_r^* i_{r*} IS_{r-1} \ar[r] \ar[d]^{Id} & cone(\lambda \circ j_{r*} j_r^* \beta) \ar[d]^\phi\\
j_{r*} j_r^* i_{r*}IS_{r-1} \ar[r]^{\lambda \hspace{5mm}} & \tau_{\leq \bar q(r)} j_{r*} j_r^* i_{r*} IS_{r-1} \ar[r]&cone(\lambda)}
 $$

 Then, if $\rho: cone (a) \rightarrow cone (a)$ is an isomorphism completing the triangles isomorphism
$$
\xymatrix{i_{r*} IS_{r-1} \ar[r]^{a\hspace{5mm}} \ar[d]^{\beta} & j_{r*} j_r^* i_{r*}IS_{r-1} \ar[r] \ar[d]^{j_{r*} j_r^* \beta}&
cone(a)\ar[d]^\rho \\ i_{r*} IS_{r-1} \ar[r]^{a\hspace{5mm}} & j_{r*} j_r^* i_{r*}IS_{r-1} \ar[r] & cone(a)}
$$
the diagram
$$ \xymatrix{ cone(\lambda\circ j_{r*} j_r^* \beta) \ar[r]^{\hspace{5mm}[1]} \ar[d]^\phi & cone(a) \ar[d]^\rho\\ cone(\lambda)
\ar[r]^{[1]} & cone(a)}$$
is commutative.

Therefore, there exist an isomorphism $cone(\lambda\circ j_{r*} j_r^* \beta \circ a) \cong cone(\lambda\circ a )$, completing
the following morphism between triangles
$$ \xymatrix{ cone(\lambda\circ j_{r*} j_r^* \beta) \ar[r]^{\hspace{5mm}[1]} \ar[d]^\phi & cone(a) \ar[d]^\rho \ar[r] &
cone(\lambda\circ j_{r*} j_r^* \beta \circ a) \ar[d]\\ one(\lambda) \ar[r]^{[1]} & cone(a) \ar[r] & cone(\lambda\circ a )}$$

Now, suppose there exist an intersection space complex in $U_{r+1}$, $IS_{r}$ such that $(IS_{r})_{|U_r}\cong IS_{r-1}$.
We have to prove that the triangle (\ref{tr1}) is split.

Let $h:IS_{r} \rightarrow i_{r*} IS_{r-1}$ be the composition of the canonical morphism $IS_{r} \rightarrow
i_{r*} (IS_{r})_{|U_r}$ and an isomorphism $i_{r*} (IS_{r})_{|U_r}\cong i_{r*} IS_{r-1}$ and let $C^\bullet:=cone(h)$. Then,
there is a distinguished triangle: $$IS_{r}\xrightarrow{h} i_{r*} IS_{r-1} \xrightarrow{g} C^\bullet \xrightarrow{[1]}$$

Note that $i_r^*h:i_r^* IS_{r} \rightarrow IS_{r-1}$ is an isomorphism. So, $i_r^* C^\bullet$ is isomorphic to $0$ in the derived category.
Then, the canonical triangle
$$ i_{r!}i_r^* C^\bullet \rightarrow C^\bullet \rightarrow j_{r*}j_r^* C^\bullet \xrightarrow{[1]}$$
implies the isomorphism $ C^\bullet \cong j_{r*}j_r^* C^\bullet$.

Moreover, the long exact sequence of cohomology associated to the triangle
$$j_{r*}j_r^*IS_{r}\xrightarrow{j_{r*}j_r^* h} j_{r*}j_r^*i_{r*} IS_{r-1} \xrightarrow{j_{r*}j_r^*g} j_{r*}j_r^*C^\bullet
\xrightarrow{[1]}$$
implies that $$\mathcal{H}^i(j_{r*}j_r^*C^\bullet)\cong\left\{ \begin{array}{cc} 0 & \text{if}\, i> \bar q(r) \\
\mathcal{H}^i(j_{r*}j_r^*i_{r*} IS_{r-1}) & \text{if}\, i\leq \bar q(r)\end{array}\right.$$

Applying the functor $\tau_{\leq \bar q(r)}$ to $j_{r*}j_r^*g$, we obtain the following commutative diagram:
\begin{equation}\label{retract1}
\xymatrix{ \tau_{\leq \bar q(r)} j_{r*}j_r^*i_{r*} IS_{r-1} \ar[r]^{a} \ar[d]^{f} & \tau_{\leq \bar q(r)} j_{r*}j_r^*
 C^\bullet\ar[d]^{c} \\ j_{r*}j_r^*i_{r*} IS_{r-1} \ar[r]^{b}& j_{r*}j_r^* C^\bullet}
\end{equation}
where $f$ and $c$ are the canonical morphisms, $a =\tau_{\leq \bar q(r)} j_{r*}j_r^*g$ and $b=j_{r*}j_r^*g$. Moreover, $a$
and $c$ are isomorphisms and the composition $\lambda=a^{-1} \circ c^{-1} \circ b$ is a retract of $f$. So, the triangle
(\ref{tr1}) is split by Lemma~\ref{lem:split}.

Let $IS'_{r} \in D^b_{cc}(U_{r+1})$ be isomorphic to $IS_{r}$. Then, we have an isomorphism $(IS'_{r})_{|U_r}\cong IS_{r-1}$ and
$IS'_{r}$ is an intersection space complex in $U_{r+1}$.

Let $h':IS'_{r} \rightarrow i_{r*} IS_{r-1}$ be the composition of the canonical restriction morphism and an isomorphism
$i_{r*} (IS'_{r})_{|U_r}\cong i_{r*} IS_{r-1}$, let $K^\bullet:=cone(h')$ and consider the triangle
$$IS'_{r}\xrightarrow{h'} i_{r*} IS_r \xrightarrow{g'} K^\bullet \xrightarrow{[1]}$$

Applying the functor $\tau_{\leq \bar q(r)}$ to $j_{r*}j_r^*g'$, we obtain the following commutative diagram:
\begin{equation}\label{retract2}
\xymatrix{\tau_{\leq \bar q(r)} j_{r*}j_r^*i_{r*} IS_{r-1} \ar[r]^{a'} \ar[d]^{h} & \tau_{\leq \bar q(r)} j_{r*}j_r^* K^\bullet
\ar[d]^{c'}\\j_{r*}j_r^*i_{r*} IS_{r-1} \ar[r]^{b'}& j_{r*}j_r^* K^\bullet}
\end{equation}
where $h$ and $c'$ are the canonical morphisms, $a' =\tau_{\leq \bar q(r)} j_{r*}j_r^*g'$ and $b'=j_{r*}j_r^*g'$. Moreover,
$a'$ and $c'$ are isomorphisms and $\lambda':=a'^{-1} \circ c'^{-1} \circ b'$ is a retract of $f$.

Let $\alpha:IS_{r} \rightarrow IS'_{r}$ be an isomorphism; let $h$ be the composition of the canonical morphism $IS_{r}
\rightarrow i_{r *} (IS_{r})_{|U_r}$ and an isomorphism $\gamma: i_{r *} (IS_{r})_{|U_r} \rightarrow i_{r *} IS_{r-1}$; let
$h'$ be the composition of the canonical morphism $IS'_{r} \rightarrow i_{r *} (IS'_{r})_{|U_k}$ and an isomorphism
$\gamma': i_{r *} (IS'_{r})_{|U_k} \rightarrow i_{r *} IS'_{r-1}$. Finally define $\beta:= \gamma' \circ i_{r *} i_r^* \alpha \circ
\gamma^{-1}$. Then, there is an isomorphism between triangles:
$$\xymatrix{j_{r*}j_r^*IS_{r}\ar[r]^{j_{r*}j_r^* h} \ar[d]^{ j_{r*}j_r^* \alpha} & j_{r*}j_r^*i_{r*} IS_{r-1}
\ar[r]^{j_{r*}j_r^*g} \ar[d]^{j_{r*}j_r^* \beta} &  j_{r*}j_r^*C^\bullet \ar[d]^\delta \\
j_{r*}j_r^* IS'_{r}\ar[r]^{j_{r*}j_r^* h'} & j_{r*}j_r^* i_{r*} IS_{r-1} \ar[r]^{j_{r*}j_r^*g'} & j_{r*}j_r^* K^\bullet}$$
where $\alpha$, $\beta$ and $\delta$ are isomorphisms.

Then, the morphisms of diagrams (\ref{retract1}) and (\ref{retract2}) have the following relations:
$$ a'=  \tau_{\leq \bar q(r)} \delta \circ  a  \circ (\tau_{\leq \bar q(r)} j_{r*}j_r^*\beta)^{-1}$$
$$ b'=   \delta \circ  a  \circ (j_{r*}j_r^*\beta)^{-1}$$
$$ c'=   \delta \circ  c  \circ (\tau_{\leq \bar q(r)} \delta)^{-1}$$
and, we obtain $\lambda= (\tau_{\leq \bar q(r)} j_{r*}j_r^* \beta)^{-1} \circ \lambda' \circ j_{r*}j_r^* \beta$. So,
$\lambda$ is equivalent to $\lambda'$ by the equivalence relation $\sim$.
\end{proof}

\begin{remark}
There is a unique object in $D^b_{cc}(U_2)$ up to isomorphism verifying $(a)$ and $(b)$ from $[AXS1]_k$, $\QQ_{U_2}$.
\end{remark}

Now we are going to study the equivalence relation $\sim$, which appears in Theorem \ref{th:existunic}.

Remember that two retracts $\lambda_1$ and $\lambda_2$ of $\tau_{\leq \bar q(r)} j_{r*} j_r^* i_{r*} IS_{r-1} \xrightarrow{f}
j_{r*} j_r^* i_{r*} IS_{r-1}$ are equivalent by $\sim$ if and only if there exit isomorphisms
$$\alpha:\tau_{\leq \bar q(r)}
j_{r*} j_r^* i_{r*} IS_{r-1} \rightarrow \tau_{\leq \bar q(r)} j_{r*} j_r^* i_{r*} IS_{r-1}$$
$$\beta:i_{r*} IS_{r-1} \rightarrow i_{r*} IS_{r-1}$$
such that $\lambda_2 = \alpha \circ \lambda_1 \circ j_{r*} j_r^* \beta $.

Let $\alpha:\tau_{\leq \bar q(r)} j_{r*} j_r^* i_{r*} IS_{r-1} \rightarrow \tau_{\leq \bar q(r)} j_{r*} j_r^* i_{r*} IS_{r-1}$ and
$\beta:i_{r*} IS_{r-1} \rightarrow i_{r*} IS_{r-1}$ be isomorphisms and let $\lambda$ be a retract of $f$. Then, $\alpha \circ \lambda
\circ j_{r*} j_r^* \beta $ is a retract of $f$ if and only if $\alpha \circ \lambda \circ j_{r*} j_r^* \beta \circ f =Id$,
that is, if $\alpha$ is equal to $(\lambda \circ j_{r*} j_r^* \beta \circ f)^{-1}$. So, $\alpha$ is determined by $\beta$ and $\lambda$ and
the set of retracts of $f$ which are equivalent to
$\lambda$ is determined by the automorphisms $i_{r*} IS_{r-1} \rightarrow i_{r*} IS_{r-1}$.

Since $i_r^*$ is left-adjoint to $i_{r *}$, the space of automorphisms $Aut(i_{r*} IS_{r-1})$ is isomorphic to the space of
isomorphisms $Iso(i_r^* i_{r*} IS_{r-1}, IS_{r-1})=Aut(IS_{r-1})$.

\subsection{The space of obstructions to existence and uniqueness.}

In particular, if $r$ is the dimension of the largest non-trivial
stratum, we have $U_r =U_2$ and $Aut(IS_{r-1})$ is isomorphic to $Aut(\QQ_{U_2})$, which is the space of homothetic transformations.
Moreover, if $\beta \in Aut(i_{r*} IS_{r-1})$ is a homothetic transformation $\alpha =(\lambda \circ j_{r*} j_r^* \beta \circ
f)^{-1}$ is the inverse homothetic transformation. So,if $r$ is the dimension of the largest non-trivial stratum, the
equivalence relation $\sim$ is trivial. Consequently, there is a bijective map
$$
 \left\{\begin{array}{c}\text{intersection space complexes}\\ IS_{r} \in D^b_{cc}(U_{r+1}) \end{array}\right\} /
  \left\{\text{isomorphism}\right\} \longleftrightarrow \left\{\text{retracts of}\, f\right\}
 $$

\begin{remark}\label{rem:modretracts}
For any $r$, the triangle (\ref{tr1}) induces a exact sequence
$$...\rightarrow [\tau_{\leq \bar q(r)} j_{r*} j_r^* i_{r*} IS_{r-1}, \tau_{\leq \bar q(r)} j_{r*} j_r^* i_{r*} IS_{r-1} [-1]]
\rightarrow $$ $$\to[\tau_{> \bar q(r)} j_{r*} j_r^* i_{r*} IS_{r-1}, \tau_{\leq \bar q(r)} j_{r*} j_r^* i_{r*} IS_{r-1}]
\xrightarrow{\widetilde g}[j_{r*} j_r^* i_{r*} IS_{r-1}, \tau_{\leq \bar q(r)} j_{r*} j_r^* i_{r*}
IS_{r-1}] \xrightarrow{\widetilde f}$$ $$\xrightarrow{\widetilde f} [\tau_{\leq \bar q(r)} j_{r*} j_r^* i_{r*} IS_{r-1}, \tau_{\leq \bar q(r)} j_{r*} j_r^* i_{r*}
IS_{r-1}] \rightarrow ...$$

Moreover, the retracts of $f$ are de elements $\lambda \in [j_{r*} j_r^* i_{r*} IS_{r-1}, \tau_{\leq \bar q(r)} j_{r*} j_r^* i_{r*}
IS_{r-1}]$ such that $\widetilde{f} (\lambda)$ is the identity. So, the space $\{$retracts of $f\}$ is modulated by the vector space
$$[\tau_{> \bar q(r)} j_{r*} j_r^* i_{r*} IS_{r-1}, \tau_{\leq \bar q(r)} j_{r*} j_r^* i_{r*} IS_{r-1}]$$
\end{remark}

The following is an immediate consequence of the previous construction.

\begin{corollary}
\label{cor:obstructions}
Suppose there exist an intersection space complex $IS_{r-1}$ in $U_r$, that is, $IS_{r-1}$ is an object
  in $D^b_{cc}(U_r)$ and it verifies $[AXS1]_k$ for $k=2,...,r-1$.
 \begin{itemize}
  \item The obstruction to existence of intersection space in the next stratum lives in
 $$Ext^1(\tau_{> \bar q(r)} j_{r*} j_r^* i_{r*} IS_{r-1}, \tau_{\leq \bar q(r)} j_{r*} j_r^* i_{r*} IS_{r-1})=[\tau_{> \bar q(r)} j_{r*} j_r^* i_{r*} IS_{r-1}, \tau_{\leq \bar q(r)} j_{r*} j_r^* i_{r*} IS_{r-1}[1]].$$
  \item The obstructions for uniqueness live in  the group
 $$Hom(\tau_{> \bar q(r)} j_{r*} j_r^* i_{r*} IS_{r-1}, \tau_{\leq \bar q(r)} j_{r*} j_r^* i_{r*} IS_{r-1})=[\tau_{> \bar q(r)} j_{r*} j_r^* i_{r*} IS_{r-1}, \tau_{\leq \bar q(r)} j_{r*} j_r^* i_{r*} IS_{r-1}],$$
 modulo the equivalence relation described above. The equivalence relation is trivial for the second stratum.
 \end{itemize}
\end{corollary}

\section{A mixed Hodge module structure in intersection space complexes of algebraic varieties}

\begin{theorem}
\label{th:mhm}
Let $X$ be an algebraic variety. Consider a stratification
$$
X=X_d\supset X_{d-2} \supset ... \supset X_0 \supset X_{-1}=\emptyset
$$
by algebraic subvarieties, which makes $X$ a topological pseudomanifold. An intersection space complex on $X$ associated with the stratification above
admits a lifting to the derived category $D^bMHM(X)$ of mixed Hodge modules on $X$ if and only if the choices of the retractions are morphisms
of mixed Hodge modules. Consequently, its hypercohomology groups carry a rational polarizable mixed Hodge structure.
\end{theorem}
\begin{proof}
The proof follows the inductive construction of the intersection space complex considered in the proof of Theorem~\ref{th:existunic}.
To start with, we notice that $\QQ_{U_2}$ is a mixed Hodge module
on the algebraic variety $U_2$. Assume, by induction, that the intersection space complex $IS_{r-1}$ defined in $U_r$ is a mixed Hodge module.

In order to
construct $IS_{r}$ we proceed as follows: observe that the triangle
$$\tau_{\leq\bar q(r)}j_{r*}j^{*}_r i_{r*} IS_{r-1} \to j_{r*}j^{*}_ri_{r*}IS_{r-1}\to \tau_{>\bar q(r)}j_{r*}j^{*}_r i_{r*} IS_{r-1}\xrightarrow{[1]}$$
is a triangle of mixed Hodge modules in $U_{r+1}$. This is true because the derived category of mixed Hodge modules are preserved by Grothendieck's
$6$ operations, and because the truncation triangle is a distinguished triangle in the derived category of mixed Hodge modules (use Saito's anomalous
$t$-structure).

By Lemma~\ref{lem:split}, the triangle splits if and only if the connecting morphism equals $0$, and if this happens in the derived category of
constructible sheaves, then it happens too in the derived category of mixed Hodge modules, since the functor $rat$ is faithful. We consider a
retraction $\lambda:j_{r*}j^{*}_ri_{r*}IS_{r-1}\to \tau_{\leq\bar q(r)}j_{r*}j^{*}_r i_{r*}IS_{r-1}$, which by assumption is a morphism in $D^bMHM(U_{r+1})$.
Since $IS_{r}$ is, by definition, the shifted cone $cone(\varphi_\lambda)[-1]$, where $\varphi_\lambda$ is the composition of $\lambda$ with
the canonical morphism $a:i_{r*}IS_{r-1}\to j_{r*}j^{*}_ri_{r*}IS_{r-1}$ we have that $IS_{r}$ belongs to $D^bMHM(U_{r+1})$.
\end{proof}

The obstructions to existence and uniqueness can be lifted to mixed Hodge modules:

\begin{corollary}
\label{cor:obstructionsmhm}
Let $X$ be an algebraic variety. Consider an stratification
$$
X=X_d\supset X_{d-2} \supset ... \supset X_0 \supset X_{-1}=\emptyset
$$
by algebraic subvarieties, which makes $X$ a topological pseudomanifold.
Suppose there exist an intersection space complex $IS_{r-1}$ in $U_r$ which belongs to $D^b(MHM(U_{r+1}))$.
 \begin{itemize}
  \item The obstruction to existence of intersection space in the next stratum lives in
 $$Ext^1_{D^bMHM(U_{r+1})}(\tau_{> \bar q(r)} j_{r*} j_r^* i_{r*} IS_{r-1}, \tau_{\leq \bar q(r)} j_{r*} j_r^* i_{r*} IS_{r-1}).$$
  \item The obstructions for uniqueness live in  the group
 $$Hom_{D^bMHM(U_{r+1})}(\tau_{> \bar q(r)} j_{r*} j_r^* i_{r*} IS_{r-1}, \tau_{\leq \bar q(r)} j_{r*} j_r^* i_{r*} IS_{r-1}),$$
 modulo the equivalence relation described above. The equivalence relation is trivial for the second stratum.
 \end{itemize}
\end{corollary}

A simplification of the proof of Theorem~\ref{th:mhm} yields:

\begin{theorem}
\label{th:icmhm}
Let $X$ be an algebraic variety. Let $\bar p$ be any perversity. The intersection cohomology complex associated to it belongs to the derived category of
mixed Hodge modules of $X$. Consequently the intersection homology complexes $IH_{\bar p}^k(X,\QQ)$ carry a canonical polarizable mixed Hodge structure.
\end{theorem}

\section{Classes of spaces admitting intersection space complexes and counterexamples}
\label{sec:examples}
In this section we provide some examples and counterexamples to illustrate our theory.
First, we introduce two classes of varieties which admit an intersection space complex for every
perversity. The first class depends on the tubular neighbourhoods of the strata: if every stratum admits
a trivial tubular neighbourhood, then there exist the intersection space complex for every perversity.
The second class depends on the dimension of the strata: if every singular stratum has homological dimension for locally
constant sheaves bounded by $1$, then there exist the intersection space complex for every perversity. Then, we find concrete examples of pseudomanifolds
(including an algebraic variety) not admiting intersection space complexes and, hence, not admiting intersection space pairs.

\subsection{Pseudomanifolds satisfying triviality properties}

Let $X$ be a topological pseudomanifold with the following stratification:
$$
X=X_d\supset X_{d-2} \supset ... \supset X_0 \supset X_{-1}=\emptyset
$$
For $k=2,...,d+1$, let $U_k:=X\setminus X_{d-k}$. We will also denote by $i_k:U_k \to U_{k+1}$,
\\$j_k:X_{d-k} \setminus X_{d-k-1} \to U_{k+1}$ and $i_{k_1,k_2}:U_{k_1} \to U_{k_2}$ the canonical inclusions.

If $X$ has a trivial conical structure then, by Theorem~\ref{th:existTr}, there exists an intersection space pair. Hence, by Theorem~\ref{th:existints}, $X$ has an
intersection space complex. We have shown:

\begin{corollary}
 If $X$ has a trivial conical structure, then it has an intersection space complex.
\end{corollary}

\begin{example}
\label{ex:toric2}
Toric varieties admit
an intersection space pairs and intersection space complex for every perversity.
\end{example}

However, as we prove below, in order to ensure the existence of the intersection space complex we can relax the triviality hypothesis on the stratification:
one only needs that the triviality property ($T_r$) (see~\ref{def:Tr}) is satisfied for any stratum. Having a trivial conical structure requires further
compatibilities between the trivializations predicted by properties $(T_r)$ (see Definition~\ref{def:systriv}).

\begin{definition}
\label{def:formal}
A complex of sheaves is {\em formal and constant} if it is quasi-isomorphic to the direct sum of its cohomology sheaves with zero differentials and these cohomology sheaves are constant.
\end{definition}

\begin{definition}\label{def:Pk}
If $B^\bullet_{k-1}$ is a complex of sheaves in $U_k$, we say $B^\bullet_{k-1}$ verifies the property $(P_r)$, where $r\geq k$, if
$j_{r}^*i_{k, r+1 *} B^\bullet_{k-1}$ is formal and constant.
\end{definition}

\begin{remark}
If $B^\bullet$ is a formal and constant complex of sheaves in $U_k$ and
$(X,X_{d-2})$ has a conical structure which verifies the property $(T_r)$ of Definition \ref{def:Tr} for some $r\geq k$, then
$B^\bullet$ verifies the property $(P_r)$.
\end{remark}

\begin{proposition}\label{prop:intsptriv}
Given $k\in\{2,...,d\}$, if there exist an intersection space $IS_{k-1}$ with perversity $\bar p$ in $U_k$, which
verifies $(P_k)$, then there exist an intersection space $IS_{k}$ with perversity $\bar p$ in $U_{k+1}$, such that
$(IS_{k})_{|U_k}$ is quasi-isomorphic to $IS_{k-1}$.
\end{proposition}
\begin{proof}
$j_{k}^*i_{k *} IS_{k-1}$ is (up to isomorphism in the derived category) a complex of constant
sheaves with zero differentials. So, the triangle
$$\tau_{\leq q} j_{k *} j_{k}^*i_{k *} IS_{k-1} \rightarrow j_{k *} j_{k}^*i_{k *} IS_{k-1} \rightarrow
\tau_{> q} j_{k *} j_{k}^*i_{k *} IS_{k-1}\xrightarrow{[1]}$$
is split for every $q\in \ZZ$ and, applying Theorem \ref{th:existunic}, we conclude.
\end{proof}

\begin{lemma}
\label{lem:existintspace}
Let us suppose that there exist an intersection space complex $IS_{k-1}$ with perversity $\bar p$ in $U_k$.
If $(IS_{k-1})_{|U_{k-1}}$ verifies the properties $(P_{k-1})$ and $(P_{r})$ for a certain $r\geq k$ and $(X,X_{d-2})$ has a conical
structure which verifies the property $(T_{r})$ of Definition \ref{def:Tr}, then $IS_{k-1}$ verifies the property $(P_r)$.
\end{lemma}
\begin{proof}
By Theorem \ref{th:existunic},
$j_{r}^* i_{k, r+1 *} IS_{k-1}[1]$ is quasi-isomorphic to the cone of a morphism
$$j_{r}^* i_{k-1, r+1 *} (IS_{k-1})_{|U_{k-1}} \to j_{r}^* i_{k, r+1 *} \tau_{\leq \bar q(k-1)} j_{k-1 *} j_{k-1}^* i_{k-1 *} (IS_{k-1})_{|U_{k-1}}$$
Since $(IS_{k-1})_{|U_{k-1}}$ verifies the properties $(P_{k-1})$ and $(P_{r})$, the complexes $j_{r}^* i_{k-1, r+1 *}(IS_{k-1})_{|U_{k-1}}$ and
$\tau_{\leq \bar q(k-1)} j_{k-1}^* i_{k-1 *} (IS_{k-1})_{|U_{k-1}}$
are formal and constant. Then, using that $(X,X_{d-2})$ has a conical
structure verifying the property $(T_r)$, we deduce that the constructible complex 
$j_{r}^* i_{k, r+1 *} \tau_{\leq \bar q(k-1)} j_{k-1 *} j_{k-1}^* i_{k-1 *} (IS_{k-1})_{|U_{k-1}}$
is also formal and constant. So, $j_{r}^* i_{k, r+1 *} IS_{k-1}$ is formal and constant as well and we conclude.
\end{proof}

\begin{theorem}\label{th:exintspacecomplex}
If the pair $(X,X_{d-2})$ has a conical structure which verifies the property $(T_r)$ of Definition \ref{def:Tr}
for any $r$, then there exist the intersection space complex of $X$ for every perversity.
\end{theorem}
\begin{proof}

The constant sheaf $\QQ_{U_2}$ verifies $(P_r)$ for every $r \geq 2$ such that the pair $(X,X_{d-2})$ has a conical
structure with the property $(T_r)$. So, if $(X,X_{d-2})$ has a conical structure which verifies the property $(T_r)$
for any $r$, using Lemma \ref{lem:existintspace} and Proposition \ref{prop:intsptriv}, we can construct inductively for every $k$
an intersection space complex with perversity $\bar p$ in $U_k$ which verifies $(P_{r})$ for every $r\geq k$.
\end{proof}

\subsection{Pseudomanifolds with strata of small homological dimension}

\begin{definition}
\label{def:homdim}
 A space $Y$ has homological dimension for locally constant sheaves bounded by $m$ if any locally constant sheaf in $Y$ has no cohomology in degree
 higher than $m$.
\end{definition}

\begin{theorem}
\label{th:homdim}
Let $X$ be a topological pseudomanifold with the following stratification:
$$
X=X_d\supset X_{d-2} \supset ... \supset X_0 \supset X_{-1}=\emptyset
$$
such that $X_{d-r}\setminus X_{d-r-1}$ has homological dimension for locally constant sheaves bounded by $1$ for any $r$.

Then, there exist the intersection space complex of $X$ for every perversity $\bar p$.

Moreover, if $X_{d-r}\setminus X_{d-r-1}$ has homological dimension for locally constant sheaves bounded by $0$ for any $r$,
the intersection space complex is unique.
\end{theorem}
\begin{proof}
To prove the existence it is enough to prove that, for any topological space $Y$ which has  homological dimension for locally constant sheaves
bounded by $1$, any complex of sheaves $B^\bullet$ in $Y$ and any integer $m$, the triangle
$$\tau_{\leq m} B^\bullet \rightarrow B^\bullet \rightarrow \tau_{>m} B^\bullet \xrightarrow{\phi}
\tau_{\leq m} B^\bullet [1] \rightarrow ...$$
is split.

This triangle is split if and only the morphism
$$\phi \in Ext^1 (\tau_{> m} B^\bullet, \tau_{\leq m} B^\bullet)$$
is $0$.

$Ext^1 (\tau_{> m} B^\bullet, \tau_{\leq m} B^\bullet)$ is the first hipercohomology group of the complex
of sheaves \\$\mathcal Hom^\bullet(\tau_{> m} B^\bullet, \tau_{\leq m} B^\bullet)$.

Let $E_r^{p,q}$ be the local to global spectral sequence of
$\mathcal Hom^\bullet(\tau_{\leq m} B^\bullet, \tau_{> m} B^\bullet)$.
Then,
$$E_2^{p,q}= \HH^p (Y, \mathcal Ext^q(\tau_{> m} B^\bullet, \tau_{\leq m} B^\bullet)).$$
Moreover, the fiber of the sheaf $\mathcal Ext^q(\tau_{\leq m} B^\bullet, \tau_{> m} B^\bullet)$ in a point
$x$ is $[\tau_{> m} B^\bullet_x, \tau_{\leq m} B^\bullet_x[q]]$. Since $ \tau_{\leq m} B^\bullet_x[q]$ is a
complex of injective sheaves, $[\tau_{> m} B^\bullet_x, \tau_{\leq m} B^\bullet_x[q]]$ is equal to $0$
for every $q \geq 0$. So, the sheaf $\mathcal Ext^q(\tau_{\leq m} B^\bullet, \tau_{> m} B^\bullet)$ is $0$
for every $q \geq 0$.

Since  $Y$ has  homological dimension for locally constant sheaves bounded by $1$, the group
$\HH^p (Y, \mathcal Ext^q(\tau_{> m} B^\bullet, \tau_{\leq m} B^\bullet))$ is equal
to $0$ for every $p> 1$. So, $E_2^{p,q}=0$ for every $p, q \in \ZZ$ such that $p+q=1$.

Hence, $Ext^1 (\tau_{> m} B^\bullet, \tau_{\leq m} B^\bullet)=0$, and we conclude the proof of the existence.

Moreover, by Remark \ref{rem:modretracts}, the retracts of the triangle are modulated by $Ext^0 (\tau_{> m} B^\bullet, \tau_{\leq m} B^\bullet)$.
With the previous method, we show that, if $Y$ has homological dimension for locally constant sheaves bounded by $0$,
$Ext^0 (\tau_{> m} B^\bullet, \tau_{\leq m} B^\bullet)=0$. Then, the retract of the triangle is unique and we conclude.
\end{proof}

\begin{corollary}
\label{cor:smalldim}
If the strata $X_{d-r}\setminus X_{d-r-1}$  has the homotopy type of a $1$-dimensional $CW$-complex of dimension bounded by $1$ for any $r$,
then here exist the intersection space complex of $X$ for every perversity $\bar p$.
\end{corollary}

\begin{example}
Any algebraic variety with $1$ dimensional critical set and the canonical Whitney stratification if each connected component of the critical set is not a stratum as
a whole (that happens in the non Whitney-equisingular case) satisfy the hypothesis of the previous corollary, and hence admits an intersection space complex.
\end{example}

\subsection{Counterexamples}
Now, we illustrate the limits of our theory with a class of varieties which does not admit an intersection space complex for some perversities.
With this purpose, the following proposition gives a necessary condition
for the splitting of a triangle
$$\tau_{\leq m} B^\bullet \rightarrow B^\bullet \rightarrow \tau_{>m} B^\bullet \xrightarrow{[1]} $$

\begin{proposition}
\label{prop:lerayobs}
Let $X$ be a topological space, let $B^\bullet$ be a bounded complex of sheaves on $X$ and
let $E_r^{p,q}$ be the local to global spectral sequence of $B^\bullet$.

Then, if the canonical triangle $$\tau_{\leq m} B^\bullet \rightarrow B^\bullet \rightarrow
\tau_{>m} B^\bullet \xrightarrow{[1]} $$ is split, the morphisms $d_r^{p,q}: E_r^{p,q}
\rightarrow E_r^{p+r,q-r+1}$ are $0$ for every $r\geq 2$, $p\in \ZZ$ and $m<q\leq m+r-1$.
\end{proposition}
\begin{proof}
Let us suppose the triangle is split and let $\lambda:B^\bullet \rightarrow \tau_{\leq m} B^\bullet$
be a retract of the canonical morphism.

Let $E_r^{p, \, q}$ be the local to global spectral sequence of $B^\bullet$, $E'^{p, \, q}_r$ the
local to global spectral sequence of $\tau_{\leq m} B^\bullet$ and $\lambda^{p,\, q}_r:E^{p, \, q}_r \rightarrow
E'^{p, \, q}_r$ the morphism induced by $\lambda$.

For $r=2$,
$$\lambda^{p,q}_2: \HH^p(X, \mathcal H^q(B)) \rightarrow \HH^p(X, \mathcal H^q(\tau_{\leq m} B))$$
is an isomorphism if $q \leq m$ and $\HH^p(X, \mathcal H^q(\tau_{\leq m} B))= 0$ if $q > m$.

Given $r\geq 2$, suppose $\lambda_{r}^{p,\, q}:E^{p, \, q}_{r} \rightarrow E'^{p, \, q}_{r}$ is an isomorphism
for every $q \leq m$ and $E'^{p, \, q}_{r}= 0$ for every $q > m$. Then, $E'^{p, \, q}_{r+1}= 0$ for every $q >m$.

Moreover, let us consider the commutative diagram:

$$\xymatrix@C=3cm{E^{p, \, q}_{r} \ar[r]^{\lambda^{p, \, q}_{r}} \ar[d]^{d^{p, \, q}_{r}} & E'^{p, \, q}_{r}
\ar[d]^{d'^{p, \, q}_{r}} \\ E^{p+r, \, q-r+1}_{r} \ar[r]^{\lambda^{p+r, \, q-r+1}_{r}}& E'^{p+r, \, q-r+1}_{r}}$$

If $q\leq m$, $\lambda^{p, \, q}_{r}$ induces an isomorphism between $\ker(d^{p, \, q}_{r})$ and $\ker(d'^{p, \, q}_{r})$
and $\lambda^{p+r, \, q-r+1}_{r}$ induces an isomorphism between $\im(d^{p, \, q}_{r})$ and $\im(d'^{p, \, q}_{r})$.

If $q >m $ and $q-r+1\leq m$, $E'^{p, \, q}_{r}=0$ and, since the diagram is commutative, $d^{p, \, q}_{r}=
d'^{p, \, q}_{r}=0$. So, we deduce $d_r^{p,q}$ is $0$ for every $m<q\leq m+r-1$.

Moreover, $\im(d^{p, \, q}_{r})=\im(d'^{p, \, q}_{r})=0$. Therefore, for every $q \leq m$, $\lambda_{r+1}^{p,\, q}$ is
an isomorphism and we can finish the proof by induction.
\end{proof}

\begin{corollary}\label{cor:cex}
Let $X$ be a topological pseudomanifold with stratification
$$
X=X_d\supset X_{d-2} \supset ... \supset X_0 \supset X_{-1}=\emptyset
$$
and let $k$ be the codimension of $X_{d-2}$, that is, $X_{d-2}=X_{d-k}$.

Let $\bar p$ a perversity and $\bar q$ its complementary perversity. If the local to global spectral sequence of
$j_k^* i_{k *} \QQ_{U_2}$ has any differential $d_r^{p,q}: E_r^{p,q} \rightarrow E_r^{p+r,q-r+1}$ different from $0$
for some $r\geq 2$, $p\in \ZZ$ and $\bar q(k) < q \leq \bar q(k)+r-1$, then there does not exist any intersection space complex of $X$ with perversity $\bar p$.
\end{corollary}

Now, we construct an example which verify these conditions using the Hopf fibration.

\begin{example}
\label{ex:hopf}
Let $\rho^\partial: S^3 \rightarrow S^2$ be the Hopf fibration and let $\rho:cyl(\rho^\partial) \rightarrow S^2$
be the cone of the fibration (see definition \ref{def:fibrecone}). If $s:S^2 \rightarrow cyl(\rho^\partial)$ is the
vertex section, we consider the space $X:=cyl(\rho^\partial)$ with the stratification
$$X\supset s(S^2)$$

Let $U:= X \setminus s(S^2)$ and let $i:U\rightarrow X$ and $j:s(S^2)\rightarrow X$ be the canonical inclusions.
Then, since the fiber of $\rho^\partial$ is $S^1$
$$\H^i(j_* j^*i_* \QQ_U)\left\{\begin{array}{cc}
\QQ_{S^2} & \text{if} \, i=0,1 \\
0 & \text{otherwise}
\end{array}\right.$$

So, if $E_r^{p,q}$ is the hypercohomology spectral sequence of $j_* j^*i_* \QQ_U$, $E_2^{p,q}$ is
$$\begin{array}{ccc|c|c|c}
    & & 1 & \QQ & 0 & \QQ  \\ \cline{3-6}
    q \uparrow& & 0 & \QQ & 0 & \QQ  \\ \cline{3-6}
    & & & 0 & 1 & 2  \\
    \multicolumn{6}{c}{}\\
    & &\multicolumn{4}{c}{\overrightarrow{p}}
\end{array}$$
where the differential $d_2^{0,1}$ is different from $0$.

Moreover, given any perversity $\bar p$, $\bar p(2)=0$. So, applying Corollary \ref{cor:cex}, there does not exist an
intersection space complex of $X$ with stratification $X \supset s(S^2)$ with any perversity.

Hence, applying Remark \ref{re:relationintspace}, there does not exist any intersection space pair of $X$ with the
previous stratification.
\end{example}

The stratification of $X$ given in the previous example is not natural. Since $X$ is smooth, the natural
stratification of $X$ has no stratum different from $X$ and $\emptyset$. The following example is more natural
in the sense that the nontrivial stratum is the singular part of the variety.

\begin{example}
\label{exhopf2}
Let $\rho^\partial: S^3 \rightarrow S^2$ be the Hopf fibration and let us consider the locally trivial fibration
$$\sigma^\partial: S^3 \times_{S^2} S^3 \to S^2 $$

Moreover, let $\sigma:cyl(\sigma^\partial) \rightarrow S^2$ be the cone of $\sigma^\partial$ and
$s:S^2 \rightarrow cyl(\sigma^\partial)$ the vertex section. Then, we define $X:=cyl(\sigma^\partial)$ and we
consider the stratification
$$X\supset s(S^2)$$

Let $U:= X \setminus s(S^2)$ and let $i:U\rightarrow X$ and $j:s(S^2)\rightarrow X$ be the canonical inclusions.
Then, since the fiber of $\sigma^\partial$ is $S^1\times S^1$
$$\H^i(j_* j^*i_* \QQ_U)\left\{\begin{array}{ccc}
\QQ_{S^2} & \text{if} & i=0,2 \\
\QQ^2_{S^2} & \text{if} & i=1 \\
0 & \multicolumn{2}{c}{\text{otherwise}}
\end{array}\right.$$

So, if $E_r^{p,q}$ is the hypercohomology spectral sequence of $j_* j^*i_* \QQ_U$, $E_2^{p,q}$ is
$$\begin{array}{ccc|c|c|c}
    & & 2 & \QQ & 0 & \QQ  \\ \cline{3-6}
    q \uparrow& & 1 & \QQ^2 & 0 & \QQ^2  \\ \cline{3-6}
    & & 0 & \QQ & 0 & \QQ  \\ \cline{3-6}
    & & & 0 & 1 & 2  \\
    \multicolumn{6}{c}{}\\
    & &\multicolumn{4}{c}{\overrightarrow{p}}
\end{array}$$
where the differentials $d_2^{0,1}$ and $d_2^{0,2}$ are different from $0$.

Moreover, given any perversity $\bar p$, either $\bar p(3)=0$ or $\bar p(3)=1$. So, applying Corollary \ref{cor:cex},
there does not exist an intersection space complex of $X$ with stratification $X \supset s(S^2)$ with any perversity.

Hence, applying Remark \ref{re:relationintspace}, there does not exist any intersection space pair of $X$ with the
previous stratification.
\end{example}

A great number of examples can be constructed with this technique. For example, if one wishes to have simply connected link and strata one, can use instead
of Hopf fibration the fibration $\phi:S^7\to S^4$ with fibre $S^3$.

Now, we give an example of algebraic variety for which the intersection space does not exist for the middle perversity.

\begin{example}
\label{ex:algebraic}
Let $Fr(2,3)$ be the frame bundle over the Grassmannian $Gr(2,3)$, that is,
$Fr(2,3):=\{M \in \rm{Mat}(3 \times 2, \CC) | \rm{rk}(M)=2\}$ and the canonical bundle
$$\pi:Fr(2,3)\rightarrow Gr(2,3) \cong \mathbb P_\CC^2$$
is a $GL(2,\CC)$-pincipal bundle with the action
$$\xymatrix@R=1mm{\rm{GL} (2, \CC) \times Fr(2,3) \ar[r] & Fr(2,3) \\ (A, M) \ar[r] & A\cdot M}$$

Let $R_1^2:=\{M \in \rm{Mat}(2 \times 2, \CC) | \rm{rk}(M)\leq 1\}$ and let us consider the action
$$\xymatrix@R=1mm{\rm{GL} (2, \CC) \times R_1^2 \ar[r] & R_1^2 \\ (A, M) \ar[r] & A\cdot M}$$

Let $X:=Fr(2,3) \times_{\rm{GL} (2, \CC)} R_1^2$. Since $Sing(R_1^2)=\{0\}$, we have the equality
$$Sing(X)=Fr(2,3) \times_{\rm{GL} (2, \CC)} \{0\} \cong Gr(2,3) \cong \mathbb P_\CC^2$$
and the induced fiber bundle
$$\xymatrix@R=1mm{Fr(2,n) \times_{\rm{GL} (2, \CC)} R_1^2 \setminus \{0\} \ar[r] & \mathbb P_\CC^2 \\
(M_1, M_2) \ar[r] & \pi(M_1)}$$
is the fibration of links over the singularity. The fiber of this morphism is $R_1^2 \setminus \{0\}$.

Now, let us consider the action
$$\xymatrix@R=1mm{\rm{GL} (2, \CC) \times \CC^2 \ar[r] & \CC^2 \\
(A, (a,b)) \ar[r] & A\cdot  \left({\begin{array}{c} a \\ b \end{array}}\right)}$$
and let $Y:=Fr(2,3) \times_{\rm{GL} (2, \CC)} \CC^2$.

The morphism
$$\xymatrix@R=1mm{\CC^2  \ar[r]^f & R_1^2 \\
(a,b) \ar[r] & \left({\begin{array}{cc} a & a\\ b & b \end{array}}\right)}$$
is compatible with the actions. So, it induces a morphism $ g:Y \rightarrow X$.

Moreover, $g^{-1}(Sing(X))=Fr(2,3) \times_{\rm{GL} (2, \CC)} \{0\} \cong Gr(2,3) \cong \mathbb P_\CC^2$
and the fiber bundle
$$\xymatrix@R=1mm{Fr(2,3) \times_{\rm{GL} (2, \CC)} \CC^2 \setminus \{0\} \ar[r] & \mathbb P_\CC^2 \\
(M_1, M_2) \ar[r] & \pi(M_1)}$$
is the fibration of links. The fiber of this morphism is $\CC^2 \setminus \{0\}$.

In addition,
$$\xymatrix{Fr(2,3) \times_{\rm{GL} (2, \CC)} \CC^2 \setminus \{0\} \ar[rr]^g \ar[rd] & &
Fr(2,n) \times_{\rm{GL} (2, \CC)} R_1^2 \setminus \{0\}\ar[ld] \\ & \mathbb P_\CC^2 &}$$
is a morphism of fibrations which induces in the fiber the morphism
$f:\CC^2 \setminus \{0\} \rightarrow R_1^2 \setminus \{0\}$.

Let us denote $U_X:=Fr(2,n) \times_{\rm{GL} (2, \CC)} R_1^2 \setminus \{0\}$ and $U_Y:=Fr(2,n)
\times_{\rm{GL} (2, \CC)} \CC^2 \setminus \{0\}$. Moreover, let $j_X:\mathbb P_C^2\rightarrow X$,
$i_X:U_X \rightarrow X$, $j_Y:\mathbb P_C^2
\rightarrow Y$ and $i_Y:U_Y \rightarrow Y$ be the canonical inclusions.

The morphism between fibrations $g$ produces a morphism of complexes
$$j_{Y *} j_Y^* i_{Y *} \QQ_{U_Y}\xrightarrow{\gamma} j_{X *} j_X^* i_{X *} \QQ_{U_X}.$$

Moreover, $\CC^2 \setminus \{0\}$ is homotopically equivalent to $S^3$, $R_1^2 \setminus \{0\}$ is
homotopically equivalent to $S^3 \times S^2$ and $f:\CC^2 \setminus \{0\} \rightarrow R_1^2 \setminus
\{0\}$ induces an isomorphism between the $0$-th and the third cohomology groups. Then $\gamma$ induces
an isomorphism between the cohomology sheaves
$$\H^0(j_{Y *} j_Y^* i_{Y *} \QQ_{U_Y})\cong\H^0(j_{X *} j_X^* i_{X *} \QQ_{U_X})$$
and
$$\H^3(j_{Y *} j_Y^* i_{Y *} \QQ_{U_Y})\cong\H^3(j_{X *} j_X^* i_{X *} \QQ_{U_X}).$$

Let $E^{p,q}_r$ be the local to global spectral sequence of $j_{X *} j_X^* i_{X *} \QQ_{U_X}$,
let $E'^{p,q}_r$ be the local to global spectral sequence of hypercohomology of $j_{Y *} j_Y^* i_{Y *} \QQ_{U_Y} $, and
$\gamma^{p,q}_r: E'^{p,q}_r\rightarrow E^{p,q}_r$ the morphism induced by $\gamma$. Then,
$$\gamma^{p,q}_2:\HH^p(\mathbb P_C^2, \H^q(j_{Y *} j_Y^* i_{Y *} \QQ_{U_Y}))\rightarrow
\HH^p(\mathbb P_C^2, \H^q(j_{X *} j_X^* i_{X *} \QQ_{U_X}))$$
is an isomorphism if $q=0,3$.

$E'^{p,q}_2$ is
$$\begin{array}{ccc|c|c|c|c|c}
    & & 3 & \QQ & 0 & \QQ & 0 & \QQ \\ \cline{3-8}
    & & 2 & 0 & 0 & 0 & 0 & 0 \\ \cline{3-8}
    q \uparrow& & 1 & 0 & 0 & 0 & 0 & 0 \\ \cline{3-8}
    & & 0 & \QQ & 0 & \QQ & 0 & \QQ \\ \cline{3-8}
    & & & 0 & 1 & 2 & 3 & 4 \\
    \multicolumn{8}{c}{}\\
    & &\multicolumn{6}{c}{\overrightarrow{p}}
\end{array}$$

and it does not degenerate in $r=2$. So, $d'^{0,3}_4$ is different from $0$ since it is the unique
differential different from $0$ which can appear.

Moreover,$E^{p,q}_2$ is
$$\begin{array}{ccc|c|c|c|c|c}
    & & 6 & \QQ & 0 & \QQ & 0 & \QQ \\ \cline{3-8}
    & & 5 & 0 & 0 & 0 & 0 & 0 \\ \cline{3-8}
    & & 4 & 0 & 0 & 0 & 0 & 0 \\ \cline{3-8}
    q \uparrow & & 3 & \QQ & 0 & \QQ & 0 & \QQ \\ \cline{3-8}
    & & 2 & \QQ & 0 & \QQ & 0 & \QQ \\ \cline{3-8}
    & & 1 & 0 & 0 & 0 & 0 & 0 \\ \cline{3-8}
    & & 0 & \QQ & 0 & \QQ & 0 & \QQ \\ \cline{3-8}
    & & & 0 & 1 & 2 & 3 & 4 \\
    \multicolumn{8}{c}{}\\
    & &\multicolumn{6}{c}{\overrightarrow{p}}
\end{array}$$

We prove now $d^{0,3}_r:E^{0,3}_r \rightarrow E^{r,3-r+1}_r$ is different from $0$ for some $r\geq 0$.
If $d^{0,3}_2=0$, we have the following isomorphisms
$$E^{0,3}_4 \cong E^{0,3}_3 \cong E^{0,3}_2 \cong E'^{0,3}_2 \cong E'^{0,3}_4.$$
Moreover,
$$ E^{4,0}_4 \cong E^{4,0}_2 \cong E'^{4,0}_2 \cong E'^{4,0}_4$$
and the diagram
$$\xymatrix{E^{0,3}_4 \ar[d]^{d^{0,3}_4} \ar[r]^{\gamma^{0,3}_4} & E'^{0,3}_4 \ar[d]^{d'^{0,3}_4}\\
 E^{4,0}_4 \ar[r]^{\gamma^{4,0}_4} & E'^{4,0}_4 }$$
is cartesian.

Then, since $d'^{0,3}_4$ is not $0$, $d^{0,3}_4$ is also different from $0$.

If $\bar p$ is a perversity such that $\bar p(6)=2$ and $\bar q$ is the complementary perversity, then we have $\bar q(6)=2$. This happens for the
middle perversity.
Consequently, applying Corollary \ref{cor:cex}, there does not exist an intersection space complex of $X=Fr(2,3)
\times_{\rm{GL} (2, \CC)} R_1^2 $ with perversity $\bar p$ and stratification $X\supset Sing(X) \cong \mathbb
P_C^2$.

Hence, applying Remark \ref{re:relationintspace}, there does not exist any intersection space pair of $X$ with
perversity $\bar p$.

\end{example}

\section{Duality}
\label{sec:dual}

In this section, we establish the duality properties of the intersection space complexes. First, we study the
Verdier dual of the intersection space complex. Next, we give a version of Poincare duality for these complexes.

Let $X$ be a topological pseudomanifold with the following stratification:

$$
X=X_d\supset X_{d-2} \supset ... \supset X_0 \supset X_{-1}=\emptyset
$$

Let $U_k:=X\setminus X_{d-k}$ and let $i_k: U_k \rightarrow U_{k+1}$ and $j_k: X_{d-k} \setminus X_{d-k-1}
\rightarrow U_{k+1}$ be the natural inclusions.

\subsection{Verdier duality}

\begin{theorem}\label{th:verdual}
Let $IS_{\bar p}$ be an intersection space complex of $X$ with perversity $\bar p$ and let $\bar q$ be the
complementary perversity of $\bar p$. Then, $\mathcal{D}IS_{\bar p} [-d]$, where $\mathcal{D}$ denotes the Verdier
 dual, is an intersection space complex of $X$ with perversity $\bar q$.
\end{theorem}

\begin{proof}
We have to prove $\mathcal{D}IS_{\bar p} [-d]$ verifies $[AXS1]_k$ for perversity $\bar q $ for $k=2,...,d$.

  \begin{itemize}
    \item[($a$)] We have a chain of isomorphisms
    $$(\mathcal{D}IS_{\bar p} [-d])_{|U_2} \cong \mathcal{D} \QQ_{U_2} [-d] \cong \QQ_{U_2}.$$

    \smallskip

    \item[($b$)] Let $x\in X$, the group $\mathcal{H}^i(j_x^*\mathcal{D}IS_{\bar p}[-d])$ is isomorphic to $\mathcal{H}^{d-i}
    (j_x^! IS_{\bar p})^v$, which is $0$ if $i \notin \{0,1,...,d\}$.

    \smallskip

    \item[($c_k$)] Let $x\in X_{d-k} \setminus X_{d-k-1}$, the group $\mathcal{H}^i(j_x^* \mathcal{D}IS_{\bar p} [-d])$ is
    isomorphic to $\mathcal{H}^{d-i}(j_x^! IS_{\bar p})^v$, which (by property ($d1'_k$)) is $0$ if $d-i > d-
    \bar p(k)-1$, that is, if $i\leq \bar p(k)$.

    \smallskip

    \item[($d1'_k$)] Let $x\in X_{d-k} \setminus X_{d-k-1}$, the group $\mathcal{H}^i(j_x^! \mathcal{D}IS_{\bar p} [-d])$ is
    isomorphic to $\mathcal{H}^{d-i}(j_x^* IS_{\bar p})^v$, which (by property ($c_k$)) is $0$ if $d-i \leq
    \bar q(k)$, that is, if $i > d-\bar q(k)-1$.

    \smallskip

    \item [($d2'_k$)] Let $x\in X_{d-k} \setminus X_{d-k-1}$, the group $\mathcal{H}^{d- \bar q(k)-1}(j_x^!\mathcal{D}
    IS_{\bar p} [-d])$ is isomorphic to $\mathcal{H}^{\bar q(k)+1}(j_x^* IS_{\bar p})^v$, the group $\mathcal{H}^{\bar p(k)+1}
    (j_x^*\mathcal{D} IS_{\bar p} [-d])$ is isomorphic to $\mathcal{H}^{d-\bar p(k)-1}(j_x^! IS_{\bar p})^v$ and the
    canonical morphism $$\mathcal{H}^{d- \bar q(k)-1}(j_x^!\mathcal{D}IS_{\bar p} [-d]) \rightarrow
    \mathcal{H}^{\bar p(k)+1} (j_x^*\mathcal{D}IS_{\bar p} [-d])$$ is the dual morphism of $\mathcal{H}^{d-
    \bar p(k)-1}(j_x^! IS_{\bar p}) \rightarrow \mathcal{H}^{\bar q(k)+1}(j_x^* IS_{\bar p})$, which is the morphism
    $0$ (by property ($d2'_k$)).
  \end{itemize}

\end{proof}

In Corollary~\ref{cor:obstructions} the space of obstructions for existence and uniqueness of intersection spaces are described. Verdier duality
$\mathcal{D}$ interchanges intersection space complexes with complementary perversities. We deduce

\begin{corollary}
 \label{cor:dualityexchange}
 Let $X$ be a topological pseudomanifold as above. Let $\bar p$ and $\bar q$ complementary perversities. An intersection space complex
 for perversity $\bar p$ exists if and only if an intersection space complex for perversity $\bar q$ exists. Verdier duality induces a bijection between
 the set of intersection space
 complexes for perversity $\bar p$ and the set of intersection space complexes for perversity $\bar q$.
\end{corollary}

\subsection{Poincare duality in the case of 2 strata}

Now, suppose $X$ has a unique non-trivial stratum. So, the stratification of $X$ is
$$X\supset X_{d-k} \supset \emptyset$$
where $k$ is the codimension of $X_{d-k}$.

According with Corollary~\ref{cor:obstructions}, the obstruction for existence of intersection space for perversity $\bar p$ lives in
$Ext^1(\tau_{> \bar q(k)} j_{k*} j_k^* i_{k*} \QQ|_{U_k}, \tau_{\leq \bar q(k)} j_{k*} j_k^* i_{k*} \QQ|_{U_k})$. Assume that the obstruction vanishes
so that intersection space complexes exist. In this case the space of intersection space complexes for perversity $\bar p$
is parametrized by the vector space
$$E_{\bar p}:=Hom(\tau_{> \bar q(k)} j_{k*} j_k^* i_{k*} \QQ|_{U_k}, \tau_{\leq \bar q(k)} j_{k*} j_k^* i_{k*} \QQ|_{U_k}).$$

Corollary~\ref{cor:dualityexchange} implies that the obstruction for existence of intersection space for perversity $\bar q$ vanishes and that the space
$$E_{\bar q}:=Hom(\tau_{> \bar p(k)} j_{k*} j_k^* i_{k*} \QQ|_{U_k}, \tau_{\leq \bar p(k)} j_{k*} j_k^* i_{k*} \QQ|_{U_k})$$
of intersection space complexes for perversity $\bar q$ is isomorphic to $E_{\bar p}$.

\begin{proposition}
\label{prop:genericbetti}
Let $E_{\bar p}$ be the space of all intersection space complexes of $X$ with perversity $\bar p$ up to isomorphisms.

The dimensions of the vector spaces $\HH^i(X, IS_{\bar p})$ with $IS_{\bar p} \in
E_{\bar p}$ have a minimum and the subset
$$\{IS_{\bar p} \in E_{\bar p} |\dim (\HH^i(X, IS_{\bar p})) \, \text{is minimum} \}\subset E_{\bar p}$$
is open for every $i$.
\end{proposition}
\begin{proof}
For every intersection space complex $IS_{\bar p} \in E_{\bar p}$, we have a triangle

$$IS_{\bar p} \rightarrow i_{k *} \QQ_{U_{2}} \rightarrow \tau_{\leq \bar q(k)} j_{k *} j_k^* i_{k *} \QQ_{U_{2}}$$

This triangle induce the long exact sequence of hypercohomology

$$... \rightarrow \HH^i(X, IS_{\bar p})\rightarrow \HH^i(X, i_{k *}\QQ_{U_{2}}) \xrightarrow{\alpha^i(IS_{\bar p})}\HH^i
(X, \tau_{\leq \bar q(k)}j_{k *} j_k^* \QQ_{U_{2}}) \rightarrow ...$$

So, for every $i\in \ZZ$, there is an isomorphism
$$\HH^i(X, IS_{\bar p}) \cong \ker (\alpha^i(IS_{\bar p})) \oplus \coker (\alpha^{i-1}(IS_{\bar p})).$$

Moreover,
$$\dim (\coker (\alpha^{i}(IS_{\bar p})))= \dim (\HH^i(X, \tau_{\leq \bar q(k)}j_{k *} j_k^* \QQ_{U_{2}}))- \dim (\HH^i(X,
i_{k *}\QQ_{U_{2}})) +\dim(\ker(\alpha^{i}(IS_{\bar p})))$$

So, $\dim (\HH^i(X, IS_{\bar p}))$ is minimum if and only if $\dim(\ker(\alpha^{i}(IS_{\bar p})))$ and $\dim(\ker(\alpha^{i-1}(IS_{\bar p})))$
are minimum.

The morphism $\alpha^i(IS_{\bar p})$ is the morphisms induced in hypercohomology by the composition

$$i_{k *} \QQ_{U_2} \xrightarrow{a} j_{k *} j_k^* i_{k *} \QQ_{U_2} \xrightarrow{\lambda(IS_{\bar p})} \tau_{\leq \bar q(k)}
j_{k *} j_k^* i_{k *} \QQ_{U_2}$$

where $a$ is the canonical morphisms and $\lambda(IS_{\bar p})$ is a retraction of the natural truncation morphism
$f:\tau_{\leq \bar q(k)} j_{k *} j_k^* i_{k *} \QQ_{U_2} \rightarrow j_{k *} j_k^* i_{k *} \QQ_{U_2}$.

Let us denote by
$$a^i:\HH^i(X, i_{k *} \QQ_{U_2}) \rightarrow \HH^i(X, j_{k *} j_k^* i_{k *} \QQ_{U_2}),$$
$$\lambda^i (IS_{\bar p}):\HH^i(X, j_{k *} j_k^* i_{k *} \QQ_{U_2}) \rightarrow \HH^i(X, \tau_{\leq \bar q(k)}j_{k *} j_k^* i_{k *}\QQ_{U_2}),$$
$$f^i:\HH^i(X, \tau_{\leq \bar q(k)}j_{k *} j_k^* i_{k *}\QQ_{U_2}) \rightarrow \HH^i(X,j_{k *}j_k^*i_{k *}\QQ_{U_2})$$
the morphisms induced in hypercohomology. Then, we have
$$\dim(\ker (\alpha^i(IS_{\bar p})))= \dim ((a^i)^{-1}(\ker(\lambda^i(IS_{\bar p})))) = \dim(\ker(a^i)) + \dim( \im(a^i)
\cap \ker(\lambda^i(IS_{\bar p})))$$

Hence, $\dim(\ker(\alpha^{i}(IS_{\bar p})))$ is minimum if and only if $\dim(\im(a^i) \cap \ker(\lambda^i(IS_{\bar p})))$
is minimum.

Since $\lambda^i(IS_{\bar p}) \circ f^i$ is the identity, the homomorphism $\lambda^i(IS_{\bar p})$ is surjective and we have the equality
$$\dim (\ker(\lambda^i(IS_{\bar p})))= \dim (\HH^i(X, j_{k *} j_k^* i_{k *}\QQ_{U_2}))- \dim(\HH^i(X, \tau_{\leq \bar q(k)}
j_{k *} j_k^* i_{k *}\QQ_{U_2})).$$
So, $\dim (\ker(\lambda^i(IS_{\bar p})))$ is independent of $IS_{\bar p}$. Then, for every $IS_{\bar p}\in E_{\bar p}$, there is an isomorphism
$$\HH^i(X, j_{k *} j_k^* i_{k *}\QQ_{U_2}))/\ker(\lambda^i(IS_{\bar p}))\cong  \QQ^{d^i}$$
where $d^i:=\dim(\HH^i(X, \tau_{\leq \bar q(k)}j_{k *} j_k^* i_{k *}\QQ_{U_2})).$

Now, consider the composition of morphisms
$$\xymatrix@C=1.2cm{\HH^i(X, i_{k *} \QQ_{U_2}) \ar[r]^{a^i\hspace{4mm}} \ar@/^{8mm}/[rr]^{\phi(IS_{\bar p})} & \HH^i(X,
j_{k *} j_k^* i_{k *}\QQ_{U_2})\ar[r]^{\pi(IS_{\bar p})\hspace{1.9cm}} & \HH^i(X, j_{k *} j_k^* i_{k *}\QQ_{U_2}))/\ker
(\lambda^i(IS_{\bar p}))\cong  \QQ^{d^i}}$$
where $\pi(IS_{\bar p})$ is the canonical projection.

Then, $\im(a^i)\cap \ker(\lambda^i(IS_{\bar p}))$ gets the minimum dimension when the morphism $\phi(IS_{\bar p})$ gets the
maximum rank, which happens in an open subset.
\end{proof}

\begin{definition}
The \emph{general} $i$-\emph{th Betti number} of the intersection space complexes of $X$ with perversity $\bar p$ is the
minimum of the dimensions of the vector spaces $\HH^i(X, IS_{\bar p})$ with $IS_{\bar p} \in E_{\bar p}$.
\end{definition}

\begin{definition}
A \emph{general intersection space complex} of $X$ with perversity $\bar p$ is an intersection space complex $IS_{\bar p}
\in E_{\bar p}$ such that $ \dim (\HH^i(X, IS_{\bar p}))$ is the general $i$-th Betti number for $i=0,1,...,d$.
\end{definition}

\begin{theorem}
\label{th:genericduality}
Let $\bar p$ be a perversity and let $\bar q$ be its complementary perversity. If $IS_{\bar p}$ is a general intersection space complex of $X$ with perversity $\bar p$ and $IS_{\bar q}$ is a general intersection space complex of $X$ with perversity
$\bar q$, then, for $i=0,1,...,d$, there is an isomorphism of $\QQ$-vector spaces
$$\HH^i (X, IS_{\bar p}) \cong \HH^{d-i} (X, IS_{\bar q})^v$$
\end{theorem}
\begin{proof}
Given any intersection space complex $IS_{\bar p}$ of $X$ with perversity $\bar p$, we have
$$\HH^i (X, \mathcal{D}IS_{\bar p} [-d])^v \cong \HH^{i-d} (X, \mathcal{D}IS_{\bar p})^v \cong \HH^{d-i}(X, IS_{\bar p}).$$
Applying Theorem~\ref{th:verdual}, the complex $\mathcal{D}IS_{\bar p} [-d]$ is an intersection space complex of $X$ with perversity $\bar q$.
We denote $IS_{\bar q}:=\mathcal{D}IS_{\bar p} [-d]$.

Suppose $IS_{\bar q}$ is not a general intersection space complex of $X$. Then, there exist another intersection space complex of
$X$ with perversity $\bar q$, $IS'_{\bar q}$, such that we have the strict inequality
$$\sum_{i=0}^d \dim (\HH^i (X, IS'_{\bar q})) < \sum_{i=0}^d \dim (\HH^i (X, IS_{\bar q})).$$

Consequently we have,
$$\sum_{i=0}^d \dim (\HH^i (X, \mathcal{D}IS'_{\bar q} [-d]))=\sum_{i=0}^d \dim (\HH^i (X, IS'_{\bar q})^v) <
\sum_{i=0}^d \dim (\HH^i (X, IS_{\bar q})^v)=$$ $$=\sum_{i=0}^d \dim (\HH^i (X, \mathcal{D}IS_{\bar q}[-d]))=
\sum_{i=0}^d \dim (\HH^{i}(X, IS_{\bar p})).$$

So, $IS_{\bar p}$ is not a general intersection space complex of $X$.

We deduce that if $IS_{\bar p}$ is a general intersection space complex, then $IS_{\bar q}$ is also a general intersection space complex. So, there are isomorphisms
$$\HH^i (X, IS_{\bar p}) \cong \HH^{d-i} (X, IS_{\bar q})^v$$
for some general intersection space complexes $IS_{\bar p}$ and $IS_{\bar q}$.

Since the hypercohomology groups of general intersection space complexes with the same perversity are isomorphic, we conclude.
\end{proof}


\begin{thebibliography}{}

\bibitem{Ban07}M. Banagl. \emph{Topological invariants of stratified spaces.} Springer, Berlin, 2007.

\bibitem{Ban1} M. Banagl. {\em Singular Spaces and Generalized Poincar\'e Complexes}, Electron. Res. Announc. Math. Sci. \textbf{16} (2009), 63-73.

\bibitem{Ban10} M. Banagl. \emph{Intersection spaces, spatial homology truncation and string theory} Lecture notes in mathematics \textbf{1997}
Springer-Verlag, Berlin, Heilderberg, 2010.

\bibitem{Ban2} M. Banagl. {\em First Cases of Intersection Spaces in Stratification Depth 2}, J. of Singularities \textbf{5} (2012), 57 - 84.

\bibitem{BBM} M. Banagl. N. Budur, L. Maxim. {\em Intersection Spaces, Perverse Sheaves and Type IIB String Theory (pdf)}, Adv. Theor. Math. Phys. \textbf{18} (2014), no. 2, pp. 363 - 399.

\bibitem{BaCh} M. Banagl, B. Chriestenson. {\em Intersection Spaces, Equivariant Moore Approximation and the Signature}, J. of Singularities \textbf{16} (2017), pp. 141 - 179.

\bibitem{BaHu} M. Banagl, E. Hunsicker. {\em Hodge Theory for Intersection Space Cohomology}, arXiv:1502.03960 (2015).

\bibitem{BM1} M. Banagl. L. Maxim. {\em Intersection Spaces and Hypersurface Singularities}, J. of Singularities \textbf{5} (2012), 48 - 56.

\bibitem{BM2} M. Banagl. L. Maxim. {\em Deformation of Singularities and the Homology of Intersection Spaces}, J. Topol. Anal. \textbf{4} (2012), no. 4, 413 - 448.

\bibitem{CaMi} M. A. de Cataldo, L. Migliorini. \emph{The Decomposition Theorem, perverse sheaves and the topology of algebraic maps}
Bull. Amer. Math. Soc. \textbf{46}, no 4, October 2009, 535–633.

\bibitem{Ge} C. Geske. {\em Algebraic intersection spaces}, arXiv:1802.03871v1 (2018).

\bibitem{GWPL} C. G. Gibson, K. Wirthmüller, A. A. du Plessis, E. J. N. Looijenga. \emph{Topological stability of smooth mappings},
Lecture Notes in Math., Vol. 552, Springer-Verlag, Berlin, New York, 1976.

\bibitem{GorMP} M. Goresky, R. MacPherson. \emph{Intersection homology theory II.} Invent. Math. 72 (1983), no.1, 77--129.

\bibitem{Kl} M. Klimczak. {\em Poincare duality for spaces with isolated singularities}. arXiv:150707407v3 (2016).

\bibitem{Mas} W. S. Massey. \emph{A basic course in algebraic topology} Graduate text in mathematics \textbf{127} Springer-Verlag, New York, 1991.

\bibitem{Max} L. Maxim. \emph{Intersection spaces, perverse sheaves and string theory.} Journal of Singularities \textbf{15} (2016), 118-125

\bibitem{Ram} S. Ramanan. \emph{Global calculus} Graduate Studies in Mathematics, 65. American Mathematical Society, Porvidence, Rhode Island, 2004.

\end{thebibliography}
\end{document}